\documentclass[10pt,a4paper]{article}

\usepackage{fullpage}
\usepackage{enumitem}
\usepackage{indentfirst}
\usepackage[utf8]{inputenc}

\usepackage{color}
\usepackage{bbm}					% for black board bold numbers
\usepackage{relsize}					% command for size of math symbols

\usepackage{rotating}

\usepackage{hhline}

\usepackage[all]{xy}
\usepackage{tikz-cd}
\usetikzlibrary{decorations.pathmorphing}

\usepackage{amsmath}
\usepackage{amsthm}
\usepackage{amssymb}
\usepackage{amscd}
\usepackage{MnSymbol}
\usepackage[mathcal]{euscript}			% replace mathcal by prettier mathscr

\usepackage{multirow}
\usepackage{tabularx}
\usepackage{hyperref}
\usepackage{xcolor}
\hypersetup{
    colorlinks=true,
    linkcolor=[RGB]{13 71 161},		% Couleur des liens internes
    citecolor=[RGB]{13 71 161},		% Couleur des numéros de la biblio dans le corps
    urlcolor=[RGB]{13 71 161}		% Couleur des url
}
\usepackage[capitalize,nameinlink]{cleveref}
\crefname{thm}{Theorem}{Theorems}
\Crefname{thm}{Theorem}{Theorems}
\crefname{prop}{Proposition}{Propositions}
\Crefname{prop}{Proposition}{Propositions}
\crefname{rem}{Remark}{Remarks}
\Crefname{rem}{Remark}{Remarks}
\usepackage[shortcuts]{extdash}	% tirets insécables

\newcommand{\ie}{{i.e. }}
\newcommand{\eg}{{e.g. }}

\newcommand{\oo}{$\infty$\=/} 	% avec tiret insécable
\newcommand{\bbone}{{\mathbbm 1 }}

\newcommand{\Kappa}{\mathrm{K}}

\renewcommand{\AA}{\mathbb{A}}

\newcommand{\CC}{\mathbb{C}}

\newcommand{\JJ}{\mathbb{J}}

\newcommand{\NN}{\mathbb{N}}

\newcommand{\RR}{\mathbb{R}}

\newcommand{\UU}{\mathbb{U}}

\newcommand{\WW}{\mathbb{W}}

\newcommand{\cA}{\mathcal{A}}
\newcommand{\cB}{\mathcal{B}}
\newcommand{\cC}{\mathcal{C}}
\newcommand{\cD}{\mathcal{D}}
\newcommand{\cE}{\mathcal{E}}
\newcommand{\cF}{\mathcal{F}}

\newcommand{\cJ}{\mathcal{J}}

\newcommand{\cL}{\mathcal{L}}
\newcommand{\cM}{\mathcal{M}}

\newcommand{\cP}{\mathcal{P}}

\newcommand{\cR}{\mathcal{R}}
\newcommand{\cS}{\mathcal{S}}

\newcommand{\cW}{\mathcal{W}}

\newcommand{\eqnand}{\qquad \text{and} \qquad}            % and between two equations
\newcommand{\truncated}[1]{^{\leq #1}}            % truncation index
\newcommand{\ot}{\leftarrow}                    % short left arrow

\newcommand{\tto}{{\begin{tikzcd}[ampersand replacement=\&]{}\ar[r]\&{}\end{tikzcd}}}

\newcommand{\ntto}[1]{{\begin{tikzcd}[ampersand replacement=\&]{}\ar[r, "{#1}"]\&{}\end{tikzcd}}}
\newcommand{\mto}{{\begin{tikzcd}[ampersand replacement=\&]{}\ar[r,mapsto]\&{}\end{tikzcd}}}

\newcommand{\stto}{{\begin{tikzcd}[ampersand replacement=\&, sep=small]{}\ar[r]\&{}\end{tikzcd}}}
\newcommand{\nstto}[1]{{\begin{tikzcd}[ampersand replacement=\&, sep=small]{}\ar[r, "{#1}"]\&{}\end{tikzcd}}}

\newcommand{\subto}{\hookrightarrow}            % subobject
\newcommand{\mono}{\rightarrowtail}                % monomorphisms
\newcommand{\onto}{\twoheadrightarrow}            % surjections
\newcommand{\Ra}{\Rightarrow}                    % double right arrows
\newcommand{\RA}{\Longrightarrow}                % long double right arrows
\newcommand{\La}{\Leftarrow}                    % double left arrows
\renewcommand{\iff}{\Leftrightarrow}            % short logical equivalence
\newcommand{\IFF}{\Longleftrightarrow}            % long logical equivalence

\newcommand{\xto}[1]{\xrightarrow {#1}}            % variable size right arrow
\newcommand{\xot}[1]{\xleftarrow {#1}}            % variable size left arrow

\newcommand{\pbmark}{\ar[dr, phantom, "\ulcorner" very near start, shift right=1ex]}
\newcommand{\pomark}{\ar[ul, phantom, "\lrcorner" very near start, shift right=1ex]}

\newcommand{\map}[2]{\left[#1,#2\right]}                               % external hom between #1 and #2
\newcommand{\relmap}[3]{\left[#2, #3\right]_{#1}}                % external hom  between #2 and #3 relative to #1

\newcommand{\intmap}[2]{\left\lsem #1, #2 \right\rsem}            % internal hom between #1 and #2
\newcommand{\bracemap}[2]{\left\{#1,#2\right\}}                 % external hom between #1 and #2

\newcommand{\intbracemap}[2]{\left\{\!\!\left\{#1,#2\right\}\!\!\right\}}              % internal hom between #1 and #2

\newcommand{\s}{\mathrm{s}}                                % source functor
\renewcommand{\t}{\mathrm{t}}                            % target functor

\DeclareMathOperator*{\colim}{colim}                     % colimit operator
\newcommand{\colimit}[1]{\underset{#1}{\mathrm{colim}}\ }  % colimit operator with better handling of index category

\newcommand{\fun}[2]{\left[#1,#2\right]}                % functor category
\newcommand{\Arr}[1]{{#1}^\rightarrow}                    % arrow category
\newcommand{\iso}{^\simeq}                                % category of isomorphism
\newcommand{\op}{^{op}}                                    % opposite category
\newcommand{\cc}{_\mathrm{cc}}                            % cocomplete / cocontinuous / cocompletion
\newcommand{\lex}{^\mathrm{lex}}                        % finitely complete / finite continuity / finite completion
\newcommand{\pointed}{^\bullet}                            % exponent for a pointed object
\newcommand{\DKLoc}[2]{{L\!\left(#1,#2\right)}}                    % Dwyer-Kan localization of categories
\newcommand{\relDKLoc}[4]{{L_{#1}^{#2}\!\left(#3,#4\right)}}    % Dwyer-Kan localization of #1-cocomplete et #2-complete categories

\newcommand{\commaindex}{\downarrow}
\newcommand{\comma}{\!\downarrow\!}

\newcommand{\ini}{0}                                        % notation for the initial object of the current category

\newcommand{\Set}{\mathrm{\mathcal{S} et}}                    % category of small sets
\newcommand{\Fin}{\mathrm{ \mathcal{F} in}}                        % category of finite sets / spaces
\newcommand{\Sp}{\mathrm{\mathcal{S}p}}                            % category of spectra
\newcommand{\Cat}{\mathrm{\mathcal{C} at}}                         % category of small categories
\newcommand{\pp}{\,\square\,}                                        % pushout-product with extra spaces
\newcommand{\pbh}[2]{\left\langle #1, #2 \right\rangle}                % pullback-hom
\newcommand{\intpbh}[2]{\left\llangle #1, #2 \right\rrangle}            % internal pullback-hom

\renewcommand{\P}[1]{{\mathcal{P}\!\left(#1\right)}}                        % cocompletion of / small presheaves on #1
\newcommand{\relP}[2]{{\mathcal{P}_{#1}\!\left(#2\right)}}                    % cocompletion of #2 preserving #1-colimits

\newcommand{\Ind}[1]{{\mathrm{\mathcal{I} nd}\!\left(#1\right)}}            % ind-completion of #1
\newcommand{\relInd}[3]{{\mathrm{\mathcal{I} nd}_{#1}^{#2}\!\left(#3\right)}}    % #1-ind-completion of #2
\renewcommand{\S}[1]{{\mathcal{S}\!\left[#1\right]}}                    % Free logos on #1
\newcommand{\join}{\star}                                    % join product

\newcommand{\intperp}{\text{\begin{sideways}$\!\Vdash$\end{sideways}}\,}    % Internal orthogonality operator 
\newcommand{\iperp}{\ \intperp\ }                                    % Internal orthogonality with space on sides
\newcommand{\fwperp}{\upModels}                                        % Fiberwise orthogonality
\newcommand{\lexperp}{\underset{\text{lex}}{\perp}}                                        % lex orthogonality

\definecolor{mat}{rgb}{0.77,0.85,1}

\newcommand{\Quillenmap}[1]{q\left(#1\right)}
\newcommand{\Quillenmapn}[2]{q^{#1}\left(#2\right)}
\newcommand{\Quillen}[1]{T_Q\!\left(#1\right)}
\newcommand{\Quillenn}[2]{T_Q^{#1}\!\left(#2\right)}

\newcommand{\Kellymap}[1]{k\left(#1\right)}
\newcommand{\Kellymapn}[2]{k^{#1}\left(#2\right)}
\newcommand{\Kellyconstruction}{T_K}
\newcommand{\Kelly}[1]{T_K\!\left(#1\right)}
\newcommand{\Kellyvar}[1]{T_K{#1}}
\newcommand{\Kellyn}[2]{T_K^{#1}\!\left(#2\right)}

\newcommand{\Plusoriginal}[1]{{#1}^+}
\newcommand{\Plusconstruction}{{T_+}}
\newcommand{\Plus}[1]{T_{+}\!\left(#1\right)}
\newcommand{\Plusvar}[1]{T_{+}{#1}} % sans parenthèses
\newcommand{\Plusmap}[1]{{#1}^+}
\newcommand{\Minusmap}[1]{{#1}^-}
\newcommand{\Plusplus}[1]{T_{++}\!\left({#1}\right)}
\newcommand{\Plusplusvar}[1]{T_{++}{#1}}
\newcommand{\Plusplusmap}[1]{{#1}^{++}}
\newtheorem{thm}{Theorem}[subsection]
\newtheorem{prop}[thm]{Proposition}
\newtheorem{cor}[thm]{Corollary}
\newtheorem{lem}[thm]{Lemma}
\newtheorem{lemma}[thm]{Lemma}

\newtheorem{thmintro}{}
\newenvironment{thm-intro}[1]
  {\renewcommand\thethmintro{#1}\thmintro}
  {\endthmintro}

\theoremstyle{definition}
\newtheorem{defi}[thm]{Definition}

\newtheorem{rem}[thm]{Remark}

\begin{document}

\title{
Small object arguments, \\
plus-construction, \\
and left-exact localizations
}
\author{
M.~Anel%
\footnote{Department of Philosophy, Carnegie Mellon University}
\and 
C.~Leena~Subramaniam%
\footnote{IRIF, Université Paris 7 -- Denis Diderot}
}
\maketitle

\begin{abstract}
We present a variant of the small object argument, inspired by Kelly, better suited to construct unique factorisation systems.
Our main result is to compare it to the plus-construction involved in sheafification.
We apply this to construct localizations, modalities and left-exact localizations explicitly from generators.
%All our constructions are valid in 1-categories and $(\infty,1)$\=/categories
\end{abstract}

\setcounter{tocdepth}{2}
\tableofcontents

\newpage

\section*{Introduction}
\addcontentsline{toc}{section}{Introduction}

The small object argument is a technique of construction of the ``completion" of some object by means of the transfinite iteration of a ``partial completion" operator.
The idea can be traced back at least to the construction of injective modules, where it improves on previous completion techniques based on Zorn's lemma~\cite{Cartan-Eilenberg, Grothendieck:Tohoku}. 
The construction was given a proper setting--and a name--when Quillen presented it as the construction of a weak factorization system~\cite{Quillen:HA}. 
It was then adapted to strong factorization systems by Gabriel and Ulmer~\cite{Gabriel-Ulmer} and Kelly~\cite{Kelly} with the motivation of constructing reflective subcategories.

The purpose of this work is to revisit the construction of Kelly in terms of the pushout~product/pullback~hom closed monoidal structure on arrow categories and in the context of \oo categories.
Let $W\to \Arr \cC$ be a diagram of arrows in some cocomplete category $\cC$. 
From $W$ and a map $f:A\to B$, we can construct
the coend $\int^w \pbh wf \pp w$, where $\pbh w f$ is the pullback-hom of $w$ and $f$,
and $\pbh w f\pp w$ is the pushout-product of $\pbh w f$ and $w$.
This map comes with a natural counit
\[
\int^{w\in W} \pbh wf \pp w\ntto \epsilon f
\]
corresponding to a commutative square in $\cC$.
Let 
\[
A\ntto {u(f)} \Kelly f \ntto {\Kellymap f} B
\] 
be the factorization of $f$ via the cocartesian gap map $\Kellymap f$ of the square $\epsilon$.
Our first result is the following version of \cite[Theorem 11.5]{Kelly}.

\begin{thm-intro}{\cref{thm:!SOA2}}[Construction of factorization systems]
Let $\cC$ be a cocomplete \oo category and $W\to \Arr \cC$ be a small diagram of arrows with small  domains and codomains.
Then $\Kellymap f$ is Kelly's operator and the transfinite iteration of $k$ converges to the factorization associated to the strong orthogonal system $({^\perp}(W^\perp),W^\perp)$.
\end{thm-intro}
\noindent We apply this to the construction of a number of factorization systems in \ref{sec:factorizations} (image factorization, Postnikov towers, nullification, long localization and conservative functors, cofinal functors and fibrations).
We also apply this to recover Kelly's result on reflective subcategories in \cref{thm:orthogonal-reflection}.

\medskip

Our main result, though, is the comparison of Kelly's construction $k$ with the plus-construction involved in sheafification~\cite{SGA41,MLM}.
We shall show in \ref{sec:+construction} that the plus-construction $\Plusmap f$ of a map $f:A\to B$ is involved in a factorization of $f$
\[
A \ntto {\Minusmap f}\Plus f \ntto {\Plusmap f} B,
\]
where $\Plus f = \colimit {W\comma f} \text{codomain}(w)$.

\begin{thm-intro}{\cref{thm:+construction}}[plus-construction]
If $\cC$ is a presentable category and $W\to \Arr \cC$ is a pre-modulator (\cref{defi:pre-modulator}), then there exists a natural isomorphism $\Plusmap f = \Kellymap f$.
In particular, the factorization of $f$ generated by such a $W\to \Arr \cC$ can be produced by a transfinite iteration of the plus-construction $f\mapsto \Plusmap f$.
\end{thm-intro}

This theorem is useful because it simplifies the colimit formula for $\Plus f$ is usually more workable than the coend formula for $\Kelly f$.
We use it to derive a recipe to construct {\em modalities}, that is, factorization systems stable under base change.
\begin{thm-intro}{\cref{thm:+modality}}[Stable plus-construction]
Let $\cC$ is a presentable locally cartesian closed category.
If $W\to \Arr \cC$ is a modulator (\cref{defi:modulator}), then the factorisation system generated by $W$ is a modality.
\end{thm-intro}

But our main applications are to the construction of {\em left-exact localizations} for \oo topoi.
Such localizations are in one-to-one correspondence with {\em left-exact modalities}, that is, factorization systems that are stable by finite limits in the arrow category (and not only by base change) (see \cref{prop:stable/lex-FS,prop:lex-loc=lex-mod}).
We introduce the notion of {\em lex modulator} (\cref{defi:lex-modulator}), which enlarges the notion of a Grothendieck topology. 
We prove in \cref{prop:lex-mod-4-loc-lex} that any accessible left-exact localization can be presented by means of a lex modulator.

\begin{thm-intro}{\cref{thm:lex+construction}}[Left-exact plus-construction]
Let $\cE$ be a $n$\=/topos  ($n\leq \infty$) and $W\to \Arr \cE$ a lex modulator (\cref{defi:lex-modulator}),
then the factorization system generated by $W$ is a lex modality (\ie a lex localization).
\end{thm-intro}

We also characterize the left-exact modality generated by an arbitrary diagram of maps $W\to \Arr \cE$.
Let $W^\Delta$ be the collection of all the diagonals $\Delta^n w$ of all maps $w$ in $W$,
and $W^{\Delta mod}$ be the {\em modulator envelope} (\cref{defi:base-change-envelope}) of $W^\Delta$ with respect to some generators
(essentially, $W^{\Delta mod}$ is the stabilisation by base change of $W^\Delta$).

\begin{thm-intro}
{\cref{thm:lex-loc-gen-by-map}}[Generation of left-exact localizations]
Let $W\to \Arr \cE$ be a diagram of maps in an $n$\=/topos $\cE$ ($n\leq \infty$).
\begin{enumerate}[label=(\alph*), itemsep=0em, leftmargin=*]
\item The left-exact modality generated by $W$ is the factorization system generated by the modulator $W^{\Delta mod}$.
\item The left-exact localization $P_W$ of $\cE$ generated by $W$ is the iteration of the plus-construction associated to the modulator $W^{\Delta mod}$.
\item An object $X$ in $\cE$ is local for the left-exact localization generated by $W$ if and only if $X$ is orthogonal to all the diagonals of $W$ and all their base change.
\end{enumerate}\end{thm-intro}

\paragraph{Acknowledgments }
The authors are thankful to
Steve Awodey,
Georg Biedermann,
Eric Finster,
Jonas Frey,
André Joyal,
Paul-André Melliès,
Egbert Rijke,
Andrew Swan,
and Joseph Tapia
for many helpful conversations.
The first author gratefully acknowledges the support of the Air Force Office of Scientific Research through MURI grant FA9550-15-1-0053. 
%Any opinions, findings and conclusions or recommendations expressed in this material are those of the authors and do not necessarily reflect the views of the AFOSR.

%\newpage

\section{Arrow categories}

\paragraph{Vocabulary conventions}
Our basic framework is the theory of \oo categories~\cite{Joyal:QC,Lurie:HTT}.
To avoid the ``\oo'' prefix everywhere, we shall simply say {\em category} instead of {\em \oo category} and \emph{groupoid} instead of \emph{\oo groupoid}. 
In particular, we shall say simply {\em topos} for the notion of \oo topos of~\cite{Lurie:HTT}.
When classical categories or topoi are needed, we shall refer to them as 1-categories or 1-topoi. 

We shall denote by $\cS$ the category of small \oo groupoids. 
As it is common, we shall also call its objects {\em (small) spaces}. 
The morphisms between two objects $x$ and $y$ of a category $\cC$ form an \oo groupoid (not necessarily small) that we denote most of the time by $\map x y$ and by $\relmap \cC x y$ if the category $\cC$ needs to be recalled. 
We shall say that a map in an \oo category is {\em invertible} or an {\em isomorphism} if it admits both a left and a right inverse.

For an object $X$ in a category $\cC$, the slice and coslice categories shall be denoted as the comma categories $\cC\comma X$ and $X\comma \cC$. 
Since a number of computations of this work involve comma categories, this choice of notations makes things easier.

If $C$ is a small category, the {\em internal groupoid} of $C$ is the maximal subgroupoid of $C$, and the {\em external groupoid} of $C$ is the groupoid obtained by localizing all arrows. 
We shall sometimes denote this last object by $|C|$. 
It can also be defined as $\colim_C1$ in $\cS$.

\subsection{The many enrichments of arrow categories}

We recall the definition of the pushout-product and pullback-hom operations on the maps of a category~\cite[Ch. 11]{Riehl:CHT}.
These operations are involved in a number of enrichments or monoidal structures on arrow categories summarized in \cref{table:enrichment-arrow-cat}.
We recall also their application to the definition of codiagonal and diagonal of a map.
More about all these notions and their relationship to factorisation systems can be found in~\cite{ABFJ:GBM}.

\subsubsection{Cartesian closed structure on arrows}
%\paragraph{pushout-product and pullback-hom}
\label{sec:cartesian-structure-on-arrows}

Let $\cS$ be the category of spaces and $I$ the arrow category $\{0<1\}$.
We denote by $\Arr \cS = \fun I \cS$ the category of arrows of $\cS$.
Let $f:A\to B$ and $g:X\to Y$ be two maps in $\cS$.
The cartesian product $f\times g$ in $\Arr \cS$ is the map $A\times X \to B\times Y$.
The unit is the identity map of $1$.
This cartesian structure is closed, and the corresponding internal hom is given by
\[
\bracemap fg \quad:=\quad \map fg \stto \map {B}{Y},
\]
where $\map fg = \map AX \times_{\map AY}\map BY$ is the space of maps from $f$ to $g$ in $\Arr \cS$.

For $\cC$ another category, the arrow category $\Arr \cC = \fun I \cC$ is naturally enriched over $(\Arr \cS, \times, 1, \bracemap--)$.
We shall also denote by $\bracemap fg$ this enrichement.

\medskip
More generally, if $(\cC,\times,1,\intmap--)$ is a cartesian closed category, then so is the category $\Arr \cC$.
The cartesian product is computed termwise and the internal hom is 
\[
\intbracemap fg \quad:=\quad \intmap fg \stto \intmap {B}{Y},
\]
where $\intmap fg = \intmap AX \times_{\intmap AY}\intmap BY$.

\subsubsection{pushout-product and pullback-hom of arrows}
\label{sec:pp-pbh-structure-on-arrows}

Let $f:A\to B$ and $g:X\to Y$ be two maps in $\cS$.
The {\em pushout-product} $f\pp g$ is defined as the following cocartesian gap map
\[
\begin{tikzcd}
A\times X \ar[r] \ar[d] & A\times Y \ar[d]\ar[ddr, bend left] \\
B\times X \ar[r] \ar[rrd, bend right = 15 ] & B\times X \coprod_{A\times X} A\times Y \pomark \ar[rd, "f\pp g", dashed]\\
&& B\times Y.
\end{tikzcd}
\]
If $0$ and $1$ are the initial and terminal objects of $\cS$, the map $0\to 1$
is the unit for the pushout-product.
The {\em pullback-hom} $\pbh f g$ is defined as the following cartesian gap map
\[
\begin{tikzcd}
\map B X \ar[rd,"\pbh f g", dashed] \ar[rrd, bend left=15] \ar[ddr, bend right]	\\
&\map A X \times_{\map A Y}\map B Y \ar[r]\ar[d]  \pbmark & \map B Y \ar[d]	\\
&\map A X \ar[r] & \map A Y.
\end{tikzcd}
\]
The pushout-product $\square$ and the pullback-hom $\pbh--$ turn $\Arr \cS$ into a symmetric monoidal closed structure.

\medskip

For a general category $\cC$, and two maps $f:A\to B$ and $g:X\to Y$ in $\cC$, the definition of $\pbh f g$ still makes sense if $\map --$ is interpreted as the space of maps in $\cC$.
This provides a functor $\pbh -- : (\Arr \cC)\op \times \Arr \cC \to \Arr \cS$ which is an enrichment of $\Arr \cC$ over $(\Arr \cS, \pp, 0\to 1, \pbh--)$ (which is different from the one of \ref{sec:cartesian-structure-on-arrows}).
If $\cC$ is cocomplete and complete, this enrichment is moreover tensored and cotensored. 
For $f:A\to B$ in $\cS$ and $g:X\to Y$ in $\cC$, the tensor structure if given by the previous formula of $f\pp g$ where the cartesian products $A\times X$ are interpreted as $\coprod_AX$.
Similarly, the cotensor is defined by the formula of $\pbh f g$ where $\map A X$ is interpreted as $\prod_A X$.
The adjunction property between pushout-product and pullback-hom gives a canonical invertible map $\pbh {u\pp f} g = \pbh u {\pbh f g }$.

\medskip

Recall that, for any category $\cC$, the identity maps of objects of $\cC$ provide a functor $\cC\to \Arr \cC$ whose (essential) image is the full category generated by the invertible maps of $\cC$.
This category of invertible maps has some important absorption properties with respect to the pushout-product and pullback-hom.
The proof of the following lemma is straightforward using the definitions of $\pp$ and $\pbh--$.
\begin{lemma}[Absorption properties of invertible maps]
\label{lem:abs-iso}
Let $f$ and $g$ be two maps in $\cC$ and $u$ be a map in $\cS$.
\begin{enumerate}[label={\em (\alph*)}, leftmargin=*]
\item If $u$ or $f$ is invertible then so is $u\pp f$.
\item If $f$ or $g$ is invertible then so is $\pbh g f$.
\end{enumerate}
\end{lemma}

\paragraph{Diagonal and codiagonal}

As an application of the previous structure, we recall how the codiagonal and
diagonal of maps can be written in terms of the pushout-product and pullback-hom.
For $n\geq -1$, let $S^n$ be the $n$\=/dimensional sphere, viewed as an object in $\cS$, with the convention that $S^{-1}=\ini$ (the empty space).
We denote by $s^n:S^n\to 1$ the canonical map to the point.
An easy computation shows that $s^m\pp s^n = s^{n+m+1}$.
Let $\cC$ be a category with finite colimits, 
then for a map $f:A\to B$ in $\cC$ its {\em codiagonal} is the map $\nabla f : B\coprod_AB\to B$
and its {\em iterated codiagonals} are the maps defined recursively by $\nabla^0 f = f$ and $\nabla^{n+1} f = \nabla (\nabla^n f)$.
A short computation proves that $\nabla f = s^0 \pp f$ and $\nabla^n f = s^{n-1} \pp f$.
Dually, if $\cC$ has finite limits, the {\em iterated diagonals} are $\Delta^n f = \pbh {s^{n-1}} f$.
Note also the formula $\Delta^n \pbh gf = \pbh g {\Delta^nf} = \pbh {\nabla^n g}f$.

\smallskip
Recall that a map $f:A\to B$ in $\cS$ is a {\em cover} if it induces a surjection $\pi_0(A) \to \pi_0(B)$.

\begin{lemma}[Whitehead completion]
\label{lem:Whitehead}
A map $w$ in $\cS$ is invertible iff all its iterated diagonals $\Delta^n w =\pbh {s^{n-1}} w$ are covers.
\end{lemma}

\begin{proof}
Recall that, for $\=/2\leq k\leq \infty$ a map $w$ in $\cS$ is $k$\=/connected iff all maps $\Delta^n w$ are covers for $n\leq k+1$~\cite[Prop. 6.5.1.18]{Lurie:HTT}.
The result follows from the fact that $\infty$\=/connected maps in $\cS$ are invertible by Whitehead's theorem.
\end{proof}

\begin{table}[htbp]
\begin{center}	
\caption{The different enrichments on arrow categories}
\label{table:enrichment-arrow-cat}
\medskip
\renewcommand{\arraystretch}{2}
\begin{tabularx}{.9\textwidth}{
|>{\hsize=1\hsize\linewidth=\hsize\centering\arraybackslash}X
|>{\hsize=1\hsize\linewidth=\hsize\centering\arraybackslash}X
|}
\hline
{\em $\cC$ arbitrary category}
    & {\em Enrichments of $\Arr \cC$}
\\
\hline
over $\cS$ 
    & $\map--$ hom space
\\
\hline
over $(\Arr \cS, \times)$
    & $\bracemap--$ external ``arrow-hom''
    
    Tensor $\times$ exists if $\cC$ cocomplete
\\
\hline
over $(\Arr \cS, \pp)$
    & $\pbh--$ external pullback-hom
    
    Tensor $\pp$ exists if $\cC$ cocomplete
\\
\hline
\hline
{\em $(\cC,\times,\intmap--)$ cartesian closed}
& {\em Enrichments of $\Arr \cC$}
\\
\hline
over $(\cC,\times)$ 
    & $\intmap--$  enriched hom
    
    Tensor always exists
\\
\hline
over $(\Arr \cC, \times)$
    & $\intbracemap--$  internal hom

    Cartesian closed structure
\\
\hline
over $(\Arr \cC, \pp)$
    & $\intpbh--$ internal pullback-hom

    Symmetric monoidal closed structure
\\
\hline
\end{tabularx}
\end{center}
\end{table}

\subsection{Orthogonal systems}

Two arrows $f:A\to B$ and $g:X\to Y$ of $\cC$ are said to be {\em weakly orthogonal} (resp. {\em orthogonal}) if the map $\pbh f g$ is a cover in $\cS$ (resp. an invertible map in $\cS$).
We shall denote respectively by $f\pitchfork g$ and $f\perp g$ these two relations.
The first relation says that for any commutative square $\alpha:f\to g$ there exists a diagonal lift.
The second says that such a lift exists and is unique.
Given a diagram $W\to \Arr \cC$ we can define $W^\pitchfork = \{g\, |\, \forall f \in W, f\pitchfork g\}$ and $^\pitchfork W = \{f\, |\, \forall g \in W, f\pitchfork g\}$.
We shall always look at $W^\pitchfork$ and $^\pitchfork W$ as (replete) full subcategories of $\Arr \cC$.
However, the generating data $W\to \Arr \cC$ can be an arbitrary diagram,
indexed by a set or any other category.
We define similarly $W^\perp$ and $^\perp W$.
We shall say that a pair $(\cA,\cB)$ of full subcategories of $\Arr \cC$ is a
{\em weak orthogonal system} if $\cB = {^\pitchfork}\cA$ and $\cA = \cB^\pitchfork$.
For any diagram $W\to \Arr \cC$, the pair $(\cA_W,\cB_W) := \left(^\pitchfork (W^\pitchfork), W^\pitchfork\right)$ is a weak orthogonal system, said to be generated by $W$. 
The pair $(\cA_W,\cB_W)$ depends only on the essential image of the functor $W\to \Arr \cC$.
We shall say that a pair $(\cL,\cR)$ of full subcategories of $\Arr \cC$ is an
{\em orthogonal system} if $\cR = {^\perp}\cL$ and $\cL = \cR^\perp$.
For any diagram $W\to \Arr \cC$, the pair $(\cL_W,\cR_W) := \left(^\perp (W^\perp), W^\perp\right)$ is an orthogonal system, said to be generated by $W$.
The pair $(\cL_W,\cR_W)$ depends only on the essential image of the functor $W\to \Arr \cC$.

\medskip
We recall some properties of (weak) orthogonal systems.

\begin{lemma}
\label{lem:prop-wfs}
Let $(\cA,\cB)$ be the weak orthogonal system generated by some $W$.
Then $\cA$ is stable by discrete sums (sums indexed by a set), cobase change and transfinite composition.
\end{lemma}

\begin{rem}
\label{rem:defect-sum}
The previous statement about discrete sums cannot a priori be extended to more general sum, that is, colimits indexed by a space.
Let $(\cA,\cB)$ be a weak factorization, $f$ be a map in $\cB$, and $w: I \to \cA$ a family of maps in $\cA$ indexed by a space $I$.
All maps $\pbh {w_i} f$ are covers and the question is whether the product $\prod_{i:I} \pbh {w_i} f$ is still a cover.
When $I$ is a set, this is true, but arbitrary products of covers are not covers in general.
As a counter-example, consider the weak factorization system on $\cS$ generated by the map $0\to 1$.
A map is in right class if and only if it is surjective on $\pi_0$ if and only if it is a cover.
The sum indexed by a space $I$ of copies of $0\to 1$ is the map $0\to I$.
But covers do not have a right lifting property with respect to maps $0\to I$.
\end{rem}

\begin{lemma}
\label{lem:prop-ufs}
Let $(\cL,\cR)$ be the orthogonal system generated by some $W$.
\begin{enumerate}[label={(\alph*)}, leftmargin=*]
\item \label{lem:prop-ufs:1}
	$\cL$ is absorbing for $\pp$: for any $u$ in $\Arr \cS$ and any $f$ in $\cL$, the map $u\pp f$ is in $\cL$.
\item \label{lem:prop-ufs:1-dual}
	$\cR$ is absorbing for $\pbh--$: for any $u$ in $\Arr \cS$ and any $f$ in $\cR$, the map $\pbh u f$ is in $\cR$.
\item \label{lem:prop-ufs:2}
	$\cL$ is stable by colimits in $\Arr \cC$.
\item \label{lem:prop-ufs:2-dual}
	$\cR$ is stable by limits in $\Arr \cC$.
\end{enumerate}

\end{lemma}

\begin{proof}
By duality, it is enough to prove the statements for the left class.

\smallskip
\noindent \ref{lem:prop-ufs:1}
For any $g$, we have $\pbh {u\pp f} g = \pbh u {\pbh f g}$.
By the absorption properties of \cref{lem:abs-iso}, this map is invertible as soon as $f\perp g$.
Applied when $g\in W^\bot$, this proves the first assertion.

\smallskip
\noindent \ref{lem:prop-ufs:2}
Let $f_i$ be a diagram in $\cL$, we have $\pbh {\colim f_i} g = \lim \pbh {f_i} g$.
If $g\in W^\bot$, all maps $\pbh {f_i} g$ are invertible, and so is their limit.
This proves that $\colim f_i$ is in $\cL$.
\end{proof}

\begin{lemma}
\label{lem:FSisWFS}
A pair $(\cL,\cR)$ is an orthogonal system iff it is a weak orthogonal system whose left class is stable by codiagonal.
\end{lemma}
\begin{proof}
We prove first that the condition is necessary.
If $(\cL,\cR)$ is an orthogonal system, the class $\cL$ is stable by codiagonal as a consequence of \cref{lem:prop-ufs}.\ref{lem:prop-ufs:1}.
Then, we need to prove that $(\cL,\cR)$ is a weak orthogonal system, that is, that $\cL^\pitchfork = \cR$.
We have for free that $\cR \subset \cL^\pitchfork$.
Let $g$ be a map in $\cL^\pitchfork$, then for any $f$, the map $\pbh fg $ is a cover in $\cS$.
Consequently, all maps $\Delta^n \pbh fg = \pbh {\nabla^nf} g $ are covers and by \cref{lem:Whitehead} the map $\pbh fg$ is invertible.
The argument to prove that $\cL ={^\pitchfork}\cR$ is similar.

Reciprocally, if $(\cL,\cR)$ is a weak orthogonal system and $\cL$ is stable by codiagonal, we have that $\Delta^n \pbh fg = \pbh {\nabla^nf} g $ is a cover for any $g$ in $\cR$ and $f$ in $\cL$.
This proves that $\cR\subset \cL^\perp$.
Reciprocally, we have $\cL^\perp \subset \cL^\pitchfork = \cR$. This proves $\cR = \cL^\perp$.
The equality $\cL= {^\perp} \cR$ is proven similarly.
\end{proof}

\subsection{Factorization systems}

A {\em weak factorisation system} (WFS) on $\cC$ is the data of a weak orthogonal system $(\cA,\cB)$ such that, for every map $f$ in $\cC$, there exists a factorisation $f= ab$ where $a\in \cA$ and $b\in \cB$.
This factorization is only assumed to exist, it is thus a property of a weak orthogonal system. The factorization is not unique in general.
%
%\medskip
A {\em factorisation system} (FS) on $\cC$ is the data of an orthogonal system $(\cL,\cR)$ such that, for every map $f$ in $\cC$, there exists a factorisation $f= ab$ where $a\in \cL$ and $b\in \cR$.
Again the factorization is only assumed to exist and is a property of an orthogonal system. 
However, in this case the factorization can be proven to be unique (precisely, the category of such factorizations is a contractible groupoid).
For a map $f:A\to B$, we shall denote by $\lambda(f):A\to C$ and $\rho(f):C\to B$ the two components of the factorization.
The functors $f\mapsto \lambda(f)$ and $f\mapsto \rho(f)$ define respectively
right and left adjoints to the inclusions $\cL\subset \Arr \cC \supset \cR$.
In particular, $\cR$ is a reflective subcategory of $\Arr \cC$.

\medskip
A weak orthogonal system shall be called {\em accessible} if it is of the form $\left(^\pitchfork (W^\pitchfork), W^\pitchfork\right)$ for $W\to \Arr \cC$ a small diagram of maps with small domain.
An orthogonal system shall be called {\em accessible} if it is of the form $\left(^\perp (W^\perp), W^\perp\right)$ for $W\to \Arr \cC$ a small diagram of small objects in $\Arr \cC$.
These definitions apply in particular to factorization systems, weak or not.

\subsubsection{Modalities and lex modalities}
\label{sec:stableFS}

A factorization system is called {\em stable}, or a {\em modality} if the factorization is stable under base change~\cite{ABFJ:GBM,RSS}.
This is equivalent to asking that both classes $\cL$ and $\cR$ be stable under base change.
Since the right class $\cR$ is always stable under base change, the condition is in fact only on $\cL$.

A factorization system is called a {\em left-exact modality} (or a {\em lex modality} for short) if the factorization is stable under finite limits in $\Arr \cC$.
Again, since the right class $\cR$ is always stable under limits, the condition is in fact only on the class $\cL$.
A number of characterizations of lex modalities is given in~\cite[Thm 3.1]{RSS}, but not the one we just used as a definition.
\Cref{prop:stable/lex-FS}\ref{prop:stable/lex-FS:enum2} proves that our definition is equivalent to~\cite[Thm 3.1(x)]{RSS}.
Recall that a map $f\to g$ in $\Arr \cC$ is cartesian if the corresponding square in $\cC$ is cartesian.

\begin{prop}
\label{prop:stable/lex-FS}
Let $(\cL,\cR)$ be a factorization system on $\cC$.
\begin{enumerate}[label=(\alph*), leftmargin=*]
\item \label{prop:stable/lex-FS:enum1} $(\cL,\cR)$ is a modality if and only if
  the reflection $\rho:\Arr \cC \to \cR$ preserves cartesian maps.
\item \label{prop:stable/lex-FS:enum2} $(\cL,\cR)$ is a lex modality if and only if the reflection $\rho:\Arr \cC \to \cR$ is left-exact. 
        In particular, a lex modality is a modality.
\end{enumerate}
\end{prop}
\begin{proof}
\noindent \ref{prop:stable/lex-FS:enum1}
Let $f:A\to B$ and $f':A'\to B'$ be two maps in $\cC$ and let $\alpha:f'\to f$ be a cartesian map between them.
The factorization of $f'$ and $f$ induces a diagram
\[
\begin{tikzcd}
A'  \ar[d,"\lambda(f')"'] \ar[rr] \ar[rrd,phantom,"\lambda(\alpha)"]
    && A \ar[d,"\lambda(f)"]
\\
C'  \ar[d,"\rho(f')"'] \ar[rr] \ar[rrd,phantom,"\rho(\alpha)"]
    && C \ar[d,"\rho(f)"]
\\
B'  \ar[rr] 
    && B
\end{tikzcd}
\]
Then, the factorization system is a modality if and only if both
$\lambda(\alpha)$ and $\rho(\alpha)$ are cartesian squares.
But by the cancellation property of cartesian squares, this is equivalent to $\rho(\alpha)$ being cartesian only.

\smallskip
\noindent \ref{prop:stable/lex-FS:enum2}
Let $f_i:A_i\to B_i$ be a finite diagram in $\Arr \cC$ and $f= \lim f_i:A\to B$.
We have a canonical diagram
\[
\begin{tikzcd}
A  \ar[d,"\lambda(f)"'] \ar[rr, equal]
    && \lim A_i \ar[d,"\lim \lambda(f_i)"]
\\
C  \ar[d,"\rho(f)"'] \ar[rr]
    && \lim C_i \ar[d,"\lim \rho(f_i)"]
\\
B  \ar[rr, equal] 
    && \lim B_i.
\end{tikzcd}
\]
The factorization system is left-exact if and only the map $C\to \lim C_i$ is an isomorphism.
But this is equivalent to $\rho (\lim f_i ) = \lim \rho(f_i)$.

For the last statement, let $f:X\to Y$ and $f':X'\to Y'$ be two maps in $\cC$ and $\alpha:f'\to f$ be a map in $\Arr \cC$.
We use the remark that $\alpha$ is cartesian iff the following square is cartesian in $\Arr \cC$
\[
\begin{tikzcd}
f' \ar[r] \ar[d] & f\ar[d] \\
1_{Y'} \ar[r]& 1_{Y}.
\end{tikzcd}
\]
Since the reflection $\rho:\Arr \cC \to \cR$ does not change the codomain of maps, the image by $\rho$ of such a square is a square of the same type.
This proves that a lex modality is a modality.
\end{proof}

It is possible to strengthen the orthogonality relation to capture the stability
condition of a modality (see the notion of fiberwise orthogonality in~\cite{ABFJ:GC}).
It is also possible to define an orthogonality relation
suited to lex modalities, but we shall not need this (see \cref{table:strengthFS}).

\subsubsection{Enriched factorization systems}
\label{sec:enrichedFS}

When $(\cC,\otimes,\bbone,\intmap--)$ is a symmetric monoidal closed cocomplete category, the pushout-product and pullback-hom define a symmetric closed monoidal structure $(\Arr \cC,\boxtimes,0 \to \bbone,\intpbh--)$.
We define the {\em enriched orthogonality} relation $f\intperp g$ by the condition that the enriched pullback-hom $\intpbh f g$ be an invertible map in $\cC$.
The enriched and external orthogonality are related by
\[
f \iperp g \quad\IFF\quad \forall X,\ (X\otimes f) \perp g \quad\IFF\quad \forall X,\ f \perp \intmap X g.
\]
An {\em enriched orthogonal system} is a pair $(\cL,\cR)$ of full subcategories of $\Arr \cC$ such that $\cR=\cL^\intperp$ and $\cL={^\intperp}\cR$.
An external orthogonal system $(\cL,\cR)$ is enriched iff $\cL$ is stable by $X\otimes-$ for any $X$ in $\cC$, iff $\cR$ is stable by $\intmap X-$ for any $X$ in $\cC$.
An {\em enriched factorization system} is an enriched orthogonal system $(\cL,\cR)$ such that, for every map $f$ in $\cC$, there exists a factorisation $f= \rho(f)\lambda(f)$ where $\lambda(f)\in \cL$ and $\rho(f)\in \cR$.
Such a factorization can be proven to be unique.
An enriched orthogonal system shall be called {\em accessible} if it is of the form $\left(^\intperp (W^\intperp), W^\intperp\right)$ for $W\to \Arr \cC$ a small diagram of (internally) small objects in $\Arr \cC$.

In the case where $(\cC,\otimes,\bbone,\intmap--)$ is a cartesian closed category, the enriched orthogonality relation also satisfies
\[
f \iperp g \quad\RA\quad \forall X,\ f\perp X\otimes g.
\]
This is because $X\otimes g = X\times g$ is the base change of $g:A\to B$ along $X\times B\to B$ and the orthogonality relation is always stable by base change in its second variable.
Let $(\cL,\cR)$ be an enriched factorization system in $\cC$, then the factorization of a map $f:A\to B$ is always stable by base change along a projection $X\times B\to B$.
This is weaker than being a modality, for which the
factorization has to be stable under arbitrary base changes (see \cref{table:strengthFS}).

\begin{table}[htbp]
\begin{center}	
\caption{Strength of orthogonality relations and of factorization systems}
\label{table:strengthFS}
\medskip
\renewcommand{\arraystretch}{2}
\begin{tabularx}{.9\textwidth}{
|>{\hsize=.8\hsize\linewidth=\hsize\centering\arraybackslash}X
|>{\hsize=1.2\hsize\linewidth=\hsize\centering\arraybackslash}X
|>{\hsize=1\hsize\linewidth=\hsize\centering\arraybackslash}X
|}
\hline
$\cC$ is any category 
    & {\sl External orthogonality}
    
    $w\perp f$ 
    & Factorization systems
\\
\hline
$\cC$ is a cartesian closed category 
    & {\sl Enriched orthogonality}
    
    $w \iperp f \ :\iff\ \forall X, (X\times w) \perp f$ 
    & Enriched FS 
    
    (stable by product)
\\
\hline
$\cC$ is a locally cartesian closed category with 1
    & {\sl Fiberwise orthogonality} 
    
    $w\fwperp f \ :\iff\ w'\perp f$ 
    
    $\forall$ base change $w'\to w$ 
    & Modality 
    
    (FS stable by base~change)
\\
\hline
$\cC$ is a topos
    & {\sl Lex orthogonality}
    
    for $W\subset \Arr \cC$, 
    
    $W\lexperp f\ :\iff\ \lim_iw_i\perp f$ 
    
    $\forall$ finite diagram $I\to W$
    & Lex~modality/ lex~localization 
    
    (FS stable by finite limits)
\\
\hline
\end{tabularx}
\end{center}
\end{table}

\subsubsection{Reflective factorization systems}
\label{sec:localFS}

If $\cC$ has a terminal object, we shall say that a factorization system $(\cL,\cR)$ is {\em reflective} if the class $\cL$ satisfies the ``2~out~of~3" condition \cite{CHK:slex}.
Since left classes of orthogonal systems are always stable by composition and right cancellation, this condition is equivalent to $\cL$ having the left cancellation condition.
If $(\cL,\cR)$ is a factorization system, we denote by $\cR_1$ the full
subcategory of $\cC$ spanned by objects $X$ such that $X\to 1$ is in $\cR$.
For any $X$ in $\cC$, the factorization of the map $X\to 1$ produces a reflection of $\cC$ into $\cR_1$.
When $(\cL,\cR)$ is a reflective factorization system, the ``2~out~of~3" condition implies that the class of maps inverted by this reflection is exactly $\cL$.

\medskip

Any modality $(\cL,\cR)$ defines a reflective localization $P:\cC\rightleftarrows \cR_1:\iota$.
This localization is always associated to a reflective factorization system $(\cL^{3/2},\cR^{2/3})$ where $\cL^{3/2}$ is the completion of $\cL$ under the ``2~out~of~3" condition and $\cR^{2/3}\subset \cR$ is the class of maps $X\to Y$ in $\cC$ that can be obtained by base change from some map in $\cR_1$.
The factorization $X\to Z\to Y$ of a map $X\to Y$ is given by the diagram
\[
\begin{tikzcd}
X \ar[r]\ar[rd]\ar[rr,bend left=30,"\lambda(X\to 1)"]  & Z \ar[r] \ar[d] \pbmark & PX \ar[d] \\
    & Y \ar[r,"\lambda(Y\to 1)"'] & PY.
\end{tikzcd}
\]
The maps $X\to PX$ and $Y\to PY$ are in $\cL$ by definition of $P$.
The map $Z\to PX$ is in $\cL$ because $\cL$ is stable by base change.
Hence, the map $X\to Z$ is in $\cL^{3/2}$.
The map $Z\to Y$ is on $\cR^{2/3}$ since it is a base change of the map $PX\to PY$ in $\cR_1$.
From there, the relations $\left(\cL^{3/2}\right)^\perp = \cR^{2/3}$ and $\left(\cR^{2/3}\right)^\perp = \cL^{3/2}$ can be proved easily, 
so $(\cL^{3/2},\cR^{2/3})$ is indeed a factorization system.

Recall that a reflective localization $P:\cC\rightleftarrows \cD$ is {\em semi-left-exact} in the sense of \cite{CHK:slex} or {\em locally cartesian closed} in the sense of \cite{Gepner-Kock:Univalence} if $P$ preserves the pullbacks along maps in $\cD$ (viewed as maps in $\cC$). 
The following result appear also in \cite{Rijke:modal-descent}.

\begin{prop}
\label{prop:mod-slex}
The localization $P:\cC\rightleftarrows \cR_1$ associated to a modality is always semi-left-exact.
\end{prop}
\begin{proof}
Any map in $\cR_1$ is of the type $PX\to PY$ for some map $X\to Y$ in $\cC$.
We consider a pullback square 
\begin{equation}
\tag{\ensuremath{\star}}
\label{eqn:mod-slex}
\begin{tikzcd}
X' \ar[r,"f"]\ar[d] \pbmark & PX \ar[d] \\
Y' \ar[r,"g"'] & PY.
\end{tikzcd}
\end{equation}
and its factorization for the modality $(\cL,\cR)$
\[
\begin{tikzcd}
X' \ar[r,"\lambda(f)"]\ar[d] \pbmark & X'' \ar[r,"\rho(f)"]\ar[d] \ar[rd,"P(\ref{eqn:mod-slex})" description, phantom] \pbmark & PX \ar[d] \\
Y' \ar[r,"\lambda(g)"'] & Y'' \ar[r,"\rho(g)"'] & PY.
\end{tikzcd}
\]
Both squares are cartesian by \cref{prop:stable/lex-FS}\ref{prop:stable/lex-FS:enum1}.
Since the maps $PX\to 1$ and $PY\ot 1$ are in $\cR$ by definition of $P$, the objects $X''$ and $Y''$ are in fact $PX'$ and $PY'$.
Because $P$ is idempotent, the square $P(*)$ is then the image of the square (\ref{eqn:mod-slex}) by $P$.
This proves that $P$ preserves such cartesian squares.
\end{proof}

\begin{cor}[{\cite[Prop. 1.4]{Gepner-Kock:Univalence}}]
\label{cor:mod-LCCC}
In the situation of \cref{prop:mod-slex}, the local category $\cR_1$ has always universal colimits.
Moreover, if $\cC$ is locally cartesian closed with a terminal object, then so is $\cR_1$.
\end{cor}

\begin{prop}[{\cite[Thm 3.1]{RSS}}]
\label{prop:lex-loc=lex-mod}
The following conditions on a modality $(\cL,\cR)$ are equivalent:
\begin{enumerate}
\item the reflection $\cC \to \cR_1$ is left-exact;
\item $(\cL,\cR)$ is a reflective factorization system;
\item $(\cL,\cR)$ is a lex modality.
\end{enumerate}
\end{prop}

\begin{table}[htbp]
\begin{center}	
\caption{Summary of the properties of factorization systems}
\label{table:propFS}
\medskip
\renewcommand{\arraystretch}{2}
\begin{tabularx}{\textwidth}{
|>{\hsize=1\hsize\linewidth=\hsize\centering\arraybackslash}X
|>{\hsize=1\hsize\linewidth=\hsize\centering\arraybackslash}X
|>{\hsize=1\hsize\linewidth=\hsize\centering\arraybackslash}X
|>{\hsize=1\hsize\linewidth=\hsize\centering\arraybackslash}X
|>{\hsize=1\hsize\linewidth=\hsize\centering\arraybackslash}X
|>{\hsize=1\hsize\linewidth=\hsize\centering\arraybackslash}X|}
\hline
\multicolumn{2}{|c|}{\multirow{2}{*}{\em Factorization~system}}&\multicolumn{4}{c|}{\em Stability properties of $\cL$}\\
\cline{3-6}
\multicolumn{2}{|c|}{\multirow{1}{*}{$(\cL,\cR)$}}& {\em none} & {\em base change along $\cR$} & {\em all base changes} & {\em finite limits}\\
\hline
\multirow{2}{6em}{\em Cancellation properties of $\cL$}	& {\em composition and right cancellation}		& any FS & semi-modality & modality & lex modality \\
\cline{2-5}
	& {\em + left cancellation (2~out~of~3)}	& reflective FS	& semi-lex reflective FS 	& \multicolumn{1}{c}{lex reflective FS}	& =	\\
\hline
\end{tabularx}
\end{center}  
\end{table}

\subsubsection{Modal topoi}
\label{sec:modal-topos}
\label{sec:perverse-topos}

When $\cE$ is an \oo topos with a modality $(\cL,\cR)$, the category of {\em modal objects} $\cR_1\subset \cE$ is not a topos in general, but it still enjoys many good properties. 
As we mentioned above, it is always a locally cartesian closed category with a terminal object and it is presentable as soon as the modality is accessible.
Since a number of our constructions are valid in such a category, we shall call them {\em modal topoi} (and {\em accessible modal topoi} if the modality is accessible).
We shall see in \cref{prop:PLCCC=modal-topos} that accessible modal topoi are the same thing as presentable locally cartesian closed categories.

An important example of an accessible modal topos is the category of {\em local objects} for a map $A\to 1$ in a topos $\cE$, that is objects $X$ such that $X\to X^A$ is invertible (where $X^A$ is the exponential in $\cE$).
In this case, the enriched local objects are also the {\em modal objects} for the (enriched) modality generated by the object $A$.
When $A=S^{n+1}$, the notion of $A$\=/modal topos gives back that of $n$\=/topos.
In the context of real cohesion, when $A=\RR$ the corresponding modal topos is the topos of discrete objects \cite{Schreiber:cohesion,Shulman:cohesion}.
In the context of algebraic geometry and of the topos of Nisnevich sheaves over smooths schemes, when $A=\AA^1$ the corresponding modal topos is the category of $\AA^1$\=/local (unstable) sheaves \cite{Robalo:motives}.

\section{Small object arguments}

The small object argument (SOA) is a construction to enhance a weak orthogonal system into a weak factorisation system~\cite{Quillen:HA,Riehl:CHT}.
The variants presented in \ref{sec:!SOA1} and \ref{sec:!SOA2} are constructions to enhance an orthogonal system into a factorisation system.
In order to better compare the classical SOA and its variants, we shall recall the classical version first in \ref{sec:SOA}.

\paragraph{Source and target of a map} 
For a map $f$ in $\cC$, we introduce the notation $\s f$ and $\t f$ to refer to the source (domain) and target (codomain) of $f$, that is, we have $f:\s f\to \t f$.
We shall use the terms source and domain (resp. codomain and target) interchangeably.

\subsection{SOA for WFS after Quillen}
\label{sec:SOA}

\begin{rem}[Adaptation for \oo categories]
\label{rem:set-v-space}
The small object argument is usually formalized in a 1-category.
We are going to provide a formulation valid also in an \oo category $\cC$, but it turns out that a mild adaptation is needed because the left class $\cA$ of a {\em weak} factorization system is only stable under {\em discrete} sums and not sums indexed by arbitrary spaces (see \cref{rem:defect-sum}).
Therefore, for any two objects $A$ and $B$, the argument cannot use the {\em space} of maps $\map A B$ and we need to replace it by an actual {\em set} of maps.
In order to do so, we shall consider fixed, for any pair $(A,B)$ of objects, a set $\relmap 0 A B $ together with a cover $\relmap 0 A B  \onto \map A B$. 
The set $\relmap 0 A B$ will play the role of a set of maps from $A$ to $B$.
Using Dwyer-Kan~\cite{Dwyer-Kan:1987}, such a choice of sets can always be assumed to be compatible with composition.
\end{rem}

We fix a small set of maps $W\to \Arr \cC$ in a cocomplete category $\cC$.
For $f:X\to Y$ a map in $\cC$, the small object argument computes a factorisation of $f$ for the weak orthogonal system $(\cA_W,\cB_W)$ by a transfinite iteration of the following construction $\Quillenmap f$.
First, remark that the condition $W\pitchfork f$ is equivalent to the
existence of a diagonal lift in the single canonical square
\[
\epsilon\ :\ \coprod_{w\in W} \map w f _0 \times w\stto f
\]
where $\map{w}{f}_0$ is a chosen set of maps between $w$ and $f$ as in \cref{rem:set-v-space},
and $\map w f _0 \times w = \coprod_{\map w f _0} w$ is the tensor structure of
$\Arr \cC$ over $\cS$.
The square defined by $\epsilon$ is 
\begin{equation}
\tag{$\boxslash$}\label{eqn:liftcoend}
\begin{tikzcd}
\coprod_{w} \map w f _0 \times \s w \ar[r]\ar[d,"\coprod_w\map{w}{f} \times w"'] & \s f\ar[d, "f"]\\
\coprod_{w} \map w f _0 \times \t w \ar[r] & \t f.
\end{tikzcd}
\end{equation}
We define the map $\Quillenmap f:\Quillen f\to \t f$ as the cocartesian gap map of this square
\begin{equation}
\tag{$q$\=/construction}
\label{eqn:qconstruction}
\begin{tikzcd}
\coprod_{w} \map w f _0\times \s w \ar[r]\ar[d,"\coprod_w\map w f _0 \times w"'] & \s f\ar[d,"u(f)"'] \ar[ddr, bend left, "f"]&\\
\coprod_{w} \map w f _0\times \t w \ar[r, "\ell"'] \ar[rrd, bend right] & \Quillen f \ar[rd,"\Quillenmap f"'] \pomark &\\
&& \t f.
\end{tikzcd}
\end{equation}
It is in fact more convenient to look at this as a diagram in $\cC^\to$:
\begin{equation}
\tag{$q$\=/lift}
\label{eqn:qlift}
\begin{tikzcd}
\coprod_{w} \map w f _0\times \s w \ar[r]\ar[d,"\coprod_w\map w f _0 \times w"'] & \s f\ar[d, "f", near end] \ar[r, "u(f)"] & \Quillen f\ar[d, "\Quillenmap f"]\\
\coprod_{w} \map w f _0\times \t w \ar[r] \ar[rru, dashed, "\ell" near start] & \t f \ar[r,equal] & \t f.
\end{tikzcd}
\end{equation}
The map $\Quillenmap f$ can be understood as the best approximation of $f$ on the right that admits a lift from the map $\coprod_w\map w f _0 \times w$.

\medskip
We get a factorisation of $f$ as $\Quillenmap fu(f):\s f\to \Quillen f\to \t f$.
The map $\coprod_w\map w f _0 \times w$ is in $\cA_W$
since $\cA_W$ is stable by discrete sums,
and the map $u(f):\s f\to \Quillen f$ is in $\cA_W$ since $\cA_W$ is stable by cobase change (\cref{lem:prop-wfs}).
The map $\Quillenmap f:\Quillen f\to \t f$ need not be in $\cB_W$, but it will be after a transfinite iteration.

\begin{thm}[Small object argument]
\label{thm:SOA}
Let $\cC$ be a cocomplete category and $W$ a small set of arrows in $\cC$ with small domains.
Then, the transfinite iteration of the $q$\=/construction converges to a factorization for the weak orthogonal system $(\cA_W, \cB_W)$ generated by $W$.
\end{thm}

\begin{proof}
The natural transformation $u:f\to \Quillenmap f$ defines a transfinite sequence 
\[
f\ntto{u(f)} \Quillenmap f\ntto{u(\Quillenmap f)} \Quillenmapn2 f\ntto{u(\Quillenmapn2 f)} \dots
\]
We assumed that the domains $\s w$ were small for all $w\in W$. 
Let $\kappa$ be a regular cardinals majoring the size of all $\s w$.
The map $u^\kappa(f)$ is in $\cA_W$ because it is a transfinite composition of maps in $\cA_W$.
The result will be proven if we prove that the map $\Quillenmapn \kappa f $ is in $\cB_W=W^\pitchfork$.

Let $w$ be in $W$, $\alpha:w\to \Quillenmapn \kappa f $ be a map and $L(\alpha)$ be its space of lifts, we want to show that $L(\alpha)$ is not empty.
We define $\Quillenn \kappa f:=\s \Quillenmapn \kappa f$.
Because $\s w$ is $\kappa$\=/small and $\kappa$ is $\kappa$\=/filtered, the map $\s w\to \Quillenn \kappa f $ associated to $\alpha$ factors through some $\Quillenn \lambda f $ for some ordinal $\lambda<\kappa$.
By construction of $\Quillenn {\lambda+1} f $, we have a commutative diagram
\[
\begin{tikzcd}
\s w \ar[r]\ar[d, "w"'] 
& \Quillenn \lambda f \ar[d, "\Quillenmapn \lambda f "] \ar[rr, "u(\Quillenmapn \lambda f )"] 
&& \Quillenn {\lambda+1} f \ar[d, "\Quillenmapn {\lambda+1} f "]\ar[r]
&\dots \ar[r]
& \Quillenn \kappa f \ar[d, "\Quillenmapn \kappa f"]
\\
\t w \ar[r] \ar[rrru, dashed, bend left=10] 
& \t f \ar[rr, equal]
&& \t f \ar[r, equal] 
& \dots  \ar[r,equal] 
& \t f.
\end{tikzcd}
\]
and the dotted map provides the desired lift.
This finishes the proof that $f = \Quillenmapn \kappa f  u^\kappa(f)$ provides a factorisation of $f$ in $(\cA_W,\cB_W)$.
Moreover, the proof shows that this factorisation can be assumed natural in $f$.
\end{proof}

\begin{rem}[A variant]
\label{rem:Quillen-variant}
When $\cC$ is a 1-category, the operator $W_Q(f) = \coprod_{w\in W} \map w f _0
\times w$ can be understood as the density comonad associated to the diagram $W\to \Arr \cC$.
\[
\begin{tikzcd}
W \ar[d] \ar[r] & \Arr \cC \ar[loop right, "W_Q"]   \\
\fun {W\op} \Set \ar[ru, dashed, "R", shift left] \ar[from=ru, dashed, "N", shift left]
\end{tikzcd}
\]
The Quillen operator $q$ is then the pushout in the arrow category
\[
\begin{tikzcd}
W_Q(f) \ar[r] \ar[d] & f \ar[d]
\\
\t_Q f \ar[r] & \Quillenmap f \pomark .
\end{tikzcd}
\]
where $\t_Q f$ is the identity map of the codomain of $W_Q(f)$.

When $\cC$ is a general \oo category, the choice of set of morphisms of \ref{rem:set-v-space} prevents such an interpretation.
Nonetheless, it is possible to define a variant on Quillen's construction where we do not use such a set of maps, but the canonical space of maps.
This variant supports the diagram $W\to \Arr \cC$ to be indexed by an arbitrary category and not only a set.
For such a diagram, we can define a density comonad $W_M$ 
\[
\begin{tikzcd}
W \ar[d] \ar[r] & \Arr \cC \ar[loop right, "W_M"]   \\
\fun {W\op} \cS \ar[ru, dashed, "R", shift left] \ar[from=ru, dashed, "N", shift left]
\end{tikzcd}
\]
and an operator $f\mapsto M(f)$
\[
\begin{tikzcd}
W_M(f) \ar[r] \ar[d] & f \ar[d]
\\
\t_W f \ar[r] & M(f) \pomark .
\end{tikzcd}
\]
where $\t_M f$ is the identity map of the codomain of $W(f)$. 
The codomain of $M(f)$ is $\t f$ and thus $M(f)$ provides a factorization of $f$.
Let $(\cA_W,\cB_W)$ be the WFS generated by $W$.
The factorization of $f$ resulting from the iteration of $M$ has its right component in $\cB_W$.
However, the left component is not in general in $\cA_W$:
the map $W_M(f)= \colim_{W\comma F}w$ is a colimit of maps in $W$ but $\cA_W$ is not stable under colimits (since it would then be a FS). 
This construction seems to provide a WFS between the WFS $(\cA_W,\cB_W)$ and the FS $(\cL_W,\cR_W)$.
It is not clear to us what should be the extra property of such a WFS.
Also, this setting seems related to the plus-construction, since the codomain of $W_M(f)$ is exactly the operator $\Plus f$ of \cref{defi:+construction}.
\end{rem}

\begin{rem}[Garner construction]
\label{rem:garner1}
Quillen's construction provides a pointed endofunctor $\pi_f:f \to \Quillenmap f$ on $\Arr \CC$.
In~\cite{Garner:SOA}, Garner establishes a variant on the construction by considering not the colimit of the sequence of $f \to \Quillenmap f\to \Quillenmapn2 f\to \dots$ but the free monad $\gamma$ generated by the pointed endofunctor $\pi:1\to q$.
The map $\gamma(f)$ is essentially the quotient of $\Quillenmap f$ imposing that the two maps $q(\pi_f),\pi_{\Quillenmap f}:\Quillenmap f \to \Quillenmapn2 f$ coincide (as well as their analogs between higher iterations of $k$).
The resulting map $\gamma(f)$ is still in the right class $\cR$ and is equipped with a free algebra structure over the pointed endofunctor $q$.
The two constructions coincide if $\pi:1 \to q$ is a well-pointed endofunctor.
\end{rem}

\subsection{SOA for FS after Gabriel-Ulmer}
\label{sec:!SOA1}

This section explains a first adaptation of the SOA in order to construct the FS generated by the maps $W$.
This argument is in the spirit of the one used by Gabriel and Ulmer~\cite{Gabriel-Ulmer}. 
It consists in adding maps to $W$ in order to force the uniqueness of the lifts when applying the $q$ construction.
Given a square $\alpha:w\to f$, let us assume that we have two diagonal lifts
\[
\begin{tikzcd}
A\ar[d, "w"'] \ar[rr] && X \ar[d,"f"]\\
B\ar[rr]\ar[rru,dashed, bend left=10, "h"]\ar[rru,dashed, bend right=10, "h'"'] && Y
\end{tikzcd}
\]
These two lifts provide a new commutative square
\[
\begin{tikzcd}
B\coprod_AB\ar[d, "\nabla w"'] \ar[rr, "h \coprod_A h'"] && X \ar[d,"f"]\\
B\ar[rr] && Y
\end{tikzcd}
\]
where $\nabla w$ is the codiagonal of $w$.
The existence of a diagonal lift for this square is equivalent to an identification $h\simeq h'$.
For 1-categories this would be sufficient to prove that $h=h'$, but within $\infty$\=/categories, such an identification has itself to be proven to be unique (\cref{lem:Whitehead})

\begin{lemma}
\label{lem:codiagonal}
Let $\cC$ be a cocomplete category and $W$ a small set of objects in $\Arr \cC$ with small domains and codomains.
Let $W^\nabla$ be the smallest class of arrows in $\cC$ containing $W$ and stable by codiagonal, then 
\begin{enumerate}[label={(\alph*)}, leftmargin=*]
\item \label{lem:codiagonal:1}
	$W^\nabla$ is a small set of maps with small domains, and
\item \label{lem:codiagonal:2}
	the orthogonal system generated by $W$ coincides with the weak orthogonal system generated by $W^\nabla$.
\end{enumerate}
\end{lemma}

\begin{proof}
    
\noindent \ref{lem:codiagonal:1}
$W^\nabla$ can be constructed explicitly as follows.
Let $W_0=W$ and $W_{n+1}$ be the completion of $W_n$ for the codiagonals: $W_{n+1}=W_n\cup S^0\pp W_n$.
Then $W^\nabla$ is the colimit of the inclusion $W_n\subset W_{n+1}$ for all $n\in \NN$.
This proves that $W^\nabla$ is set if $W$ is.
For $f:A\to B$ in $W$, the domain of $\nabla f$ is $B\coprod_AB$.
This is a small object of $\cC$ since both $A$ and $B$ are assumed small and small objects are stable by finite colimits.
Then, a recursive argument proves that the domains of $\nabla^nf$ are all small.

\smallskip

\noindent \ref{lem:codiagonal:2}
We need to prove that $W^\bot = (W^\nabla)^\pitchfork$ and $^\bot(W^\bot) = {^\pitchfork}((W^\nabla)^\pitchfork)$.
A map $g$ is in $(W^\nabla)^\pitchfork$ iff the map $\pbh {s^n\pp w} g = \pbh {s^n} {\pbh w g}$
is a surjection for all $n \geq -1$.
Using \cref{lem:Whitehead}, this last condition is equivalent to $\pbh w g$ being invertible.
This proves that $(W^\nabla)^\pitchfork \subset W^\bot$.
Since the reverse inclusion is always true, we have in fact $(W^\nabla)^\pitchfork = W^\bot$.
Now, by \cref{lem:abs-iso}, if $\pbh w g$ is invertible, then so is $\pbh {s^n\pp w} g = \pbh {s^n} {\pbh w g}$.
This proves that $W^\bot \subset (W^\nabla)^\bot$.
Again, the reverse is always true, and we get $(W^\nabla)^\bot = W^\bot$.
Finally, we have $(W^\nabla)^\pitchfork = (W^\nabla)^\bot = W^\bot$.
A similar reasoning proves that ${^\pitchfork}((W^\nabla)^\pitchfork) = {^\pitchfork}(W^\bot)= {^\bot}(W^\bot)$.
\end{proof}

\begin{thm}[Gabriel-Ulmer construction]
\label{thm:!SOA1}
Let $\cC$ be a cocomplete category and $W$ a small set of small objects in $\Arr \cC$.
Then, the orthogonal system $(\cL, \cR)$ generated by $W$ is a factorisation system.
\end{thm}

\begin{proof}
Using \cref{lem:codiagonal}, we simply apply \cref{thm:SOA} to $W^\nabla$.
\end{proof}

\begin{rem}
A ``homotopical" variant of this argument is implicitly used in many cofibrantly generated simplicial model structures.
Let $f$ be a map in a such a model category.
If the map $s^n:S^n\to 1$ is replaced by the map $\sigma^n:\partial \Delta[n] \to \Delta[n]$, 
the map $\sigma^n\pp f$ is a homotopical version of the codiagonal of $\nabla^nf=s^n\pp f$.
The stability of cofibrations (trivial or not) under $\sigma^n\pp-$ is a way to stabilize them under codiagonal and to force the uniqueness of lift up to homotopy.
More generally, this property can be used to construct homotopy factorisation systems in the sense of~\cite[App. H.2]{Joyal:TQC}.
\end{rem}

\subsection{SOA for FS after Kelly}
\label{sec:!SOA2}

This section explains a second adaptation of the SOA in order to construct the FS generated by the class of maps $W$. 
This argument is essentially the one used by Kelly in~\cite[Sections 10 and 11]{Kelly}.
It consists in keeping the set of generators $W$ fixed but changing the construction $q$ for a new one that will force the uniqueness of lifts.
Let us explain the idea in the case of the orthogonal system generated by a single map $w$.
In the classical SOA, when constructing $\Quillenmap f$ the lifts of $w$ are added freely to $f$, independently of the potentially existing lifts.
Moreover, if several lifts exist, we need to collapse them into the same lift.
These two remarks suggest replacing the lifting diagram
\[
\begin{tikzcd}
\map{w}{f}\times \s w \ar[r]\ar[d] & \s f\ar[d, "f"]\\
\map{w}{f}\times \t w \ar[r] & \t f.
\end{tikzcd}
\]
by the lifting diagram
\[
\begin{tikzcd}
\map{\t w}{\s f}\times \t w\coprod_{\map{\t w}{\s f}\times \s w}\map{w}{f}\times \s w \ar[r]\ar[d] & \s f\ar[d, "f"]\\
\map{w}{f}\times \t w \ar[r] & \t f.
\end{tikzcd}
\]
The space $\map{w}{f}$ parameterizes all squares between $w$ and $f$, with or without lifts.
The space $\map{\t w}{\s f}$ parameterizes all lifts existing between $w$ and $f$ (together with the corresponding squares).
Altogether, this new lifting condition contains the old lifting condition (the right component of the pushout), but enforces existing lifts (left component) to be identified with the corresponding free lift (middle component). 
Now, the main remark is that the this new square is nothing but the counit map
\[
\pbh w f \pp w \to f
\]
From there, the construction will be exactly the same as in \cref{thm:SOA} but with the map $\pbh w f \pp w \to f$ instead of $\map w f _0 \times w \to f$.

\begin{rem}
In~\cite[Sections 10 and 11]{Kelly}, Kelly uses the sum $\map{\t w}{\s f}\times \t w\coprod \map{w}{f}\times \s w$ instead of the pushout $\map{\t w}{\s f}\times \t w\coprod_{\map{\t w}{\s f}\times \s w}\map{w}{f}\times \s w$. 
The two constructions turn out to be equivalent in the context of 1-categories, but the correct formula for $\infty$\=/categories does need the pushout. The difference concerns the treatment of the uniqueness of the lifts.
\end{rem}

\begin{rem}
The left class $\cA$ of a weak orthogonal system is stable only under discrete sums but the left class $\cL$ of an orthogonal system is stable under any colimit in $\Arr \cC$.
For this reason, the introduction of a set of maps between objects of an $\infty$\=/category of \cref{rem:set-v-space} is no longer necessary here.
The following variant of the SOA works similarly within a 1-category or an $\infty$\=/category.
\end{rem}

Related to the previous remark is the fact that the classical SOA uses a generating {\em set} of maps, but our variant allows the generating maps to be any small {\em diagram} of maps $W\to \cC^\to$. 
For such a diagram, the map
\[
W_K (f) := \int^{w\in W} \pbh w f \pp w
\]
is going to replace the map $\coprod_{w\in W} \map w f _0 \times w$ of \cref{thm:SOA}.

\medskip
The following lemma deals with lifting properties with respect to a diagram of arrows. 

\begin{lemma}[Functorial lifting]
\label{lem:caracortho}
Let $W\to \cC^\to$ be a diagram of arrows in $\cC$.
The following conditions are equivalent:
\begin{enumerate}[label={(\alph*)}, leftmargin=*]
\item \label{lem:caracortho-1}
	For all $w$ in $W$, the map $\pbh{w}{f}$ is invertible.
\item \label{lem:caracortho-2}
	For all $w$ in $W$, the identity map of $\pbh{w}{f}$ has a diagonal lift.
	\[
	\begin{tikzcd}
	\map{\t w} {\s f}	\ar[r, equal] \ar[d] 		& \map{\t w}{\s f}	\ar[d]\\
	\map w f		\ar[r, equal] \ar[ru, dashed]& \map w f.
	\end{tikzcd}
	\]

\item \label{lem:caracortho-3}
	For all $w$ in $W$, the map $\pbh{w}{f} \pp w \to f$ has a diagonal lift.
\item \label{lem:caracortho-4}
	The single map $\epsilon:\int^{w\in W} \pbh{w}{f}\pp w\to f$ has a diagonal lift.
\end{enumerate}
\end{lemma}
\begin{proof}
The equivalence \ref{lem:caracortho-1} $\iff$ \ref{lem:caracortho-2} is straightforward. 
The equivalence \ref{lem:caracortho-2} $\iff$ \ref{lem:caracortho-3} is direct by adjunction.
Notice that the lift in \ref{lem:caracortho-2} is not asked to be unique (but by uniqueness of inverses, it will be automatically) and that we use this fact in the equivalence \ref{lem:caracortho-2} $\iff$ \ref{lem:caracortho-3}.

If all maps $\pbh{w}{f}$ are invertible, then by \cref{lem:abs-iso} so are the maps $\pbh{w}{f}\pp w$ and the coend $\int^w \pbh{w}{f}\pp w$. This proves \ref{lem:caracortho-1} $\Ra$ \ref{lem:caracortho-4}.
Reciprocally, we simply use the canonical maps $\pbh w f \pp w \to \int^w \pbh w f \pp w$ to show that \ref{lem:caracortho-4} $\Ra$ \ref{lem:caracortho-3}.
\end{proof}

In details, the commutative square $\epsilon$ of condition \ref{lem:caracortho-4} is
\begin{equation}
\tag{$!\boxslash$}
\label{eqn:liftcoend!}
\begin{tikzcd}
\int^w\Big( \map{\t w}{\s f}\times \t w\coprod_{\map{\t w}{\s f}\times \s w}\map{w}{f}\times \s w \Big) \ar[r]\ar[d,"W_K (f)"'] & \s f\ar[d, "f"]\\
\int^w \map{w}{f}\times \t w \ar[r] & \t f.
\end{tikzcd}
\end{equation}
Kelly's construction then proceeds as in \cref{thm:SOA}. 
The map $\Kellymap f:\Kelly f\to \t f$ is defined as the cogap map of the square (\ref{eqn:liftcoend!})
\begin{equation}
\tag{$k$\=/construction}
\label{eqn:kconstruction}
\begin{tikzcd}
\int^w\Big( \map{\t w}{\s f}\times \t w\coprod_{\map{\t w}{\s f}\times \s w}\map{w}{f}\times \s w \Big) \ar[r]\ar[d,"W_K (f)"'] & \s f\ar[d, "u(f)"'] \ar[ddr, bend left, "f"]&\\
\int^w \map{w}{f}\times \t w \ar[r, "\ell"] \ar[rrd, bend right=20] & \Kelly f \ar[rd, "\Kellymap f"'] \pomark &\\
&& \t f.
\end{tikzcd}
\end{equation}
Let $\t_W f$ be the identity map of the codomain of $W_K (f)$, that is, of $\int^w \map{w}{f}\times \t w$. 
In the presentation of~\cite[Prop. 9.2]{Kelly}, the functor $f\mapsto \Kellymap f$ is defined via the pushout in $\Arr \cC$
\[
\begin{tikzcd}
W_K (f) \ar[r] \ar[d] & f \ar[d]
\\
\t_W f \ar[r] & \Kellymap f \pomark .
\end{tikzcd}
\]
As in \cref{thm:SOA}, it is convenient to look at this diagram in terms of the following lifting
\begin{equation}
\tag{$k$\=/lift}\label{eqn:k-lift}
\begin{tikzcd}
\int^w\Big( \map{\t w}{\s f}\times \t w\coprod_{\map{\t w}{\s f}\times \s w}\map{w}{f}\times \s w \Big) \ar[r]\ar[d] & \s f\ar[d, "f"] \ar[r,"u(f)"] & \Kelly f\ar[d, "\Kellymap f"]\\
\int^w \map{w}{f}\times \t w  \ar[r] \ar[rru, dashed, "\ell"'] & \t f \ar[r, equal] & \t f.
\end{tikzcd}
\end{equation}
The map $\Kellymap f$ can be understood as the best approximation of $f$ on the right that admits a lift from the map $\int^w\pbh{w}{f} \pp w$.

As in \cref{thm:SOA}, we have a factorisation of $f$ as $\Kellymap f u(f):\s f\to \Kelly f\to \t f$.
Because $\cL$ is absorbing for $\square$ and stable by colimits in $\Arr \cC$ (\cref{lem:prop-ufs}), 
the map $\int^w\map{w}{f} \pp w$ is in $\cL$ and so is its cobase change $u(f):X\to \Kelly f$.
The map $\Kellymap f:\Kelly f\to \t f$ need not be in $\cR$ but it will be after a transfinite iteration.

\begin{thm}[{Kelly construction \cite[Theorem 11.5]{Kelly}}]
\label{thm:!SOA2}
Let $\cC$ be a cocomplete category and $W\to \Arr \cC$ be a small diagram of arrows with small domains and codomains.
Then, the transfinite iteration of the $k$\=/construction converges to a factorization associated to the orthogonal system $(\cL_W, \cR_W)$ generated by $W$.
\end{thm}

\begin{proof}
The natural transformation $u:f\to \Kellymap f$ defines a transfinite sequence
\[
f\ntto{u(f)} \Kellymap f\ntto{u(\Kellymap f)} \Kellymapn2 f\ntto{u(\Kellymapn2 f)} \dots
\]
We are going to prove that this sequence converges, that is, eventually becomes a sequence of invertible maps.
We assumed that all arrows $w$ were small (which is to say that both sources and targets are small) and that $W$ was a small category. 
Thus, we can fix two regular ordinals $\kappa'<\kappa$ majoring the size of all $w$.
The map $u^\kappa(f)$ is in $\cL_W$ because $\cL_W$ is stable by transfinite compositions.
The result will be proven if we prove that the map $\Kellymapn \kappa f$ is in $\cR_W=W^\bot$, in which case $\Kellymapn \kappa f = \Kellymapn {\kappa+1} f$.

\medskip
For any $w$ in $W$, we want to prove that $\pbh{w}{\Kellymapn \kappa f}$ is invertible.
We define $\Kellyn \kappa f:=\s \Kellymapn \kappa f$.
This is equivalent to proving that the identity map of $\pbh{w}{\Kellymapn \kappa f}$ has a lift
\[
\begin{tikzcd}
\map{\t w}{\Kellyn \kappa f}\ar[r, equal]\ar[d] & \map{\t w}{\Kellyn \kappa f} \ar[d]\\
\map{w}{\Kellymapn \kappa f}\ar[r, equal]\ar[ru, dashed] & \map{w}{\Kellymapn \kappa f}.
\end{tikzcd}
\]
Because $w$ is $\kappa$\=/small and $\kappa$ is $\kappa$\=/filtered we have
\[
\pbh{w}{\Kellymapn \kappa f} = \colimit {\lambda<\kappa} \pbh{w}{\Kellymapn \kappa f}.
\]
Also, for a fixed $w$, the map $\ell$ of the square (\ref{eqn:k-lift}) provides a diagonal lift for the square $\pbh w f \pp w \to \Kellymap f$.
By adjunction, it also provides a diagonal lift $\delta(f)$ for the square $\pbh w f \to \pbh w {\Kellymap f}$.
\[
\begin{tikzcd}
\map {\t w}{\s f}	\ar[r]\ar[d]					& \map{\t w}{\Kelly f} \ar[d]\\
\map w f		\ar[r]\ar[ru, dashed, "\delta(f)"]	& \map{w}{\Kellymap f}.
\end{tikzcd}
\]
Finally, we have a diagram
\[
\begin{tikzcd}
\map{\t w}{\Kellyn \kappa f}\ar[r]\ar[d] & \map{\t w}{\Kellyn {\kappa+1} f}\ar[r]\ar[d] &\dots \ar[r]&  \map{\t w}{\Kellyn \kappa f}\ar[r, equal]\ar[d] & \map{\t w}{\Kellyn \kappa f} \ar[d]\\
\map{w}{\Kellymapn \kappa f}\ar[r]\ar[ru, "\delta^\lambda"] & \map{w}{\Kellymapn {\kappa+1} f}\ar[r] \ar[ru, "\delta^{\lambda+1}"]& \dots \ar[r]&  \map{w}{\Kellymapn \kappa f}\ar[r, equal]\ar[ru, dashed, "\delta"] & \map{w}{\Kellymapn \kappa f}
\end{tikzcd}
\]
where the maps $\delta^\lambda := \delta(\Kellymapn \kappa f)$ provide the desired lift $\delta$ at the limit.
\end{proof}

\begin{rem}[Comparison with Kelly]
\label{rem:Kelly}
Because of its focus on reflective subcategories, Kelly only underlines in \cite[Section 10]{Kelly} the adjunction 
\[
\begin{tikzcd}
\s (-\pp w):\Arr \cS \ar[r, shift left] \ar[from=r, shift left]& \cC : \pbh w-
\end{tikzcd}
\]
where $\cC$ is identified with the full subcategory of $\Arr\cC$ spanned by maps $X\to 1$.
But its construction depends in fact on the full adjunction 
\[
\begin{tikzcd}
-\pp w :\Arr \cS \ar[r, shift left] \ar[from=r, shift left]& \Arr \cC : \pbh w-.
\end{tikzcd}
\]
More precisely, the relevant adjunction for a diagram $W\to \Arr \cC$ of arrows is the realization-nerve adjunction
\[
\begin{tikzcd}
&W \ar[rd]\ar[ld]\\
\fun {W\op} {\Arr \cS} \ar[rr, shift left,"\int^w(-\pp w)"] \ar[from=rr, shift left,"w\mapsto \pbh w-"]&& \Arr \cC \ar[loop right, "W_K"]
\end{tikzcd}
\]
built from the diagram $W\to \Arr \cC$ using the enrichment $\pbh--$ of $\Arr \cC$ over $(\Arr \cS, \pp)$.
The comonad $W_K (f)=\int^w \pbh w f \pp w$ is then the (enriched) density comonad of the diagram $W\to \Arr \cC$.
\end{rem}

\begin{rem}[Garner construction]
\label{rem:garner2}
In keeping with \cref{rem:garner1}, we could consider the Garner construction~\cite{Garner:SOA} and replace the colimit $\rho(f)$ of the sequence $f\to \Kellymap f\to \Kellymapn2 f \to \dots$ by the construction of the free monad $\gamma$ on the pointed endofunctor $\pi:1 \to k$.
It can be proven that the two constructions coincide because the colimit of the sequence of $\Kellymapn \kappa f$ is an idempotent functor.
\end{rem}

\begin{cor}
\label{cor:+-fix=R}
Under the hypotheses of \cref{thm:!SOA2}, a map $f$ is in $W^\bot$ if and only if the canonical map $f\to \Kellymap f$ is invertible.
\end{cor}
\begin{proof}
If $f$ is in $W^\bot$, the maps $\pbh w f$ are invertible and so is the map $W_K (f)=\int^w \pbh w f \pp w$. 
Thus $f\simeq \Kellymap f$.
Reciprocally, if $f\simeq \Kellymap f$, then by \cref{lem:caracortho} and \ref{eqn:k-lift}, $f$ is in $W^\bot$.
\end{proof}

\begin{rem}[Cartesian closed enriched version]
\label{rem:enrichedSOA}
The proof of \cref{thm:!SOA2} involves the external pushout-product structure on $\Arr \cC$.
In the case where $\cC$ is cartesian closed, the enriched pushout-product and pullback-homs can also be used  (\ref{sec:enrichedFS}).
The proof is essentially the same and the resulting factorization system is the factorization of the enriched orthogonal system $({^\intperp}(W^\intperp),W^\intperp)$.
A difference is that $W$ can now be a category enriched over $\cC$ (with the appropriate modification of the coend).
Another difference is in the notion of smallness.
The external smallness condition says that an object $X$ is small if, for any
$\kappa$\=/filtered colimit $Y_i$, there is an invertible map of external homs
$\colim \map X {Y_i} \simeq \map X {\colim Y_i}$.
The enriched smallness condition says that there is an invertible map of
internal homs
$\colim \intmap X {Y_i} \simeq \intmap X {\colim Y_i}$.
The two conditions are not related unless the terminal object 1 of $\cC$ is externally $\kappa$\=/small, in which case the enriched condition implies the
external one.
Note that 1 is always internally $\kappa$\=/small for all $\kappa$.
The internal smallness condition can be understood as a condition of smallness relative to the size of 1.
We shall not develop this approach in full detail, but we shall make remarks in the examples about this variant of the construction.

\end{rem}

\begin{rem}[Locally cartesian closed internal version]
\label{rem:LCCinternal+}
More generally, when $\cC$ is locally cartesian closed with a terminal object (\eg a topos, an $n$\=/topos, a modal topos), 
\cref{thm:!SOA2} admits yet another internal version where 
$W$ is allowed to be not only a category enriched over $\cC$, but an internal category $\WW$ in $\cC$ (\eg a stack).
We shall not develop here the theory of internal categories in a locally cartesian closed \oo category.
We shall only make some remarks without proofs about this other variant.
\Cref{table:SOA-variants} summarizes the possible variants for the \cref{thm:!SOA2}.

\end{rem}

\begin{table}[htbp]
\begin{center}	
\caption{Variants for Kelly's SOA}
\label{table:SOA-variants}
\medskip
\renewcommand{\arraystretch}{2}
\begin{tabularx}{.9\textwidth}{
|>{\hsize=1.1\hsize\linewidth=\hsize\centering\arraybackslash}X
|>{\hsize=.9\hsize\linewidth=\hsize\centering\arraybackslash}X
|}
\hline
$\cC$ is any category
    & $W \to \Arr \cC$ any diagram
\\
\hline
$\cC$ is a cartesian closed category
    & $W$ can be enriched over $\cC$
\\
\hline
$\cC$ is a locally cartesian closed category with 1
    & $W$ can be an internal category in $\cC$
\\
\hline
$\cC$ is a topos
    & $W$ can be a stack over $\cC$
\\
\hline
\end{tabularx}
\end{center}
\end{table}

\begin{rem}[Analogy with linear algebra]
\label{rem:analogy-linear-alg}
Let $E$ be a vector space equipped with a positive definite scalar product $\pbh--$.
Let $u$ be a vector in $E$, and $u^\bot$ the subspace of $E$ of vectors orthogonal to $u$.
The projection $P:E\to u^\bot$ is given by
$$
P(x) = x-{\pbh u x \over \pbh u u}u.
$$
Let $K$ be the operator $K(x) = x-\pbh u x u$, then
$$
K^n(x) = x - u{\pbh u x \over \pbh u u} \big(1 - (1-\pbh u u)^n\big)
$$
which converges to $P(x)$ when $\pbh u u \leq 1$.
The analogy of $K$ with Kelly's construction should be clear enough (a pushout becomes a subtraction).
Somehow, the role of 1 in the convergence is played by the inaccessible cardinal bounding the size of small objects.
\end{rem}

\subsection{Comparison with the plus-construction}
\label{sec:+construction}

\subsubsection{The plus-construction}

The purpose of this section is to connect the $k$\=/construction of \cref{thm:!SOA2} to the plus-construction involved in sheafification~\cite{SGA41,MLM,Johnstone:topos}.
Recall the enrichment of $\Arr \cC$ over the cartesian closed category $(\Arr \cS,\times)$ (see \ref{sec:cartesian-structure-on-arrows})
\[
\bracemap w f \quad:\quad \map w f \stto \map {\t w}{\t f} .
\]
\begin{defi}[The plus-construction]
\label{defi:+construction}
Let $f$ be a map in $\cC$ and $W\to \Arr \cC$ a diagram of maps.
We define the {\em plus-construction} of the map $f$ as the map
\begin{align*}
\Plusmap f &:=\int^w \bracemap w f \times \t w 
\\
&= \int^w \map w f \times \t w \stto \int^w \map {\t w}{\t f} \times \t w
\\
&= \colimit {W\commaindex f} \t w \stto \colimit {W \commaindex {\t f}} \t w
\end{align*}
where $W \commaindex {\t f}$ is the comma category of the target functor $\t:W \to \cC$ over the object $\t f$.
We shall denote by $\Plus f$ the domain of $\Plusmap f$.
The second iteration of the plus-construction shall be denoted by $\Plusplusmap f$ and $\Plusplus f$.
\end{defi}

Let $\cC = \P C$, for a small category $C$, and let $W\subset \Arr \cC$ be the full subcategory of covering sieves $R\to X$ of some Grothendieck topology on $C$. 
For $X$ in $C$, we denote by $W(X)$ be the category of covering sieves of $X$
(\ie the fiber of the target functor $\t:W\to C$ at $X$).
Recall that for a presheaf $F$, the presheaf $\Plusoriginal F$ is defined by 
\begin{align*}
C\op &\tto \cS \\
X &\mto \Plusoriginal F(X) = \colim_{R\to X\in W(X)\op}\map R F.
\end{align*}

\begin{lemma}
\label{lem:+=+}
The presheaf $\Plusoriginal F$ is isomorphic to the domain $\Plus {F\to 1}$ of the plus-construction $\Plusmap {(F\to 1)}$.
\end{lemma}
\begin{proof}
Let $H$ be the codomain of $\Plusmap {(F\to 1)}$.
This is the presheaf defined by 
\[
H(X) = \map X {\colimit {W\commaindex F} \t w} = \colimit {W\commaindex F} \map X {\t w}.
\]
We need to construct an isomorphism $\Plusoriginal F(X) = H(X)$, natural in $X$.
We consider the category $W(X)\comma F$ which is the category of elements of the functor 
\begin{align*}
W\commaindex F &\tto \cS   \\
w &\mto \map X {\t w}.
\end{align*}
Recall that the external groupoid of a category $C$ is the colimit $|C|:=\colim_C1$ in $\cS$.
Then, we have that $\colimit {W\commaindex F} \map X {\t w} = |X\comma W\comma F|$.

Let $W(X,F)$ be the category of elements of the functor 
\begin{align*}
W(X)\op &\tto \cS   \\
R\to X &\mto \map R F.
\end{align*}
Its objects are diagrams
\[
\begin{tikzcd}[sep=small]
            & R \ar[r]\ar[d,"w"']  & F\ar[d] \\
X \ar[r]    & Y \ar[r]        & 1
\end{tikzcd}
\]
where $w:R\to Y$ is a covering sieve and whose morphisms are the natural transformations which induce the identity on $X$ and $F$.
We have that $\Plusoriginal F(X) = |W(X,F)|$.

There exists an adjunction $G:W(X)\comma F \rightleftarrows W(X,F):D$ such that
\[
G\left(
{\begin{tikzcd}[sep=small]
R \ar[r]\ar[d]  & F\ar[d] \\
X \ar[r]            & 1
\end{tikzcd}}
\right)
\quad = \quad
{\begin{tikzcd}[sep=small]
                & R \ar[r]\ar[d]  & F\ar[d] \\
X \ar[r,equal]  & X \ar[r]           & 1
\end{tikzcd}}
\]
\[
D\left(
{\begin{tikzcd}[sep=small]
            & R \ar[r]\ar[d]  & F\ar[d] \\
X \ar[r]    & Y \ar[r]           & 1
\end{tikzcd}}
\right)
\quad = \quad
{\begin{tikzcd}[sep=small]
R\times_YX \ar[r]\ar[d]  & F\ar[d] \\
X \ar[r]           & 1
\end{tikzcd}}
\]
This adjunction induces an equivalence of external groupoids $|W(X)\comma F| \simeq |W(X,F)|$ which is natural in $X$ and $F$.
This finishes to show that $\Plusoriginal F(X) = H(X)$.
\end{proof}

\begin{rem}
\label{rem:+=+}
The proof relies implicitly on the stability by base change of the set of covering sieves (to construct the right adjoint $D$, and for its naturality in $X$)
and on the fact that $W$ contains the identity of generators, that is, representable functors (to have the natural map $F\to \Plusoriginal F$).
The facts that the covering sieves are monomorphisms and local is irrelevant.
The lemma would hold with any small full subcategory $W\to \Arr {\P C}$ such that
the codomain functor $\t:W\to \Arr {\P C} \to \P C$ has values in $C$,
$W$ contains all the identity maps of $C$,
and $W$ is stable by base change along maps of $C$ (\ie $\t:W\to C$ is a fibration in categories).
This motivates the subsequent definitions of a pre-modulator (\cref{defi:pre-modulator}) and a modulator (\cref{defi:modulator}).
\end{rem}

\begin{rem}
\label{rem:values+construction-in-presheaves}
Given a diagram $W\to \Arr {\P C}$ whose codomains are in $C$ and a map $f$ in
$\P C$, let $C\comma \Plus f$ be the category of elements of the presheaf $\Plus f=\s (\Plusmap f)$.
The associated fibration in groupoids $C\comma {\Plus f}\to C$ is the right part of the (cofinal,fibration) factorization of the functor $\t:W\comma f\to C$ sending a map $w\to f$ to $\t w$.
\[
\begin{tikzcd}
&C\comma \Plus f\ar[rd, "\text{fibration}"]\\
W\comma f \ar[rr,"\t"] \ar[ru,"\text{cofinal}"] && C
\end{tikzcd}
\]

The value of the presheaf $\Plus f$ at some $X$ in $C$ is the external groupoid of the category $W(X,f)$ whose objects are diagrams
\[
\begin{tikzcd}[sep=small]
            & \s w \ar[r]\ar[d,"w"'] & \s f\ar[d,"f"] \\
X \ar[r]    & \t w \ar[r]            & \t f.
\end{tikzcd}
\]
When $W$ is stable by base change, the proof of \cref{lem:+=+} shows that
\[
\Plus f(X) = \colim_{W(X)\op}\map w f .
\]
\end{rem}

\begin{rem}
\label{rem:arbitrary+}
The core of the plus-construction is really the domain of $\Plusmap f$, that is, the object $\Plus f$.
Because we are working with maps and not only objects, we made the plus-construction into a map.
But there is something arbitrary in the choice of its codomain.
We have chosen the codomain of map $\Plusmap f$ to be 
$\colimit {W \commaindex {\t f}} \t w$ but it would have also be meaningful to
chose simply $\t f$ (in our applications the two choices coincide).
Our choice is motivated by the remark that the whole map $\Plusmap f$ can be expressed as $\int^w \bracemap w f \times w$ 
but we do not have a real conceptual reason to justify this choice.
The conceptual relationship between Kelly's construction and the plus-construction is not fully clear to us.
It could be rooted in an articulation between the cartesian and pushout-product monoidal structures on $\Arr \cC$.
\end{rem}

\medskip
We need a few more notations.
The map $\Minusmap f:=\int^{w\in W} \pbh w f \pp w$ is the cocartesian gap map of the square
\[
\begin{tikzcd}
\int^w \map{\t w}{\s f}\times \s w  
\ar[dd]
\ar[rrr]
\ar[rrrdd, "c(f)=\int^w \pbh w f \times w" description]
&&& \int^w \map w f \times \s w \ar[dd, "b(f)=\int^w \map w f \times w"]
\\
\\
\int^w \map{\t w}{\s f}\times \t w
\ar[rrr,"a(f)=\int^w \pbh w f \times \t w"'] 
&&& \int^w \map w f \times \t w
\end{tikzcd}
\]
This square can also be written in terms of colimits 
\[
\begin{tikzcd}
\colimit{W \commaindex {\s f}} \s w
\ar[dd]
\ar[rr]
\ar[rrdd, "c(f)" description]
&& \colimit{W\commaindex f} \s w
\ar[dd, "b(f)"]
\\
\\
\colimit{W \commaindex {\s f}} \t w
\ar[rr,"a(f)"] 
&& \colimit{W\commaindex f}\t w
\end{tikzcd}
\]
where, for $X$ an object of $\cC$, the category $W \comma X$ is the comma category over the identity map of $X$.

\medskip
The following diagram is going to be central in the next result.
\[
\begin{tikzcd}
\colimit{W \commaindex {\s f}} \s w 
\ar[rd, "\alpha"]
\ar[rr, "\beta"]
\ar[dddrrr, "c(f)=\int^w \pbh w f \times w" description, near start, bend right=60] 
&&\colimit{W \commaindex f} \s w 
\ar[dddr, "b(f)=\int^w \map w f \times w" description, bend right=50]
\ar[rd, "\alpha'"]
\ar[rrd, "\gamma", bend left=20]
\\
&\colimit{W \commaindex {\s f}}\t w 
\ar[ddrr, "a(f)=\int^w \pbh w f \times \t w" description, bend right=40] 
\ar[rr, "\beta'", near start, crossing over] 
&& \colimit{W \commaindex {\s f}} \t w 
   \underset{\colimit{W \commaindex {\s f}} \s w}{\coprod}
   \colimit{W \commaindex f} \s w
\ar[dd,"\Minusmap f=\int^w \pbh w f \pp w" description]
\ar[r, "\delta"']
\ar[from=ulll, "\lrcorner" description, very near end, phantom, shift left=2]
& \s f 
\ar[dd,"u(f)"  description]
\ar[dddd,"f",bend left=50]
\\
\\
&&& \Plus f=\colimit{W \commaindex f}\t w 
\ar[dd, "\Plusmap f = \int^w \bracemap w f \times \t w" description]
\ar[r, "\varepsilon"]
& \Kelly f
\ar[dd,"\Kellymap f" description]
\ar[from=uul, "\lrcorner" description, near end, phantom]
\\
{}
\\
&&&\colimit{W \commaindex {\t f}} \t w 
\ar[r, "\zeta"]
&\t f 
\end{tikzcd}
\]

\subsubsection{A coincidence condition for $+$ and $k$}

We are going give conditions for the construction $\Kellymap f$ of \cref{thm:!SOA2} to coincide with $\Plusmap f$.
We shall do it by proving that, under sufficient conditions, all the Greek-lettered maps of the diagram are invertible.
So far the conditions on $\cC$ has only been that it is a cocomplete category and the assumptions of size have been on $W$.
The proof of the coincidence of the $k$\=/ and plus-constructions is going to need the stronger assumption that $\cC$ is presentable. 
Since any object is small in a presentable category, this will remove the
smallness assumption on the domain and codomain of diagrams $W\to \Arr \cC$ when applying \cref{thm:!SOA2}. 
However, $W$ will still be a small indexing category.

\begin{defi}[Pre-modulator]
\label{defi:pre-modulator}
Let $\cC$ be a presentable category with a fixed generating small subcategory $C\subset \cC$.
We shall say that a diagram $W\to \Arr \cC$ is a {\em pre-modulator}
if it satisfies the following conditions:
\begin{enumerate}[label=(\roman*), leftmargin=*]
\item \label{enum:modulator:1} $W$ is a small diagram,
\item \label{enum:modulator:2} $W\to \Arr \cC$ is fully faithful,
\item \label{enum:modulator:3} the codomains of the maps in $W$ are all in $C$,
\item \label{enum:modulator:4} the inclusion $C\subto \Arr \cC$ sending a
  generator to its identity map factors through $W\subto \Arr \cC$.
\end{enumerate}
We shall sometimes leave implicit the choice of the generating category $C$ when considering pre-modulators.

\end{defi}

\begin{rem}
\label{rem:def-admissibility}
We consider the comma category $\cC \comma C$. 
The target functor $\t:\cC\comma C\to C$ admits a fully faithful right adjoint $C\subto \cC\comma C$ sending a generator to its identity map.
The adjunction $\t:\cC\comma C\rightleftarrows C:id$ is then a reflective localization.
A pre-modulator can equivalently be defined as a small full subcategory $W\subset \cC\comma C$ such that the previous localization restricts to a localization $\t:W\rightleftarrows C:id$.
\end{rem}

\begin{thm}[The plus-construction]
\label{thm:+=k}
\label{thm:+construction}
Let $\cC$ be a presentable category.
If $W\to \Arr \cC$ is a pre-modulator, we have the simplifications 
\[
u(f) = a(f) = b(f) = c(f) = \Minusmap f
\eqnand
\Kellymap f = \Plusmap f.
\]
In particular, the factorization of $f$ generated by such a $W\to \Arr \cC$ can be produced by a transfinite iteration of the plus-construction.
Moreover, a map $f$ is in $\cR$ if and only if $f=\Plusmap f$.
\end{thm}

\begin{proof}
We first prove that $\alpha$ is the identity map of the object $F:=\s f$.
Let $C$ be a generating small category $C\subset \cC$ relative to which $W$ is a pre-modulator.
By hypothesis \ref{enum:modulator:2} and \ref{enum:modulator:4} we get a fully faithful functor
\begin{align*}
C \comma F &\tto W \commaindex F\\
c\to F & \mto (1_c, c\to F).
\end{align*}
Then, using hypothesis \ref{enum:modulator:3} and \cref{rem:def-admissibility}, this functor has a left adjoint sending $(w,\t w\to F)$ to $\t w\to F$.
Recall that a right adjoint functor is always cofinal.
We deduce that 
\[
\colimit {W \commaindex {\s f}} w \ =\ \colimit {C\commaindex {\s f}} 1_c \ =\ 1_{\s f}.
\]
The same reasoning with $F=\t f$ gives
\[
\colimit {W \commaindex {\t f}} \t w \ =\ \t f.
\]
This proves that $\zeta$ is invertible.

Next, we prove that $\gamma$ is also the identity map of the object $\s f$.
We consider the category $C \comma W\comma f$ whose objects are triplets 
$(c,w,1_c\to w\to f)$ where $c$ is in $C$, $w$ in $W$, $1_c\to w$ in $W$ and
$w\to f$ in $\Arr \cC$, and whose morphisms $(c,w,1_c\to w\to f) \to
(c',w',1_{c'}\to w'\to f)$ are pairs $(c\to c',w\to w')$ such that the obvious diagram commutes.
The functor $h:C \comma W\comma f \stto C \comma {\s f}$ sending $1_c\to w \to f$ to $c\to \s f$
has a left adjoint given sending $c\to \s f$ to $1_c \xto{1_{1_c}} 1_c\to f$ (hypothesis \ref{enum:modulator:2} and \ref{enum:modulator:4} are used implicitly to show that the counit of this adjunction is in $W$).
In particular, $h$ is cofinal.
The result follows from the following identifications:
\begin{align*}
\colimit {W\commaindex f}\ \s w &= \colimit {W\commaindex f}\colimit {C\commaindex {\s w}} c
\\
&= \colimit {C\commaindex W\commaindex f} c
\\
&= \colimit {C\commaindex {\s f}} c & \text{by cofinality}
\\
&= \s f.
\end{align*}
We have always $\gamma\beta=1_{\s f}$.
Since $\gamma$ is invertible, then so is $\beta$. 
Then, by pushout, so are $\alpha'$ and $\beta'$.
The map $\delta$ is invertible because $\gamma$ and $\alpha'$ are.
Then so is $\varepsilon$ by pushout.
This finishes the proof that all the Greeek-lettered maps are invertible as well as that of all the identities of the theorem.

Finally, the last assertion is a consequence of \cref{cor:+-fix=R}.
\end{proof}

\begin{rem}
The virtue of the coincidence of the $+$ and the $k$\=/constructions is that the colimit formula defining the former is more suited to check exactness conditions such as stability by base change and left-exactness (see \cref{thm:+modality} and \cref{thm:lex+construction}).
\end{rem}

\begin{rem}
\label{rem:weakening-admissibility}
A careful reading of the proof of \cref{thm:+construction} shows that it depends on less than the hypothesis that $W\to \Arr \cC$ be fully faithful.
The proof that $C \comma F \to W \commaindex F$ is cofinal needs only that, for any $w$ in $W$, the canonical map $w\to 1_{\t w}$ is in $W$.
And the proof that $h:C \comma W\comma f \to C \comma {\s f}$ is cofinal needs only that $\relmap W {1_c} w = \relmap \cC c {\s w}$.
\end{rem}

\begin{rem}[Cartesian closed enriched version]
\label{rem:enriched+}
As with \cref{thm:!SOA2} in \cref{rem:enrichedSOA}, \cref{thm:+construction} admits an enriched version when $\cC$ is cartesian closed.
A first difference is in the fact that $\Arr \cC$ and $W$ are now categories enriched over $\cC$.
The second difference is in the definition of admissibility.
From the enriched point of view, the terminal object 1 of $\cC$ is always a generator (even when $\cC$ is not presentable).
Thus, hypothesis \ref{enum:modulator:3} and \ref{enum:modulator:4} can be replaced by asking that $W$ is a diagram of objects of $\cC$ containing the terminal object.
We shall say that a fully faithful small diagram $W\subto \cC$ is {\em an enriched pre-modulator} if it contains the terminal object and is made of internally small objects.
From there, the plus-construction is defined using the enrichment $\intbracemap--$ and everything proceeds similarly.
\end{rem}

\begin{rem}[Locally cartesian closed internal version]
\label{rem:internal+}

More generally, as in \cref{rem:LCCinternal+}, when $\cC$ is a presentable locally cartesian closed category, \cref{thm:+construction} admits an internal version where 
$W$ is allowed to be not only a category enriched over $\cC$, but an internal category in $\cC$ (or even a stack, if $\cC$ is a topos).
Recall that the {\em universe} of $\cC$ is the ``internal'' category $\UU$ corresponding to the codomain fibration $\t:\Arr \cC \to \cC$.%
\footnote{For size reasons, the universe cannot be, strictly speaking, an internal category in $\cC$, hence the quotation marks.}
We define an {\em internal pre-modulator} $\WW$ as a small internal category together with a fully faithful diagram $\WW\subto \UU$ to the universe of $\cC$ containing the terminal object (\ie factoring the canonical map $1\to \UU$) and made of internally small objects.
When expressed in terms of fibrations of categories over some generators of
$\cC$, an internal pre-modulator is an external modulator with the extra property of being stable by base change. 
We shall see in \cref{defi:modulator} that this is what defines our notion of modulator.
The factorization system build from such a pre-modulator is always stable by base change, that is, a modality.
\end{rem}

\begin{prop}[Pre-modulator envelope]
\label{prop:modulator-completion}
Any diagram $W\to \Arr \cC$ can be completed into a pre-modulator generating the same factorization. 
\end{prop}
\begin{proof}
Let $W\to \Arr \cC$ be a small diagram.
Because $W$ is small, we can find a small category $C$ of generators of $\cC$ containing the codomains of all maps in $W$.
We consider $W'$ the full subcategory generated by the image of $W$ and the identity maps of $C$.
Then $W'$ is a pre-modulator.
We need to see that $W^\perp = (W')^\perp$.
The definition of $(W')^\perp$ depends only on the objects of $W'$ and not the morphisms between them.
The objects of $W'$ are those of $W$ and the identity maps of $C$.
We deduce that $(W')^\perp = W^\perp \cap C^\perp$.
Since $C^\perp$ is the whole of $\Arr \cC$, this proves the equality.
\end{proof}

\begin{rem}
The $k$\=/constructions associated to $W$ and $W'$ are a priori different, but they converge to the same object.
\end{rem}

\subsubsection{An idempotency condition for $+$}

We finish with a sufficient condition under which the plus-construction is
idempotent, and consequently gives, directly, the expected factorization.

\begin{defi}
\label{defi:tower}
\begin{enumerate}

\item For $\kappa$ a regular cardinal, we shall say that a pre-modulator $W\to \Arr \cC$ is {\em $\kappa$\=/closed}, if, 
it has values in $\kappa$\=/small objects of $\Arr \cC$ and if it is closed under $\kappa$\=/small colimits in $\Arr \cC$.
In particular, for any map $f$, the slice category $W\comma f$ is $\kappa$\=/filtered.

\item We shall say that a pre-modulator $W\to \Arr \cC$ is {\em stable by towers}, if, for any two maps $v$ and $w$ in $W$ and any map $\s v\to \t w$ in $\cC$, the map $u$ of the following diagram is in $W$:
\[
\begin{tikzcd}
\s \zeta \ar[d,dashed, "w'"]\ar[dd, bend right=60, "u"'] \ar[r,dashed] \pbmark &\s w \ar[d,"w"]\\
\s v \ar[d, "v"] \ar[r]  &\t w \\
\t v.
\end{tikzcd}
\]
\end{enumerate}

\end{defi}

\begin{rem}
\label{rem:tower}
This definition of stability by towers is almost a stability by composition, 
but it does not assume that the map $w':\s \zeta \to \s v$ is in $W$.
This would imply that all the domains $\s v$ are in $C$ and that $W\subset \Arr \cC$ is actually a subcategory of $\Arr C$ (which is too strong in practice).
When it is the case that $W\subset \Arr C$, the plus-construction can be proven to be idempotent under the assumption that $W$ is stable by composition (see \cref{ex:SOA:W-composition}).
\end{rem}

\begin{prop}[Idempotent plus-construction]
\label{prop:idempotent+}
If $W$ is a pre-modulator which is $\kappa$\=/closed and stable by towers, then the plus-construction is idempotent.
\end{prop}

\begin{proof}
It is sufficient to prove that $\Plusplus f = \Plus f$.
We have $\Plusplus f = \Plus {\Plusmap f} \colim_{W\comma \Plusmap f} \t v$ where the colimit is indexed by the category of diagrams
\[
\begin{tikzcd}
\s v \ar[d] \ar[r]  & \colimit {W\commaindex f} \t w \ar[d]\\
\t v \ar[r]         & \t f
\end{tikzcd}
\]
Since $W\comma f$ is $\kappa$\=/filtered and $\s v$ is $\kappa$\=/small, any map $\s v \to \colim_{W\commaindex f} \t w$ factors through some $\t w$ and we can replace the indexing category $W\comma \Plusmap f$ by the category $V(f)$ of diagrams
\begin{equation}
\tag{$\xi$}
\label{diag:bizarre}
\begin{tikzcd}
&\s w \ar[d] \ar[r] & \s f\ar[dd]\\
\s v \ar[d] \ar[r]  &\t w \ar[rd]& \\
\t v \ar[rr]        && \t f
\end{tikzcd}
\end{equation}
with the obvious notion of morphisms.

Using the stability by towers of $W$ we can construct the following map $u$
\[
\begin{tikzcd}
\s w'\ar[dd,bend right=60,dashed, "u"'] \ar[d,dashed, "w'"'] \ar[r,dashed] \pbmark &\s w \ar[d] \ar[r] & \s f\ar[dd]\\
\s v \ar[d] \ar[r]  &\t w \ar[rd]& \\
\t v \ar[rr] && \t f.
\end{tikzcd}
\]
This provides a functor $R:V(f) \to W\comma f$ sending a diagram (\ref{diag:bizarre}) to
\begin{equation}
\tag{$\zeta$}
\label{diag:normal}
\begin{tikzcd}
\s u \ar[d] \ar[r]  & \s f\ar[d]\\
\t u \ar[r]         & \t f
\end{tikzcd}
\end{equation}
The functor $R$ has a left adjoint $L:W\comma f \to V(f)$ sending a diagram (\ref{diag:normal}) to
\[
\begin{tikzcd}
&\s u \ar[d,equal] \ar[r] & \s f\ar[dd]\\
\s u \ar[d] \ar[r,equal]  &\s u \ar[rd]& \\
\t u \ar[rr] && \t f
\end{tikzcd}
\]
Consequently, the functor $R$ is cofinal.
The composition of $\t:W\comma f \to \cC$ by $R$ induces the expected functor $V(f) \to \cC$ %sending a (\ref{diag:bizarre}) to $\t u$
and we get
\[
\Plusplus f
= \colim_{W\comma \Plusmap f} \t v 
= \colim_{V(f)} \t v 
= \colim_{W\comma f} \t v 
= \Plus f.
\]
\end{proof}

\section{Applications}
\label{sec:applications}

\subsection{Generation of factorization systems}
\label{sec:factorizations}

\subsubsection{Example: the trivial factorization systems}
\label[example]{ex:SOA:trivial}

Any category $\cC$ has always two canonical factorization systems $(iso,all)$ and $(all,iso)$.
where $all$ is the whole of $\Arr \cC$ and $iso$ is the subcategory of isomorphisms $\cC\iso\subset \Arr \cC$.
If factorization systems on $\cC$ are ordered by the inclusion of their left classes, these systems correspond respectively to the minimal and maximal element.

The factorization system $(iso,all)$ is generated by any isomorphism in $\cC$, for example the identity map $1_A$ of some object $A$.
\cref{thm:!SOA2} applied to $W = \{1_A\}$ gives
\[
\pbh {1_A} f \pp 1_A = 1_{\map A {\s f}\times A}
\qquad\qquad
u(f) = 1_{\s f} \eqnand \Kellymap f = f.
\]
A single application of $k$ gives the factorization $f= f\circ 1_{\s f}$.

If $\cC$ has a initial object and a terminal object, the factorization system $(all,iso)$ is generated by the map $\iota=0\to 1$, which is the unit of the pushout-product.
\cref{thm:!SOA2} applied to $W = \{\iota\}$ gives
\[
\pbh \iota f \pp \iota = f
\qquad\qquad
u(f) = f \eqnand \Kellymap f = 1_{\t f}.
\]
Again, a single application of $k$ gives the factorization $f= 1_{\t f}\circ f$.

\subsubsection{Example: the image factorization}
\label[example]{ex:SOA:S0}

In the category $\cS$, we consider the object $S^0=1\coprod 1$ and the map $s^0:S^0\to 1$.
A map $f:X\to Y$ is said to be {\em $(-1)$\=/truncated} if it is a monomorphism, that is,  
if the diagonal map $\Delta f:X\to X\times_YX$ is invertible.
Since $\Delta f=\pbh {s^0}f$, this condition is equivalent to $s^0\perp f$.
The factorization system generated by $s^0$ is called the ($(-1)$\=/connected,$(-1)$\=/truncated) system.
The class of $(-1)$\=/connected maps can be proven to coincide with {\em covers} (\ie maps inducing a surjection on $\pi_0$).
Let $X \to P_{-1}f \to Y$ be the factorization of $f$. 
The monomorphism $P_{-1}f\to Y$ is the {\em image} of $f$.

We detail the construction of \cref{thm:!SOA2} to the discrete category $W = \{s^0\}$ and the map $f:X\to 1$.
We have
\[
\pbh {s^0} X \pp s^0
= 
\begin{tikzcd}
S^0\times X^2 \underset{S^0\times X}{\coprod} X \ar[d]\\
X^2
\end{tikzcd}
= 
\begin{tikzcd}
X^2 \underset{X}{\coprod} X^2 \ar[d]\\
X^2
\end{tikzcd}
\]
The construction $\Kellyvar X$ is defined by the pushout
\[
\begin{tikzcd}
X^2 \underset{X}{\coprod} X^2 \ar[d] \ar[r] & X \ar[d]\\
X^2 \ar[r] & \Kellyvar X \pomark
\end{tikzcd}
\]
Using the trick that $X=X\coprod_XX$ and $X^2=X^2\coprod_{X^2}X^2$ we get 
\[
\Kellyvar X
 = X^2 \coprod_{X^2 \underset{X}{\coprod} X^2} X
 = \left(X^2 \underset{X^2}{\coprod} X^2\right) \coprod_{X^2 \underset{X}{\coprod} X^2} \left(X \underset{X}{\coprod} X\right)
\]
and by commutation of the pushouts, we find
\[
\Kellyvar X = \left(X^2 \coprod_{X^2} X\right) \coprod_{X^2\underset{X}{\coprod} X} \left(X^2 \coprod_{X^2} X\right)
= X\coprod_{X^2}X
= X \join X
\]
In $\cS$, $S^0$ is a finite object and the countable iteration of $X\mapsto X
\join X$ converges to the support $P_{-1}X$ of $X$.
This iteration is the infinite join of $X$ which is indeed known to be the support.
In the case of a map $f:X\to Y$, $\Kellymap f = f\join f$ where $f\join f$ is the following cocartesian gap map
\[
\begin{tikzcd}
X\times_YX \ar[r]\ar[d]& X\ar[d]\ar[ddr, bend left, "f"]\\
X\ar[r]\ar[rrd, bend right, "f"]& X\underset{X\times_YX}\coprod X\ar[rd,"f\join f" description] \pomark
\\
&& Y.
\end{tikzcd}
\]
The countable iteration of $k$, that is, the infinite join $f^{\join \infty}$, is the image of $f$ (see~\cite{Rijke:join} for details).

\bigskip

More generally, the same construction can be made in any presentable locally cartesian closed category (\eg a topos, an $n$\=/topos, a modal topos...)
However, the enriched pushout-product/pullback-hom needs to be used (see \cref{rem:enrichedSOA}).
The object $s^0=1\coprod 1$ is internally finite and a countable iteration of
the $k$\=/construction converges to the ($(-1)$\=/connected,$(-1)$\=/truncated) factorization.
Moreover, in such a context, it can be proven that the class of $(-1)$\=/connected maps can also be described as covers~\cite[6.2.3]{Lurie:HTT}.

\subsubsection{Example: Postnikov factorization}
\label[example]{ex:SOA:Sn}

In the previous example, the consideration of the $(n+1)$\=/th sphere $S^{n+1}$ instead of $S^0$ provides a way to construct the $n$\=/th Postnikov truncation in any presentable locally cartesian closed category $\cC$ (\eg a topos, an $n$\=/topos, a modal topos...) via the enriched version of \cref{thm:!SOA2} (see~\cref{rem:enrichedSOA}).
Let $s^{n+1}$ be the map $S^{n+1}\to 1$.
We have $s^{n+1} = s^0\pp \dots \pp s^0$ ($n+2$ factors).
An object $X$ of $\cC$ is $n$\=/truncated iff it is orthogonal to $s^{n+1}$, iff the higher diagonal map $\Delta^{n+2}:X\to X^{S^{n+1}}$ is invertible.
We consider the discrete category $W = \{s^{n+1}\}$ and apply \cref{thm:!SOA2} to $f:X\to 1$.
We have
\[
\intpbh {s^{n+1}} X \pp s^{n+1} 
 = 
\begin{tikzcd}
S^{n+1}\times X^{S^{n+1}} \underset{S^{n+1}\times X}{\coprod} X \ar[d]\\
X^{S^{n+1}}
\end{tikzcd}
\]
The construction $\Kellyvar X$ is defined by the pushout
\[
\begin{tikzcd}
S^{n+1}\times X^{S^{n+1}} \underset{S^{n+1}\times X}{\coprod} X \ar[d]\ar[d] \ar[r] & X \ar[d]\\
X^{S^{n+1}} \ar[r] & \Kellyvar X. \pomark
\end{tikzcd}
\]
This pushout can be shown to be equivalent to the colimit of a wide pushout indexed by $S^{n+1}$ (see next example for details).
\[
\left.
\begin{tikzcd}
&&X\\
X^{S^{n+1}} \ar[rru]\ar[rrd] && \vdots\\
&&X
\end{tikzcd}
\right\} \text{ $S^{n+1}$ legs}
\]
We shall denote by $X^{\lhd S^{n+1}}$ this colimit.
Since $S^n$ is finite, the countable iteration of $X\mapsto X^{\lhd S^{n+1}}$ is the $n$\=/th Postnikov truncation $P_nX$ of $X$.

\medskip
From an external point of view on $\cC$, it is not true in general that the $n$\=/th Postnikov factorization system is generated by the sole map $s^{n+1}$.
One need to cross it with generators of $\cC$.
Let $C\subset \cC$ be a small generating subcategory.
We consider the diagram $C\to \Arr \cC$ sending an object $G$ to the map $G\times s^{n+1}$.
We have
\begin{align*}
\int^G \pbh {G\times s^{n+1}} X \pp (G\times s^{n+1}) 
&= \int^G \pbh {(0\to G)\pp s^{n+1}} X \pp (0\to G) \pp s^{n+1}
\\
&= \int^G \pbh {0\to G} {\intpbh {s^{n+1}} X} \pp (0\to G) \pp s^{n+1}
\\
&= \left(\int^G G\times \intpbh {s^{n+1}} X \times G\right) \pp s^{n+1}
\\
&= \intpbh {s^{n+1}} X \pp s^{n+1}
\end{align*}
and the corresponding $k$\=/construction gives back the enriched one.

\subsubsection{Example: nullification}
\label[example]{ex:SOA:nullification}

Still in the context of a presentable locally cartesian closed category $\cC$ (\eg a topos, an $n$\=/topos, a modal topos...), we can replace $S^n$ by any object $A$.
We denote by $X^{\lhd A}$ the object defined by the pushout
\[
\begin{tikzcd}
A \times X^A \ar[d,"p_2"'] \ar[r,"{(id,ev)}"] & A \times X \ar[d] \\
X^A   \ar[r]              & X^{\lhd A}. \pomark
\end{tikzcd}
\]
The object $X^{\lhd A}$ can also be expressed as the wide pushout
\[
\left.
\begin{tikzcd}
&&X\\
X^A \ar[rru, "p_a"]\ar[rrd, "p_{a'}"'] && \vdots\\
&&X
\end{tikzcd}
\right\} \text{ $A$ legs}.
\]
This is helpful to get an intuition of what $X^{\lhd A}$ is. 
This diagram is indexed contravariantly by an internal category that we shall denote by $A\join 1$ (essentially, $1\join A$ is $A$ viewed as an internal category augmented with an terminal object).

The construction $X^{\lhd A}$ is also related to that $\cJ_F(X)$ of \cite[Lemma 2.7]{RSS} for $F$ is the family of the single map $A\to 1$.
In this case, we have
\begin{equation}
\label{eqn:lhd=J}
\tag{$\cJ$}
\begin{tikzcd}
A \times X^A \ar[d,"p_2"'] \ar[r,"{(id,ev)}"] & A \times X \ar[d] \ar[r] & X \ar[d] \\
X^A   \ar[r]              & X^{\lhd A} \ar[r] \pomark & \cJ_A(X). \pomark
\end{tikzcd}
\end{equation}
In fact, one can prove that the two construction agree. 
The argument is similar to the fact that a reflective coequalizer is equivalent to the pushout of its parallel arrows:
$X^{\lhd A}$ is the colimit of all the projections $p_a:X^A\to X$ which all admits a common section given by the diagonal $\Delta:X\to X^A$.
Hence $X^{\lhd A}$ is also the colimit of the diagram
\[
\left.
\begin{tikzcd}
&&&X\\
X\ar[r,"\Delta"] \ar[rrru, bend left, equal]\ar[rrrd, bend right, equal]
&X^A \ar[rru, "p_a"]\ar[rrd, "p_{a'}"'] && \vdots\\
&&&X
\end{tikzcd}
\right\} \text{ $A$ legs}
\]
which is equivalently described by the outer pushout of (\ref{eqn:lhd=J}).
This diagram is indexed (contravariantly) by the internal category $\JJ_A$ with two objects
\[
\begin{tikzcd}
A \ar[rr,"1"']  \ar[from=rr, shift right=2,"A"'] \ar[loop left,"A \coprod 1"]
&& 1,
\end{tikzcd}
\]
where the decorations of arrows are the objects of morphisms. 
In the case where $A=S^0$, $\WW_{S^0}$ is the symmetric reflective coequalizer category.

\medskip
When $W$ is a single map $w:A\to 1$, the $k$\=/construction is
\[
\begin{tikzcd}
X\coprod_{A\times X}A\times X^A \ar[r] \ar[d,"\intpbh w X \pp w"'] & X \ar[d] \\
X^A \ar[r] & \Kellyvar X \pomark
\end{tikzcd}
\]
Using the trick that $X = X\coprod_{A\times X} A\times X$ and $X^A = X\coprod_X X^A$, we have
\[
\begin{tikzcd}
X\coprod_{A\times X}A\times X^A \ar[r] \ar[d,"\intpbh w X \pp w"'] & X\coprod_{A\times X} A\times X \ar[d] \\
X\coprod_X X^A \ar[r] & \pomark X \coprod_X X^{\lhd A}
\end{tikzcd}
\]
and $\Kellyvar X = X^{\lhd A}$.
Then, a transfinite iteration (whose length depends on the size of $A$) of
$X\mapsto X^{\lhd A}$ converges to the {\em $A$\=/localization} $P_AX$ of $X$.

\medskip
When the map $A\mono 1$ is monic, the construction can be simplified.
In this case, we have $A\times X^A = A\times X$ and we get $\Kellyvar X = X^{\lhd A} = X^A$.
Moreover, the operator $\Kellyconstruction$ is idempotent and $X^A$ is directly the $A$\=/localization of $X$.

\subsubsection{Example: co-nullification and co-reflexion}
\label[example]{ex:SOA:co-nullification}

We stay in the context of a presentable locally cartesian closed $\cC$ (\eg a topos, an $n$\=/topos, a modal topos...). 
We consider the case where $W$ is a single map $w:0\to B$. 
The initial object being strict, the pushout defining $\Kellyvar X$ simplifies into
\[
\begin{tikzcd}
B\times X^B \ar[r] \ar[d,"\intpbh w X \pp w"'] & X \ar[d] \\
B \ar[r] & \Kellyvar X \pomark
\end{tikzcd}
\]
In particular, if $B$ is a subterminal object, we have $B\times X^B = B\times X$ and $\Kellyvar X = B\join X$ which is an idempotent operator.

\medskip
When applied to a map $f:0\to Y$, instead of $X\to 1$, the $k$\=/construction gives the factorization
\[
\begin{tikzcd}
0 \ar[rr,"u(f)"]&& Y^B\times B \ar[rr,"\Kellymap f=ev"]&& Y
\end{tikzcd}
\]
More generally, for a full subcategory $W\subset \cC$, the $k$\=/construction applied with the collection of maps $0\to B$ for $B$ in $W$ gives the factorisation
\[
\begin{tikzcd}
0 \ar[rr,"u(f)"]&& \int^{B\in W}X^B\times B \ar[rr,"\Kellymap f"]&& X
\end{tikzcd}
\]
That is, the functor $Y\mapsto \Kelly {0\to Y}$ is the (internal) density comonad associated to the subcategory $W\subset \cC$ (the ``$W$\=/cellular approximation" of $Y$).
For example, if $\cC = \P C$ and $W\subset C$, $\P W \subto \P C$ is a coreflective subcategory, the density comonad is idempotent and the coreflection coincides with the $k$\=/construction $Y\mapsto \Kelly {0\to Y}$.

\subsubsection{Example: the case of a single map}
\label[example]{ex:SOA:map}

We stay in the context of a presentable locally cartesian closed $\cC$ (\eg a topos, an $n$\=/topos, a modal topos...). 
We consider the case where $W$ is a single map $w:A\to B$. 
We obtain
\[
\intpbh w X \pp w =
\begin{tikzcd}
B\times X^B\coprod_{A\times X^B}A\times X^A \ar[d] \\
B\times X^A
\end{tikzcd}
\]
and $\Kellyvar X$ is defined by the pushout
\[
\begin{tikzcd}
B\times X^B\coprod_{A\times X^B}A\times X^A \ar[r] \ar[d,"\intpbh w X \pp w"'] & X \ar[d] \\
B\times X^A \ar[r] & \Kellyvar X \pomark
\end{tikzcd}
\]
Outside of the two particular cases of \cref{ex:SOA:nullification,ex:SOA:co-nullification}, we have not find a nice way to simplify this formula.
Also, contrary to \cref{ex:SOA:nullification}, this formula does not seems to coincide with the construction $\cJ_F(X)$ of \cite[Lemma 2.7]{RSS}.

\subsubsection{Example: essentially surjective and fully faithful functors}
\label[example]{ex:image-cat}

Let $\Cat$ be the category of small categories. 
We shall look at $\Cat$ as a cartesian closed \oo category and consider the enriched version of \cref{thm:!SOA2} following \cref{rem:enrichedSOA}.
Let $I=\{\s \to \t\}$ be the one arrow category and $2=\{\s,\t\}$ be the discrete category with two objects.
A commutative square
\[
\begin{tikzcd}
2 \ar[r,"u"]\ar[d] & C\ar[d,"f"] \\
I \ar[r]& D
\end{tikzcd}
\]
is equivalent to the choice of two objects $x,y$ in $C$ and of an arrow $u:f(x)\to f(y)$ between their image in $D$.
There exists a unique lifting if and only if $f$ is {\em fully faithful}.
We consider the factorization system generated by $I\to 1$.
The functors in the left class are {\em essentially surjective} functors.
\[
\pbh {2\to I} {C\to D} = \Arr C \to C^2 \times_{D^2} \Arr D 
\]
where $C^2 \times_{D^2} \Arr D  = f\comma f$, the comma category of $f$ with itself.
Using that $\Arr C = C\comma C$, the $K$-construction is defined by the pushout
\[
\begin{tikzcd}
C \comma C \times I \underset{C \commaindex C \times 2}\coprod \left(f\comma f \times 2 \right)
\ar[r]
\ar[d]
    & C \ar[d]
\\
f\comma f  \times I \ar[r] 
    & K(f) \pomark
\end{tikzcd}
\]
The factor $f\comma f  \times I\coprod_{f\comma f \times 2} C$
adds an arrow between $x$ and $y$ in $C$ for any arrow $f(x)\to f(y)$ in $D$.
The factor $f\comma f  \times 2\coprod_{C\comma C \times I} C$
corrects this construction by identifying an arrow $u:x\to y$ to the new arrow corresponding to $f(u)$.
In the end, $K(f)$ is exactly the fully faithful subcategory image of $f:C\to D$.
Hence, the construction converges in one step.

\subsubsection{Example: localizations and conservative functors}
\label[example]{ex:localization-cat}

We keep the notations from \cref{ex:image-cat}.
A commutative square
\[
\begin{tikzcd}
I \ar[r,"u"]\ar[d] & C\ar[d,"f"] \\
1 \ar[r]& D
\end{tikzcd}
\]
is equivalent to the choice of an arrow $u$ in $C$ such that $f(u)$ is invertible in $D$.
The square admits a lift (necessarily unique) if and only if $u$ is already invertible on $C$.
A functor $f:C\to D$ is orthogonal to $I\to 1$ if and only if $u$ is {\em conservative}.

We consider the factorization system generated by $I\to 1$.
The functors in the left class are called {\em long localizations}~\cite{Joyal:TQC}.
We are going to see that they are countable chains of localizations 
(long localizations are necessary because localizations are not stable by composition, see below).
For a functor $f:C\to D$, let $\ker(f)$ be the class of maps in $C$ sent to isomorphisms in $D$.
Let $D\to \Arr D$ be the inclusion of the identity maps, thus $\ker f = \Arr C \times_{\Arr D} D$.
We apply \cref{thm:!SOA2} to $W=\{I\to 1\}$.
We have
\[
\pbh {I\to 1} {C\to D} = C \to \ker f
\]
and the $K$\=/construction is defined by the pushout
\[
\begin{tikzcd}
C \underset{C\times I}\coprod (\ker f\times I) \ar[r]\ar[d] & C\ar[d] \\
\ker f \ar[r]& \Kellyvar C \pomark
\end{tikzcd}
\]
which can be checked to be exactly the localization $\DKLoc C {\ker f}$ of $C$ along the arrows in $\ker f$.
Since the generating map $I\to 1$ is finite in $\Cat$, the left class is then countable chains of localizations.

In general, the $k$\=/construction needs to be applied more than once.
Let $\cC$ be the category $\{A\to B \ot C \to D\}$. 
The localization $\cD=\DKLoc \cC {C\to B}$ has an arrow $A\to D$ which is the composition of $A\to B$ and $C\to D$.
The localization $\cE=\DKLoc \cD {A\to D}$ is not a localization of $\cC$.
The $k$\=/construction applied to $\cC \to \cE$ provides the factorization $\cC \to \cD \to \cE$. 
It needs to be applied a second time to see that $\cC \to \cE$ is in the left class.

\subsubsection{Example: cofinal functors and fibrations in groupoids}
\label[example]{ex:fibrations-cat}

We continue to work in the cartesian closed category $\Cat$.
Recall that $I=\{\s \to \t\}$. 
We consider $W=\{1\xto \t I\}$. 
The subcategory $W^\intperp\subset \Arr \Cat$ of maps internally orthogonal $W$ is spanned by fibrations in groupoids.
The corresponding left class is that of cofinal functors.
Any functor $f:C\to D$ induces a functor $f_!:\P C\to \P D$.
Recall that $f$ is called {\em cofinal} if $f_!1=1$~\cite{Joyal:TQC, Lurie:HTT}.
Let $|f|:= f_!1$, we have $\P D \comma |f| = \P {D\comma |f|}$.
We deduce a factorization of $f$ into $C\to D\comma {|f|} \to D$, where $C\to D\comma {|f|}$ is cofinal and $D\comma {|f|}\to D$ is a fibration.
Two particular cases are classical.
If $C=1$, the factorization of $d:1\to D$ is $1 \to D\comma d \to D$ where $D\comma d$ is the slice category over $d$.
If $D=1$, the factorization of $C\to 1$ is $C\to |C| \to 1$, where $|C|=\DKLoc
CC$ is the external groupoid of $C$.

We claim that the $K$\=/construction applied twice to $f:C\to D$ is $D\comma {|f|}$.
Given a functor $f:C\to D$, we define the comma category $D\comma C$ as the pullback
\[
\begin{tikzcd}
D\comma C \ar[d]\ar[r] \pbmark & \Arr D \ar[d,"\t"]\\
C \ar[r,"f"]& D
\end{tikzcd}
\]
In particular, if $f$ is the identity of $C$, we have $\Arr C = C\comma C$.
Then, we have
\[
\intpbh {1\xto \t I} {C\to D} = 
\begin{tikzcd}
C\comma C \ar[d]\\
D\comma C
\end{tikzcd}
\eqnand
\intpbh {1\xto \t I} {C\to D} \pp (1\xto \t I)
= \begin{tikzcd}
D\comma C \underset{C\commaindex C}\coprod I\times C\comma C \ar[d]\\
I\times (D\comma C).
\end{tikzcd}
\]
The category $\Kelly f$ is defined by the pushout
\[
\begin{tikzcd}
D\comma C \underset{C\commaindex C}\coprod I\times C\comma C \ar[d] \ar[r]& C \ar[d]\\
I\times (D\comma C) \ar[r] & \Kelly f
\end{tikzcd}
\]
Let $C_0$ be the internal groupoid of $C$, we claim that $\Kelly f = \DKLoc {D\comma C} {C_0\comma C}$ (where the localization inverts all maps in the category $C_0\comma C$).
Applied a second time, we find $\Kellyn2 f = \DKLoc {D\comma C} {D_0\comma C} = D\comma {|f|}$.
In particular, if $C=1$, we have $\Kelly f = D\comma d$, and if $D=1$, we find $\Kelly f = \DKLoc CC = |C|$.
In these cases, a single application of $K$ suffices to get the factorization.
This is still true if $f:C\to D$ is essentially surjective since this implies
$\Kelly f =\Kellyn2 f$.

\subsection{Generation of localizations}
\label{sec:localizations}

The small object argument has been developed in its strong versions with the
motivation of constructing reflective localizations.
Our construction makes the results of~\cite{Gabriel-Ulmer,Kelly} valid for \oo categories.
All the following results are folkloric but some of them are difficult to find in the literature.
We put them here for future reference.

\bigskip
Let $\cC$ be a cocomplete category with a terminal object and $W\to \Arr \cC$ a diagram.
Recall that the category $\cC_W$ of {\em $W$\=/local objects} is defined to be the be full subcategory of $\cC$ spanned by the objects $X$ such that $X\to 1$ is in $W^\perp$.
The category $\cC_W$ is reflective if the inclusion $\cC_W\to \cC$ admits a left
adjoint $P_W$.

\begin{thm}[Orthogonal reflection]
\label{thm:orthogonal-reflection}
Let $\cC$ be a cocomplete category.
For any small diagram $W\to \Arr \cC$ with small domains and codomains, the transfinite iteration of the $K$\=/construction converges to the reflection $P_W$ of $\cC$ into $\cC_W$.
\end{thm}
\begin{proof}
Let $(\cL,\cR)$ be the factorization system generated by $W$.
Recall that $\cR\subset \Arr \cC$ is a reflective subcategory, with reflection given by $\Kellymapn \kappa -$.
The canonical inclusion $\cC\to \Arr\cC$ sending an object $X$ to $X\to 1$, sends $\cC_W$ to $W^\perp = \cR$.
Then the reflection $P_W:\cC\to \cC_W$ can be build as the restriction of the reflection $\Arr \cC \to \cR$ to $\cC$ composed with the domain projection $\s:\Arr \cC\to \cC$.
The result follows from \cref{thm:!SOA2} and the fact that $K$ is the domain of the $k$\=/construction.
\end{proof}

\begin{rem}
\label{rem:+localization}
When $W$ is a pre-modulator, the localization $P_W$ can also be computed as the iteration of the operator $\Plusconstruction$ of the plus-construction. 
\end{rem}

Let $\cC$ be a cocomplete category and $W\to \Arr \cC$ a diagram.
For any other cocomplete category $\cD$, we denote by $\fun \cC \cD \cc ^W$ the full subcategory of cocontinuous functors $\cC\to \cD$ sending $W$ to invertible maps in $\cD$.
Then we define the {\em cocontinuous localization} of $\cC$ by $W$ as the functor $\cC \to \relDKLoc {cc}{} \cC W$ representing the functor $\cD \mapsto \fun \cC \cD \cc ^W$ (where $\fun \cA\cB\cc$ is the category of functors $\cA\to \cB$ which are cocontinuous).

\begin{cor}[Cocontinuous localizations]%[Fundamental theorem of category theory]
\label{cor:fundaCT}
Let $\cC$ be a cocomplete category.
For any small diagram $W\to \Arr \cC$ with small domains and codomains, 
the reflection $P_W:\cC\to \cC_W$ is the localization $\cC \to \relDKLoc{cc}{} \cC W$.
In particular, any accessible cocontinuous localization of $\cC$ is reflective.
\end{cor}
\begin{proof}
The reflection $P_W:\cC\to \cC_W$ does invert the maps in $W$, hence it factors through a functor $\relDKLoc{cc}{} \cC W \to \cC_W$.
Reciprocally, the inclusion $\cC_W\subset \cC$ induces a functor $\cC_W\to \relDKLoc{cc}{} \cC W$.
We need to prove that these two functors are inverse to each other. 
$P_W$ induces the identity $\cC_W\to \cC\to \cC_W$. 
This proves that $\cC_W\to \relDKLoc{cc}{} \cC W \to \cC_W$ is the identity.

In order to prove that $\relDKLoc{cc}{} \cC W \to \cC_W\to \relDKLoc{cc}{} \cC W$ is the identity, 
it is enough to prove that $\cC \to \cC_W\to \relDKLoc{cc}{} \cC W$ is the localization functor.
Let $\overline W \subset \Arr \cC$ be the category of maps inverted by $\cC\to \relDKLoc{cc}{} \cC W$.
$\overline W$ is stable by composition and because the localization is cocontinuous, it is also stable by colimits in $\Arr \cC$.
We deduce that $\overline W$ is stable by pushout-product by maps in $\Arr \cS$, by coend and by transfinite compositions.
For any object $X$ in $\cC$, \cref{thm:!SOA2} construct the map $X\to P_WX$ as a
transfinite composition of coends of pushout-products of maps in $\overline W$, hence it is in $\overline W$.
Therefore, the functor $\cC\to \relDKLoc{cc}{} \cC W$ sends it to an invertible map.
This proves that this functor factors through $P_W:\cC\to \cC_W$ and that the composite $\cC \to \cC_W\to \relDKLoc{cc}{} \cC W$ is the localization functor $\cC\to \relDKLoc{cc}{} \cC W$.
\end{proof}

\medskip

When a category $\cC$ is not cocomplete, the SOA cannot be applied to compute localizations.
However, we can always cocomplete $\cC$ and apply the SOA there.
The following result gives a way to compute arbitrary localizations by means of reflective localizations.

\begin{cor}[Computation of arbitrary localizations]
\label{cor:arbitrary-loc}
Let $C$ be a small category and $W\to \Arr C$ a small diagram.
The localization $\DKLoc C W$ is the full subcategory of $\P C _W$ spanned by the image of $C\to \P C \to \P C_W$.
In particular, the space of morphisms $\map xy$ in $\DKLoc C W$ can be computed as the space of morphisms $\map {P_Wx}{P_Wy}=\map x{P_Wy}$ in $\P C$.
\end{cor}
\begin{proof}
Recall that a localization $C\to \DKLoc CW$ is an essentially surjective functor.
The canonical diagram
\[
\begin{tikzcd}
C \ar[d, "\text{ess.surj.}"']\ar[r,hook]& \P C \ar[d]\\
\DKLoc C W \ar[r,hook, "\text{f.f.}"]& \P {\DKLoc C W}
\end{tikzcd}
\]
proves that $\DKLoc C W$ is the image of the functor $C\to \P C \to \P {\DKLoc C W}$.
The canonical equivalences
\begin{align*}
\fun {\relDKLoc{cc}{} {\P C} W} \cE \cc 
&= \fun {\P C} \cE \cc ^W \\
&= \fun C \cE ^W \\
&= \fun {\DKLoc C W} \cE \\
&= \fun {\P {\DKLoc C W}} \cE \cc \\
\end{align*}
prove that $\P {\DKLoc C W} = \relDKLoc{cc}{} {\P C} W$.
\Cref{cor:fundaCT} applied to $W\to \Arr {\P C}$ proves that $\relDKLoc{cc}{} {\P C} W = \P C _W$.
Altogether, this proves that $\DKLoc C W$ is the image of $C\to \P C \to \P C_W$.
The last result is a direct consequence.
\end{proof}

More generally, if $C$ has already colimits of some class $\Kappa$ (like finite
sums or finite colimits), we can replace $\P C$ in the \cref{cor:arbitrary-loc} by the cocompletion of $C$ preserving the colimits in the class $\Kappa$.
This completion is $\relP \Kappa C := \DKLoc {\P C} {\Kappa(C)}$ where $\Kappa(C)$ is the class of $\Kappa$\=/cocones in $C$.

\begin{cor}
\label{cor:Kappa-loc}
Let $C$ be a $\Kappa$\=/cocomplete small category and $W\to \Arr C$ a small diagram.
The localization $\relDKLoc \Kappa {} C W$ is the full subcategory of $\relP \Kappa C _W$ spanned by the image of $C\to \relP \Kappa C _W$.
\end{cor}
\begin{proof}
Similar to that of \cref{cor:arbitrary-loc}. 
\end{proof}

\subsubsection{Example: the external groupoid of a category}
\label[example]{ex:SOA:shape-cat}

Let $C$ be a small category and $W= \Arr\cC$.
The localization $\DKLoc C W$ is the \oo groupoidal reflection of $C$, that we denote by $|C|$.
The category $\P C_W = \P {|C|}$ considered in \cref{cor:arbitrary-loc} is the
category of {\em local systems on $C$}, that is, the full subcategory of $\P C$
spanned by $C$\=/diagram made of invertible maps only.
The image of $C\to \P C_W$ send an object $x$ of $C$ to the universal cover of $|C|$ based at the point $x$.

\subsubsection{Example: truncations of finite spaces}
\label[example]{ex:SOA:Fin-truncation}

Let $\Fin\subset\cS$ be the full subcategory spanned by finite spaces.
$\Fin$ is stable by finite colimits and its ind-completion
(\ie under $\omega$\=/filtered colimits) is $\Ind \Fin = \cS$.
Let $W$ be the single map $s^{n+1}:S^{n+1} \to 1$.
We have $\relDKLoc {\text{cc}} {} \cS {s^{n+1}} = \cS\truncated n$.
Let $P_n:\cS\to \cS$ be the $n$\=/truncation reflector.
Applying \cref{cor:Kappa-loc}, we find that $\relDKLoc {\text{rex}} {} \Fin {s^{n+1}}$ is the full subcategory of $\cS \truncated n$ spanned by the image of $\Fin$, that is the full subcategory of $\cS\truncated n$ spanned by the $P_n(X)$ for $X$ in $\Fin$.

The localization $\Fin \to \relDKLoc {\text{rex}} {} \Fin {s^{n+1}}$, that is, the nullification of $S^{n+1}$, can be understood as sending a finite space $X$ to its ``formal" $n$\=/truncation.
This localization is not reflective, even though it preserves all colimits existing in $\Fin$.
This is because the construction of the reflection by a small object argument needs a countable colimit.
The classical way to say this is that the $n$\=/truncation of a finite space is not finite in general.
In any case, the ``formal" $n$\=/truncation can be ``realized" as an Ind-object in $\Fin$, that is, a  non-finite space.
This is the meaning of the embedding $\relDKLoc {\text{rex}} {} \Fin {s^{n+1}}\subset \cS\truncated n$ provided by \cref{cor:Kappa-loc}.
In other terms, the localization $\Fin \to \relDKLoc {\text{rex}} {} \Fin {s^{n+1}}$ has a right ind-adjoint.

\subsubsection{Example: localizations stable by composition}
\label[example]{ex:SOA:W-composition}

Let $\cC = \P C$ and $W\subset \Arr C$ be a full subcategory containing all the identity maps of $C$.
The localization $\P C_W = \P{\DKLoc C W}$ is spanned by all $F$ which satisfy $F(A)\simeq F(B)$ for any $A\to B$ in $W$.
The diagram $W$ is a pre-modulator and we get
\[
\Plusvar F(X) = \colim_{Y\to X \in W(X)\op} F(Y)
\]
Let $W^{(2)}\subset C^{\to\to}$ the category of pairs of composable arrows in $C$ which are in $W$.
If $W$ is stable by composition, we have a functor $W^{(2)}\to W$.
This functor has a right adjoint given by 
\begin{align*}
W &\tto W^{(2)} \\
A\xto w B & \mto A \xto w B \xto = B
\end{align*}
In particular, the functor $W\to W^{(2)}$ is cofinal.
We deduce that the plus-construction is idempotent:
\begin{align*}
\Plusplusvar F(X)
&= \colim_{Y\to X \in W(X)\op} \Plusvar F(Y)\\
&= \colim_{Y\to X \in W(X)\op} \colim_{Z\to Y \in W(Y)\op} F(Z)\\
&= \colim_{Z\to Y\to X \in W^{(2)}(X)\op} F(Z)\\
&= \colim_{Z\to X \in W(X)\op} F(Z) \\
&= \Plusvar F(X).
\end{align*}
Thus, if $W$ is a pre-modulator and stable by composition, the localization is given by a single application of the plus-construction.
This situation is related to \cref{prop:idempotent+}, but the fact that $W$ is stable by composition simplifies things.
Using \cref{cor:arbitrary-loc}, we deduce that the hom spaces in $\DKLoc C W$ are given by a ``calculus of fractions"
\[
\map X {\Plusvar Y} = \colim_{Z\to X \in W(X)\op} \map Z Y.
\]

\paragraph{De~Rham shape}
Here is an application in algebraic geometry.
Let $C=\textsf{Ring}_{\textsf{fp}}\op$ be the {\sl opposite} of the category of finitely presented commutative rings, $\cC$ is then the category of functors $\fun {\textsf{Ring}_{\textsf{fp}}}\cS$.
Let $W_1$ be the class of first order nilpotent extensions of rings and $W$ the class of all nilpotent extensions.
Since any nilpotent extension decomposes in a tower of first order nilpotent extensions, the class $W$ is the closure of $W_1$ for composition.
Notice that, for any ring $A$ in $C\op$, the category $W(A)\op$ has a terminal object given by the reduction of $A$.
We have an equivalence of localizations of $\DKLoc C {W_1} = \DKLoc C W$.
The category $\DKLoc {\P C} W = \P C_W$ is spanned by functors $F:C\op \to \cS$ which are {\em formally étale} (over 1).
These functors can be defined from a lifting condition with respect to $W_1$ or to $W$.
According the the previous computation, the reflection is given directly by $F\mapsto \Plusoriginal F$.
And, using the reduction of a ring, we have simply $\Plusvar F(A) = F(A_{red})$.
This is the formula defining the {\em de~Rham shape} of the presheaf
$F$~\cite{Simpson:DR,Schreiber:cohesion}.
By \cref{cor:arbitrary-loc}, $\DKLoc C W$ is equivalent to the full category of $\DKLoc {\P C} W = \P C_W$ spanned by the de~Rham shapes of representable functors.

\paragraph{Nullification}
The situation applies also in the nullification of a representable object in $\P C$ (\eg the $\AA^1$\=/localization in algebraic geometry).
Let $A$ be an object of $C$ to be nullified.
We assume $C$ has cartesian products and define $W$ to be the class of maps $A^n\times X \to X$ for $X$ in $C$, which is stable by composition.
The $A$\=/localization $P_A:\P C \to \P C$ of \cref{ex:SOA:nullification}, generated by the single map $A\to 1$, is equivalently generated by $W$.
We deduce that, for a presheaf $F$,
\[
\Plusvar F(X) = \colim_{\{A^n\times X\}\op} F(A^n\times X)
\]

\subsubsection{Example: localizations in a monoidal category}
\label[example]{ex:fibrations-cat}

Let $(\cC, \otimes, \bbone)$ be a presentable, additive, symmetric monoidal closed category
(\eg the category of modules over a commutative ring, a stable category, the category of cocomplete categories...)
Because $\cC$ is additive, there exists a zero object 0.
For any map $f:M\to N$, the {\em kernel} of $f$ is defined by the fiber product
\[
\begin{tikzcd}
\ker f \ar[d] \ar[r] \pbmark & M \ar[d,"f"]\\
0 \ar[r] & N
\end{tikzcd}
\]

The endomorphisms of the unit $\bbone$ form a commutative monoid.
We fix such an endomorphism $a:\bbone \to \bbone$ and consider the morphisms in $\cC$ with the (enriched) lifting property with respect to $a$.
\[
\begin{tikzcd}
\bbone \ar[d,"a"'] \ar[r]& M \ar[d,"f"]\\
\bbone \ar[r] \ar[ru, dashed] & N
\end{tikzcd}
\]
In the case $N=0$, the unique lifting condition is equivalent to the map $\pbh a {M\to 0} = M\xto {a} M$ being invertible.
We shall call {$a$-divisible} such an $M$.
For a general map $f:M\to N$, the unique lifting condition is equivalent to the kernel of $f:M\to N$ being $a$-divisible.

Kelly's construction for $M\to 0$ gives a diagram 
\[
\begin{tikzcd}
M\underset{a,M,a}\coprod M \ar[d,"f"'] \ar[r,"a"]& M \ar[d,"f"]\\
M \ar[r,"a"]& K(M) \pomark
\end{tikzcd}
\]
where the pushout in the upper left corner is that of the diagram $M \xot a M \xto a M$.
A direct computation proves that $K(M) = M$ with the horizontal bottom map being the identity and the right vertical map being the action of $a$.
The factorization $M\to M_a\to 0$ is then given by the colimit of the sequence
\[
M_a = \colim \left( M \nstto a M \nstto a M \nstto a \dots \right).
\]
Recall that from the point of view where $\cC$ is enriched over itself, the unit $\bbone$ is always of finite presentation.
This proves that it is enough to consider a countable sequence.
The resulting functor $(-)_a:\cC\to \cC$ is a reflection of $\cC$ into the full subcategory of $a$-divisible objects.

\paragraph{Localization of modules}
If $\cC$ is the category of $A$-modules, for a commutative monoid $A$, we have $A = End(\bbone)$ and $a:\bbone \to \bbone$ is simply an object of $A$.
For an $A$-module $M$, the object $M_a$ is the localization $M[a^{-1}]$ of the module $M$ with respect to $a$.
Let $\bbone[a^{-1}] = \bbone_a$. 
This is a monoid such that the $a$-divisible modules are exactly the $\bbone[a^{-1}]$-modules.
In particular we have $M[a^{-1}] = M \otimes \bbone[a^{-1}]$.

\paragraph{Stabilization of categories}
If $\cC$ is the monoidal category of pointed cocomplete categories, the unit is the category $\cS\pointed$ of pointed spaces. 
Let $a$ be the smash product with $S^1$.
For a pointed category $\cM$, the corresponding map $\Sigma:\cM\to \cM$ is the suspension functor, left adjoint to the loop space functor $\Omega:\cM\to \cM$.
The $S^1$-divisible categories are the stable categories and the category $\cM_{S^1}$ is the stabilisation $\text{St}(\cM)$ of $\cM$
\begin{align*}
\text{St}(\cM) 
&= \colim \left( \cM \nstto {\Sigma} \cM \nstto {\Sigma} \cM \nstto {\Sigma} \dots \right)
\\
&= \lim \left( \dots \nstto {\Omega} \cM \nstto {\Omega} \cM \nstto {\Omega} \cM \right).    
\end{align*}
Let $\Sp = \text{St}(\cS\pointed)$.
It is a monoidal category and the $S^1$-divisible categories are exactly the $\Sp$-modules.
In particular, we have $\text{St}(\cM) = \cM\otimes \Sp$, see~\cite[Example 4.8.1.23]{Lurie:HA} and \cite[2.1]{Robalo:motives}.

\subsection{Generation of modalities}
\label{sec:gen-mod}

In this section, we use \cref{thm:+construction} to give sufficient conditions on $W$ for the factorization system generated by to be a modality.
We assume that $\cC$ is a presentable category with universal colimits.

\begin{defi}[Modulator]
\label{defi:modulator} 
Let $C$ be generators for $\cC$.
We shall say that a pre-modulator $W\to \Arr \cC$ is a {\em modulator} if 
\begin{enumerate}[label=(\roman*), leftmargin=*]
\setcounter{enumi}{4}
\item \label{enum:modulator:5} the codomain functor $\t:W \to C$ is a fibration.
\end{enumerate}
That is if, for any $w$ in $W$, for any $X$ in $C$ and for any map $X\to \t w$, the base change $w':\s w\times_{\t w}X\to X$ is in the (essential) image of $W$.
\end{defi}

\begin{rem}
\label{rem:modulator=univalent-fam}
As mentioned in \cref{rem:internal+}, the data of a modulator is equivalent to that of a small full subuniverse $\WW \subset \UU$ containing the terminal object 1.
This is also equivalent to the data of a univalent family containing the terminal object.
\end{rem}

\begin{rem}
\label{rem:enriched-modulator}
The notion of modulator is our tool to generate modalities, that is, factorization systems stable by base change.
In order to generate factorization system that would only be {\em enriched} in the sense \ref{sec:enrichedFS} and \cref{table:strengthFS}, one would need to consider pre-modulators satisfying the extra assumption that $W$ is stable by base change along projections $X\times Y \to X$ in $C$.
\end{rem}

\begin{lemma}
\label{lem:BC-cart}
Given a cartesian square in $\Arr \cC$
\[
\begin{tikzcd}
g'\ar[r]\ar[d] \pbmark & f'\ar[d]\\
g\ar[r] & f
\end{tikzcd}
\]
such that the map $f'\to f$ is cartesian, then the map $g'\to g$ is also cartesian.
\end{lemma}
\begin{proof}
By the cancellation property of cartesian squares.  
\end{proof}

\begin{lemma}
\label{lem:Kan-opfib}
Let $C$ and $D$ be two small categories, $F:C\to D$ be an op-fibration in categories, 
and $G:C\to \cC$ be a functor with values in a cocomplete category.
The left Kan extension of $G$ along $F$ is given point wise by taking the colimit along the fibers of $F$: 
\[
Lan_FG\,(d) = \colim_{c:F(d)} G(c).
\]
\end{lemma}
\begin{proof}
Recall that the fiber $F(d)$ of $F:C\to D$ at $d$ and the ``fat'' fiber $C\comma d$ are defined by the pullbacks
\[
\begin{tikzcd}
F(d) \ar[r, "\iota_d"]\ar[d]\pbmark  
& C\comma d \ar[r]\ar[d]\pbmark 
& C \ar[d,"F"]
\\
\{d\} \ar[r]& D\comma d \ar[r]& D
\end{tikzcd}
\]
The general formula to compute left Kan extension is 
\[
Lan_FG(d) = \colim_{c:C\comma d} G(c).
\]
Let $F(d)$ be the fiber of $F:C\to D$ at $d$.
Because $F$ is an op-fibration the canonical functor $\iota_d:F(d)\to C\comma d$ has a left adjoint.
Hence $\iota_d$ is cofinal and we get the expected formula.
\end{proof}

\begin{lemma}
\label{lem:+cartesian}
Let $\cC$ be a presentable category with universal colimits, $W\to \Arr \cC$ a modulator and  $\alpha:g\to f$ a cartesian map in $\Arr \cC$.
Then, for any ordinal $\kappa$, the two following squares are cartesian.
\[
\begin{tikzcd}
\s g\ar[d] \ar[r] \ar[rd, phantom,"(1)"]
    & \s g^{+\kappa}\ar[d] \ar[r,"g^{+\kappa}"] \ar[rd, phantom,"(2)"]
        & \t g \ar[d] \\
\s f \ar[r]         & \s f^{+\kappa} \ar[r,"f^{+\kappa}"'] & \t f
\end{tikzcd}
\]
\end{lemma}

\begin{proof}
By cancellation for cartesian squares, it is enough to prove that the square (2) is cartesian.
We prove it first for $\kappa = 1$.
We need to show that the canonical map $\s g^+ \to \s \Plusmap f \times_{\t f} \t g$ is invertible where $\s g^+ = \colimit {W\commaindex g} \t w$ and
\begin{align*}
\s \Plusmap f \times_{\t f} \t g 
&= \left(\colimit {W\commaindex f} \t w\right) \times_{\t f} \t g \\
&= \colimit {W\commaindex f} \left(\t w \times_{\t f} \t g \right) \\
&= \colimit {W\commaindex f} \t (w \times_fg).
\end{align*}
The map $\s g^+ \to \s \Plusmap f \times_{\t f} \t g$ is constructed as follows.
Given $v\to g$ in $W\comma g$, we get a square
\[
\begin{tikzcd}
v \ar[d, equal]\ar[r] & g \ar[d,"\alpha"] \\    
v \ar[r] & f
\end{tikzcd}
\]
and a map $v\to v \times_fg$.
These maps induce a map of diagrams
\[
\begin{tikzcd}
W\comma g \ar[rd] \ar[rr]   & \ar[d, "\Ra" description, phantom] & W\comma f \ar[ld]\\
& \Arr \cC
\end{tikzcd}
\]
and a map between their colimits
\[
\colimit {W\commaindex g} w \ntto \zeta \colimit {W\commaindex f} w \times_fg.
\]
The domain of $\zeta$ is the canonical map $\s g^+ \to \s \Plusmap f \times_{\t f} \t g$.
The result will be proven if we show that $\zeta$ is invertible in $\Arr \cC$.

We are going to do so by introducing an auxiliary diagram $V \to \Arr \cC$ and two morphisms of diagrams
\[
\begin{tikzcd}
W\comma g \ar[rrd]
&\ar[rd, "\La" description, phantom]
& V \ar[d] \ar[rr] \ar[ll] 
&\ar[ld, "\Ra" description, phantom]
& W\comma f \ar[lld]
\\
&& \Arr \cC
\end{tikzcd}
\]
that will induce isomorphisms between the colimits.
The category $V$ is defined as the full subcategory of $\Arr \cC \comma \alpha$ spanned by squares
\[
\begin{tikzcd}
v \ar[d, "\beta"']\ar[r] 
& g \ar[d,"\alpha"] \\    
w \ar[r] & f
\end{tikzcd}
\]
where $v$ and $w$ are in $W$ and $\beta$ is cartesian.
The diagram $V\to \Arr\cC$ is given by the extraction of $v$.
There exists an obvious functor $A:V\to W\comma g$ which admits a left adjoint given by 
\[
v\to g \mto 
\begin{tikzcd}
v \ar[d, equal]\ar[r] & g \ar[d,"\alpha"] \\    
v \ar[r] & f
\end{tikzcd}.
\]
Hence, $A$ is cofinal and $\colimit V v = \colimit {W\comma g} v$.

There also exists an obvious functor $B:V\to W\comma f$ and a morphism of diagrams essentially given by the maps $v \to w\times_fg$.
The composition of the induced map $\colimit V v \to \colimit {W\comma f} w\times_fg$ with the previous identification $\colimit {W\comma g} v = \colimit V v$ gives back the map $\zeta$.

We are going to show that $B:V\to W\comma f$ is an op-fibration in categories.
The fiber of the functor $B$ at $w\to f$ is the full subcategory $V(w\to f)$ of $W\comma w \times_fg$ spanned by cartesian maps $v\to w \times_fg$.
Let $w\to w'\to f$ be a map in $W\comma f$ and $\beta:v\to w$ a cartesian map.
Recall that, because $\cC$ has finite limits, the composite map $v\to w\to w'$ in $\Arr \cC$ can be factored uniquely as $v\to v'\to w'$ where $\beta':v'\to w'$ is a cartesian map.
\[
\begin{tikzcd}
v \ar[d, "\beta"']\ar[r, dashed] & v' \ar[d,"\beta'", dashed] \\    
w \ar[r] & w'
\end{tikzcd}
\]
Let us see that the map $v'$ is in $W$.
First, $v'$ is such that $\t v'=\t v$, that is, its codomain is in the generating category $C\subset \cC$.
By \cref{lem:BC-cart}, the map $w\times_fg\to w$ is cartesian in $\Arr \cC$.
Hence the composite map $v'\to w\times_fg\to w$ is cartesian.
This prove that $v'$ is a base change of the map $w$ in $W$.
By assumption on $W$, $v'$ is then in $W$.
This finishes the proof that $B:V\to W\comma f$ is an op-fibration.
Let $B(w\to f)$ be the fiber of $B$ at $w\to f$, it is the category of diagrams
\[
\begin{tikzcd}
v \ar[d, "\beta"']\ar[r] 
& g \ar[d,"\alpha"] \\    
w \ar[r] & f
\end{tikzcd}
\]
where $\beta$ is cartesian.
But such data is equivalent to a cartesian map $v\to w\times_fg$, that is $B(w\to f)$ is the full subcategory of $W\comma w\times_fg$ spanned by cartesian maps.
Recall that the colimit of the diagram $V\to \Arr\cC$ is its left Kan extension along $V\to 1$.
Using the factorisation $V\to W\comma f \to 1$, the associativity of left Kan extensions, \cref{lem:Kan-opfib}, and the previous description of $B(w\to f)$, we get
\[
\colim_V v = \colim_{W\comma f} \colim_{v\to w\times_fg} v.
\]
We prove now that $\colimit {v\to w\times_fg} v = w\times_fg$.
The object $X=\t (w\times_fg)$ is the colimit of the canonical diagram $C\comma X\to \cC$.
For $\xi:Y\to X$ an object of $C\comma X$, let $w_\xi:Z\to Y$ be the pullback
\[
\begin{tikzcd}
Z\ar[d,"w_\xi"']\ar[r]\pbmark
&\s (w\times_fg) \ar[d,"w\times_fg"']\ar[r]\pbmark
&\s w \ar[d,"w"] \\
Y \ar[r,"\xi"]
&\t (w\times_fg) \ar[r]
&\t w.
\end{tikzcd}
\]
By universality of colimits the map $w\times_fg$ is the colimit of the maps $w_\xi$.
But the diagram of the $w_\xi$ is exactly the diagram of cartesian maps $v\to w\times_fg$ where $v$ is in $W$.
This proves that $\colimit {v\to w\times_fg} v = w\times_fg$.
Putting all together, we have proven that $\zeta$ is invertible.
\[
\colim_{W\comma g} v = \colim_V v = \colim_{W\comma f} w\times_fg.
\]

Then, we proceed by induction on $\kappa$.
If $\kappa$ is not a limit ordinal, then the result is a consequence of the
previous computation applied to $g^{+\kappa-1} \to f^{+\kappa-1}$.
If $\kappa$ is a limit ordinal, we have $\s g^{+\kappa} = \colimit {\lambda<\kappa}\s g^{+\lambda}$.
By induction hypothesis, we have $\s g^{+\lambda} = \s f^{+\lambda}\times_{\t f}\t g$.
By universality of colimits, we deduce
\[
\s g^{+\kappa} = \colimit {\lambda<\kappa}\left(\s f^{+\lambda}\times_{\t f}\t g\right) = \left(\colimit {\lambda<\kappa}\s f^{+\lambda}\right)\times_{\t f}\t g = \s f^{+\kappa}\times_{\t f}\t g.
\]
\end{proof}

\begin{rem}
\label{rem:stable+construction}
The first step of the induction in \cref{lem:+cartesian} says that the functor
$f\mapsto \Plusmap f$ sends cartesian maps to cartesian maps in $\Arr \cC$.
The lemma then says that this is also the case for any transfinite iteration of the plus-construction.
\end{rem}

\begin{thm}[Stable plus-construction]
\label{thm:+modality}
Let $\cC$ be a presentable locally cartesian closed category with a fixed small category of generators $C\subset \cC$.
If $W\to \Arr \cC$ is a modulator, then 
\begin{enumerate}[label={(\alph*)}, leftmargin=*]
\item \label{thm:+modality:enum1} the plus-construction $f\mapsto \Plusmap f$ preserves cartesian maps in $\Arr \cC$, 
\item \label{thm:+modality:enum2} the factorisation system generated by $W$ is a modality, and
\item \label{thm:+modality:enum3} a map $f$ is modal (\ie in $\cR$) if and only if $f=\Plusmap f$.
\end{enumerate}
\end{thm}
\begin{proof}
\noindent \ref{thm:+modality:enum1}
Direct from \cref{lem:+cartesian} and \cref{rem:stable+construction}.

\smallskip
\noindent \ref{thm:+modality:enum2}
According to \cref{thm:+construction}, the reflection into the right class is given by 
$\rho(f) = f^{+\kappa}$, for a certain $\kappa$.
Then, the conclusion follows from \cref{lem:+cartesian} and \cref{prop:stable/lex-FS}.

\smallskip
\noindent \ref{thm:+modality:enum3}
This is \cref{cor:+-fix=R}.
\end{proof}

\begin{rem}
The fact that the plus-construction respects cartesian maps says,
intuitively, that it is a fiberwise process on the map $f$.
In particular, it can be defined as an endomorphism $+:\UU\to \UU$ of the universe of the category $\cC$ (see \cref{rem:internal+}). 
\end{rem}

\begin{rem}
\label{rem:+modality-internal}
In keeping with \cref{rem:internal+,rem:modulator=univalent-fam}, \cref{thm:+modality} has an internal variation where the data of the modulator $W$ can be replaced by that of an {\em internal modulator}, that is a small internal category which is a subuniverse $\WW\subset \UU$ containing the terminal object.
\end{rem}

\begin{defi}
\label{defi:base-change-envelope}
Let $C\subset \cC$ be a generating small category.
For an arbitrary small diagram $W\to \Arr \cC$, we define the {\em modulator envelope} $W^{mod}$ of $W$ (relative to $C$) as the full subcategory of $\Arr \cC$ spanned by the identity of objects in $C$ and all the pullbacks of the maps in $W$ over the objects of $C$.
The diagram $W^{mod}\subto \Arr \cC$ is always a modulator.
\end{defi}

\begin{cor}[Generation of modalities]
\label{cor:gen-modality}
If $\cC$ is a presentable locally cartesian closed category with a fixed small category of generators $C\subset \cC$, the modality generated by an arbitrary family $W\to \Arr \cC$ can be constructed by applying \cref{thm:!SOA2} to the modulator envelope $W^{mod}$.
\end{cor}
\begin{proof}
By \cref{thm:+modality}, the factorization system $(\cL,\cR)$ generated by $W^{mod}$ is a modality.
We need to prove that it is the smallest modality (with respect to the order given by inclusion of left classes) such that $W\subset \cL$.
We prove first that $W\subset \cL$.
Let $w$ be a map in $W$.
For any $X$ in $C$ we define the diagram
\begin{align*}
w(-):C\comma \t w &\tto \Arr \cC \\
\xi &\mto w(\xi):X\times_{\t w} \s w\to X.
\end{align*}
By universality of colimits, we have $\colim w(\xi) = w$.
By construction all maps $w(\xi)$ are in $W^{mod}$.
Then, $w$ is in $\cL_{W^{mod}}$ because it is stable by colimits in $\Arr \cC$.

Let $(\cL',\cR')$ be another modality such that $W\subset \cL'$, then $W^{mod}\subset \cL'$.
Moreover, $\cL'$ is stable under the operations used by
\cref{thm:!SOA2} to produce the class $\cL$ (colimits and composition), so we have $\cL\subset \cL'$.
\end{proof}

\subsubsection{Example: Postnikov truncations and nullification}
\label[example]{ex:modalities:Sn}

Let $\cC$ be a presentable locally cartesian closed category (\eg a topos, an $n$\=/topos, a modal topos) and let $A$ be an object of $\cC$.
We consider the internal category $\WW_A$ with two objects, $A$ and the terminal object 1, and all maps between them.
This category look like this
\[
\begin{tikzcd}
A \ar[rr,"1"']  \ar[from=rr, shift right=2,"A"'] \ar[loop left,"\intmap AA"]
&& 1    
\end{tikzcd}
\]
where the decorations of arrows are the objects of morphisms.
In the case where $A=S^0$, $\WW_{S^0}$ is the symmetric reflective coequalizer category.
Recall the categories $A\join 1$ and $\JJ_A$ from \cref{ex:SOA:nullification}.
There are obvious functors $A\join 1 \to \JJ_A \to \WW_A$
The following lemma is left to the reader

\begin{lemma}
The two functors $A\join 1 \to \JJ_A \to \WW_A$ are cofinal.
\end{lemma}

The category $\WW_A$ is an internal modulator in the sense of \cref{rem:+modality-internal}, but this not the case of $A\join1$ and $\JJ_A$.
The internal version of \cref{thm:+modality} then produces the nullification modality generated by $A$.
The associated localization is the $A$\=/localization of \cref{ex:SOA:nullification}.
The previous lemma implies that the internal Kelly's constructions generated by $A\to 1$ and by $\WW_A$ coincides.
Moreover, they coincide also with the construction $\cJ_A$ of \cite[Lemma 2.7]{RSS}.

\subsubsection{Example: generation of semi-left-exact localizations}
\label[example]{ex:modalities:slex}

A classical result says that a reflective factorization system $(\cW,\cM)$ is a left-exact localization iff $\cL$ is stable by base change (see \cite[Corollary 4.8]{CHK:slex} or \cite[Proposition 6.2.1.1]{Lurie:HTT}).
This leads to the question of whether the localization generated by a class $W\subto \Arr \cC$ of maps stable under base change is left-exact. 
The answer is negative (consider the Postnikov truncations) but it is always semi-left-exact.
If $W$ is a modulator, the localization generated by $W$ is the localization associated to the modality generated by $W$ by \cref{thm:+modality}.
Then, the result follows from \cref{prop:mod-slex}.

We shall see in \cref{thm:lex+construction,thm:lex-diagonal} two enhancements of the condition of stability by base change for $W$ to generate an actual left-exact localization.

\subsubsection{Example: LCC categories are modal topoi}
\label[example]{ex:modalities:ind-completion}

Let $\cC$ be a presentable locally cartesian closed category. 
There exists a cardinal $\kappa$ such that the subcategory $C^\kappa \subset \cC$ of $\kappa$\=/small objects generates $\cC$ and is stable by finite limits.
We are going to prove that there exists a modality on the topos $\P {C^\kappa}$ such that $\cC\subto \P {C^\kappa}$ is the category of modal objects.

For any $\kappa$\=/small diagram $c:I\to C^\kappa$, we denote by $\colim c_i$ the colimit in $C^\kappa$ and by $``\colim" c_i$ the colimit of the same diagram in $\P {C^\kappa}$.
There is a comparison map $``\colim" c_i \to \colim c_i$ in $\P {C^\kappa}$ and we let $W\subto \Arr {\P {C^\kappa}}$ be the full subcategory spanned by all these maps.
An object $F$ of $\P {C^\kappa}$ is in $\relInd \kappa {} {C^\kappa}=\cC$ iff it is right orthogonal to all maps of $W$.
In other words, the localization along $W$ is the localization $\P {C^\kappa}\to \relInd \kappa {} {C^\kappa}$.

We claim that $W$ is a modulator. 
The pre-modulator axioms \ref{enum:modulator:1} to \ref{enum:modulator:4} are easy, we need only to check that $W$ satisfies axiom \ref{enum:modulator:5}, \ie the stability by base change.
Let $d \to \colim c_i$ be a map in $C^\kappa$. 
Because finite colimits are universal in $C^\kappa$, we have $d= \colim \left(d\times_{\colim c_i}c_i\right)$.
Because finite colimits are also universal in $\P {C^\kappa}$, we have $d\times_{\colim c_i}``\colim" c_i = ``\colim" \left(d\times_{\colim c_i}c_i\right)$.
In other words, we have cartesian square
\[
\begin{tikzcd}
``\colim" \left(d\times_{\colim c_i}c_i\right) \ar[r]\ar[d] \pbmark & ``\colim" c_i \ar[d]\\
d=\colim \left(d\times_{\colim c_i}c_i\right) \ar[r]& \colim c_i.
\end{tikzcd}
\]
This proves that the base change of $``\colim" c_i \to \colim c_i$ along $d \to \colim c_i$ is still in $W$.
By \cref{thm:+modality}, $W$ generates a modality and the modal object are precisely the objects of $\relInd \kappa {} {C^\kappa}=\cC$.

\begin{prop}
\label{prop:PLCCC=modal-topos}
The following conditions on a category $\cC$ are equivalent:
\begin{enumerate}[label={(\alph*)}, leftmargin=*]
\item \label{prop:PLCCC=modal-topos:3} $\cC$ is presentable and locally cartesian closed;
\item \label{prop:PLCCC=modal-topos:1} $\cC$ is an accessible modal topos (see \cref{sec:modal-topos});
\item \label{prop:PLCCC=modal-topos:2} $\cC$ is an accessible semi-left-exact localization of a topos.
\end{enumerate}
\end{prop}
\begin{proof}
We just proved \ref{prop:PLCCC=modal-topos:3} $\Ra$ \ref{prop:PLCCC=modal-topos:1}.
The implication \ref{prop:PLCCC=modal-topos:1} $\Ra$ \ref{prop:PLCCC=modal-topos:2} is \cref{prop:mod-slex}, 
and \ref{prop:PLCCC=modal-topos:2} $\Ra$ \ref{prop:PLCCC=modal-topos:3} is \cite[Proposition 1.4]{Gepner-Kock:Univalence}.
\end{proof}

\subsection{Generation of lex modalities/lex localizations}
\label{sec:gen-lex-mod}
\label{sec:sheafification}

\subsubsection{The finite limit criterion}

We now turn to the problem of generating left-exact localizations, that is lex
factorization systems, or lex modalities (see \cref{sec:stableFS}).
We stay in the context of a presentable category $\cC$ with universal colimits.

\begin{defi}[Lex modulator]
\label{defi:lex-modulator}
Let $C\subset \cC$ be a generating small category. 
We shall say that a modulator $W\to \Arr \cC$ is {\em left-exact}, or {\em lex}, if
\begin{enumerate}[label=(\roman*), leftmargin=*]
\setcounter{enumi}{5}
\item \label{enum:modulator:6} for all $X$ in $C$, the fiber $W(X)$ of $\t:W\to C$ is a co-filtered category.
\end{enumerate}
We shall say that a modulator $W\to \Arr \cC$ is {\em stable under finite limits} if 
\begin{enumerate}[label=(\roman*$^+$), leftmargin=*]
\setcounter{enumi}{5}
\item \label{enum:modulator:6+} the image of $W\to \Arr \cC$ is stable under finite limits in $\Arr \cC$.
\end{enumerate}
This condition implies that the generating category $C$ must have finite limits.
A modulator stable under finite limits is lex. 
In fact, together with \ref{enum:modulator:3} and \ref{enum:modulator:4}, 
\ref{enum:modulator:6+} implies both condition \ref{enum:modulator:5} and \ref{enum:modulator:6}.
\end{defi}

\begin{thm}[Lex plus-construction]
\label{thm:lex+construction}
Let $\cE$ be an $n$\=/topos ($n\leq \infty$) and $C\subset \cE$ a small generating subcategory.
If $W\to \Arr \cE$ is a lex modulator
(in particular if it is stable under finite limits)
then
\begin{enumerate}[label={(\alph*)}, leftmargin=*]
\item \label{thm:lex+construction:1} the functor $f \mapsto \Plusmap f$ is left-exact as an endofunctor of $\Arr \cE$, 
\item \label{thm:lex+construction:2} the factorization system generated by $W$ is a lex modality (\ie a lex localization), and
\item \label{thm:lex+construction:3} A map $f$ is a relative sheaf (\ie in $\cR$) if and only if $f=\Plusmap f$.
\end{enumerate}
\end{thm}

\begin{proof}
We prove the theorem first for \oo topoi.

\smallskip
\noindent \ref{thm:lex+construction:1}
We first consider the case $\cE = \P C$ of a presheaf \oo topos.
Let $f$ be a map in $\cC$, then for each object $X$ of $C$, we have a map $f(X)$ in $\cS$.
The plus-construction is left-exact iff all functors $f\mapsto \Plusmap f(X)$ are left-exact.
Since $\t \Plusmap f = \t f$ and the functor $\t:\Arr \cE \to \cE$ is always left
exact, the condition reduces to proving that all functors $f\mapsto \Plus f(X)$ are left-exact.
Let $W(X)$ be the fiber at $X$ of the target functor $\t:W\to C$.
By \cref{lem:+=+} and \cref{rem:+=+}, we have
\[
\Plus f(X) \quad=\quad \colimit {W(X)\op}\map w f.
\]
By assumption on $W$, $W(X)\op$ is filtered and the functor $f\mapsto \Plus f(X)$ is left-exact.

If $\cE$ is a left-exact localization of some $\P C$, we have an adjunction $P:\P C \rightleftarrows \cE:\iota$ where both functors are left-exact.
For a map $f$ in $\cE$, the plus-construction is defined by a colimit indexed by $W\comma f$.
The fully faithful functor $\iota:\cE\subto \P C$ induces a fully faithful functor  $\iota:\Arr \cE \subto \Arr {\P C}$ and the category $W\comma f$ computed in $\Arr \cE$ or $\Arr {\P C}$ is the same.
Hence, the plus-construction of $\cE$ is related to that of $\P C$ by the formula
\[
\Plusmap f = P\big((\iota f)^+\big).
\]
Since the three functors $P$, $\iota$ and the plus-construction of $\P C$ are left-exact, so is the plus-construction in $\cE$.

\smallskip
\noindent \ref{thm:lex+construction:2}
Since filtered colimits of left-exact functors are still left-exact, the transfinite iterations of $+$ are also left-exact.
We deduce that the reflector $\rho$ is left-exact. 
Then, the result follows from \cref{prop:stable/lex-FS}.

\smallskip
\noindent \ref{thm:lex+construction:3}
This is \cref{cor:+-fix=R}.

\medskip
This finishes the proof for \oo topoi.
The strategy is similar for $n$\=/topoi.
The presheaf category $\P C = \fun {C\op} \cS$ needs only to be replaced by the category $\cP_n(C) = \fun {C\op} {\cS\truncated {n-1}}$, where $\cS\truncated {n-1}$ is the category of $(n-1)$\=/groupoids.
\end{proof}

\begin{rem}
The proof of \cref{thm:lex+construction} can be simplified when $W$ is stable by finite limits and $n=\infty$.
The target functor $\t:W\to \cE$ is left-exact, so its left Kan extension $\P W \to \cE$ is left-exact as well, because $\cE$ is an \oo topos~\cite[Thm 2.1.4]{Anel-Lejay:topos-exp}.
The functor $\Arr \cE \to \cE$ sending a map $f$ to $\Plus f = \colim_{W\commaindex f} \t w$ factors into $\Arr \cE \to \P W\to \cE$ and both functors are left-exact.
This proves directly that the plus-construction is left-exact.
The same proof would would apply for 1-topoi using~\cite[Prop. 2.6]{GL:lexcolimits}, and for an $n$\=/topos using the analog result.
\end{rem}

\begin{rem}
\label{rem:lex-modulator=sub-universe}
Recall from \cref{rem:modulator=univalent-fam} that a modulator is equivalent to a small full subuniverse $\WW \subset \UU$ containing the terminal object 1.
Then a lex modulator is almost a modulator such that $\WW$ is an internally cofiltered category.
The two notions coincides in a presheaf category $\cC = \P C$, but in an arbitrary $\cC$, being internally filtered is slightly weaker.
The internal version of \cref{thm:lex+construction} can be proven assuming this weaker hypothesis.
\end{rem}

\begin{rem}
\label{rem:other-limits}
The class of finite limits in \cref{thm:lex+construction} can be replaced by an arbitrary class $\Lambda$ of limits provided $\Lambda$ is a sound doctrine in the sense of~\cite{ABLR}.
The theorem would then be valid in categories $\cE$ which are obtained by a localization $\P C\to \cE$ preserving the limits in $\Lambda$.
For example, $\Lambda$ could be the class of finite products.
In this case, the condition on $W$ can be weaken into the condition that all $W(X)$ be co-sifted.
\end{rem}

\begin{lemma}
\label{lem:BC-cart2}
Given a cartesian map $f'\to f$ in $\Arr \cC$, there exists a cartesian square
\[
\begin{tikzcd}
f'\ar[r]\ar[d] \pbmark & f\ar[d]\\
1_{\t f'}\ar[r] & 1_{\t f}
\end{tikzcd}
\]
\end{lemma}
\begin{proof}
A straightforward computation.
\end{proof}

\begin{prop}
\label{prop:lex-mod-4-loc-lex}
Let $f^*:\cE\to \cF$ be an accessible left-exact localization of a topos $\cE$.
Then, there exists a lex modulator presenting $f^*$.
\end{prop}
\begin{proof}
We fix some generators $C \subset \cE$ of $\cE$.
Let $W\subset \Arr \cE$ be the subcategory of arrows inverted by $f^*$.
Because $f^*$ is a left-exact localization, $W$ is stable by finite limits in $\Arr \cE$.
By assumption, there exists a regular cardinal $\kappa$ such that $W$ is a $\kappa$\=/accessible category.
Let $W(\kappa)\subset W$ be subcategory of $\kappa$\=/compact objects.
It is possible to chose $\kappa$ such that $W(\kappa)$ is stable by finite limits and such that any object of $C$ is $\kappa$\=/compact.
For such a $\kappa$, we put $W' = W(\kappa)\cap \P C \comma C$.
Let us show that it is a lex modulator with respect to the generators $C$.
First, $W'$ contains always the isomorphisms between objects of $C$.
Therefore, the stability by finite limits of $W$ and \cref{lem:BC-cart2} implies that $W'$ is stable by base change.
Then, for any $X$ in $C$, we need to prove that $W'(X)$ is op-filtered.
Let $D:I\to W'(X)$ be a finite diagram, its limit in $W'$ has codomain $\lim_IX = X^{|I|}$ (where $|I| = \colim_I1$) which might not be an object of $C$ since we have not assumed $C$ to be stable by finite limits.
But the base change of this limit along the diagonal $X\to X^{|I|}$ is an element in $W'(X)$ which provide a cone for the diagram $D$.
This finishes to prove that $W'$ is a lex modulator.
\end{proof}

\subsubsection{The diagonal criterion}

In this section, we provide conditions to generate a lex modality which are weaker than those of \cref{thm:lex+construction}.
In consequence, the result is also going to be weaker in the sense that the $+$/$k$\=/constructions are no longer be left-exact, only the resulting reflection operator $\cR$ is going to be.

Given a diagram $W\to \Arr \cC$, if $(\cL_W,\cR_W)$ is the orthogonal system generated by $W$, we have a factorization $W\to \cL_W \subto\Arr \cC$ of the generating diagram.

\begin{lemma}
\label{lem:more-generators}
For any factorization $W\to W'\to \cL_W$, the orthogonal system generated by $W'$ coincides with the one generated by $W$.
\end{lemma}
\begin{proof}
We have $\cR_W = \cL_W^\perp \subset (W')^\perp \subset W^\perp = \cR_W$.
\end{proof}

Given a diagram $W\to \Arr \cC$, in a category $\cC$ with finite limits,
let $W\lex$ be the closure of the image of $W$ by finite limits in $\Arr \cC$.

\begin{prop}
\label{prop:lex-closure-gen}
Let $\cC$ be a presentable category.
The factorization system generated by a small diagram $W\to \Arr \cC$ is left-exact if and only if $W\lex \subset \cL_W$.
\end{prop}
\begin{proof}
If $(\cL_W,\cR_W)$ is left-exact, then $\cL_W$ is stable by finite limits and $W\lex \subset \cL_W$.
Reciprocally, by \cref{lem:more-generators}, $W\lex$ generates the same factorization system.
The conclusion follows from \cref{thm:lex+construction}.
\end{proof}

Given a diagram $W\to \Arr \cC$, we define $W^\Delta$ to be the smallest full subcategory of $\Arr \cC$ containing the image of $W$ and all their diagonals (\ie stable under the operations $\pbh {s^n}-$).

\begin{defi}[$\Delta$\=/modulator]
\label{defi:Delta-modulator}
We shall say that a modulator $W\to \Arr \cC$ is a {\em $\Delta$\=/modulator} if
\begin{enumerate}[label=(\roman*$^-$), leftmargin=*]
\setcounter{enumi}{5}
\item \label{enum:modulator:6-} $W^\Delta\subset \cL_W$.
\end{enumerate}
This condition is independent of \ref{enum:modulator:6} and weaker than \ref{enum:modulator:7}.
\end{defi}

\begin{thm}[Diagonal lex criterion]
\label{thm:lex-diagonal}
Let $\cE$ be an $n$\=/topos ($n\leq \infty$) with fixed generators. 
Let $W\to \Arr \cE$ be a $\Delta$\=/modulator, then the factorization system generated by $W$ (which is a modality by \cref{thm:+modality}) is left-exact.
\end{thm}
\begin{proof}
Using \cref{prop:lex-closure-gen}, we need to prove that $W\lex\subset \cL_W$. 
Since $\cL_W$ contains all isomorphisms, it is enough to prove that the fiber products of maps in $W$ are in $\cL_W$.
Let $f\to g \ot h$ be a pullback diagram in $W$, we decompose it in the composition of three pullbacks:
\[
\begin{tikzcd}
\s f \ar[d,equal] \ar[r]   & \s g \ar[d]       & \s h \ar[d,equal] \ar[l] \ar[d,bend left=100,phantom, "(1)" description]\\
\s f \ar[d]  \ar[r]        & \t g \ar[d,equal] & \s h \ar[d,equal] \ar[l] \ar[d,bend left=100,phantom, "(2)" description]\\
\t f \ar[d,equal]  \ar[r]  & \t g \ar[d,equal] & \s h \ar[d] \ar[l] \ar[d,bend left=100,phantom, "(3)" description]\\
\t f \ar[r]                & \t g              & \t h \ar[l]
\end{tikzcd}
\]
The pullback of (1) is in $\cL_W$ because it is a base change of $\Delta g$ (which is in $\cL$ by assumption).
The pullback of (2) is in $\cL_W$ because it is a base change of $f$.
The pullback of (3) is in $\cL_W$ because it is a base change of $g$.
Then, the pullback of $f\to g \ot h$ is in $\cL_W$ because $\cL_W$ is stable by composition.
\end{proof}

\begin{lemma}
\label{lem:diagonal-cartesian}
Let $\cE$ be a category with finite limits.
Let $g\to f$ be a cartesian map in $\Arr \cE$, then the induced map between the diagonal $\Delta g \to \Delta f$ is also cartesian.
\end{lemma}
\begin{proof}
The map $\Delta g \to \Delta f$ corresponds to the square
\[
\begin{tikzcd}
\s g \ar[r]\ar[d] & \s f \ar[d] \\    
\s g \times_{\t g} \s g \ar[r] & \s f \times_{\t f} \s f.
\end{tikzcd}
\]
Because $f\to g$ is cartesian, we have $\s g = \s f \times_{\t f} \t g$.
Using this in the previous square, we get
\[
\begin{tikzcd}
\s f \times_{\t f} \t g \ar[r]\ar[d] & \s f \ar[d] \\    
\s f \times_{\t f} \s f \times_{\t f} \t g \ar[r] & \s f \times_{\t f} \s f
\end{tikzcd}
\]
which is clearly cartesian.
\end{proof}

\begin{cor}
\label{cor:diagonal-lex-gen}
Let $\cE$ be an $n$\=/topos ($n\leq \infty$).
For any small diagram $W\to \Arr \cE$, the modality generated by $W^\Delta$ is left-exact.
\end{cor}
\begin{proof}
Let $W^{mod}$ be the modulator envelope of $W$ (\cref{defi:base-change-envelope}) and let $f$ be in $W^{mod}$.
By construction of $W^{mod}$, there exists some $w$ in $W$ and a cartesian map $f \to w$.
Using \cref{lem:diagonal-cartesian}, we have a cartesian map $\Delta f \to \Delta w$.
By \cref{cor:gen-modality}, the modality $(\cL,\cR)$ generated by $W^\Delta$ is the factorization system generated by $(W^\Delta)^{mod}$.
The class $\cL$ contains $W^\Delta$ and is stable by base change.
By the previous computation, it contains $(W^{mod})^\Delta$.
Then, the result follows from \cref{thm:lex-diagonal}.
\end{proof}

It is proven in~\cite[Cor. 3.12]{RSS}, that the modality generated by a diagram $W\to \cC^\mono$ of monomorphisms is always left-exact.
The corresponding statement in our setting is the following.

\begin{cor}
\label{cor:lex-mono-gen}
Let $\cE$ be an $n$\=/topos ($n\leq \infty$).
The factorization system generated by a modulator made of monomorphisms is always a lex modality.
\end{cor}
\begin{proof}
If $w$ is a monomorphism, all its diagonal maps are invertible.
$W^\Delta$ is then the union of $W$ and a collection of isomorphisms.
Hence, the modalities generated by $W$ and $W^\Delta$ coincide.
\end{proof}

\begin{rem}
\label{rem:mono-not-lex+}
Despite \cref{cor:lex-mono-gen}, the plus-construction of \cref{thm:lex+construction} associated to a diagram of monomorphisms $W\to \cC^\mono$ need not be left-exact (see \cref{ex:closed-loc})
The sole hypothesis on $W$ to be stable by base change is not enough in general to prove that the $W(X)$ are co-filtered (although they have a terminal object).
Further hypothesis are needed, like the locality condition of Grothendieck topologies, see \cref{ex:GTopo}.
The same remark apply for the plus-construction in \cref{thm:lex-diagonal}.
Without the stronger hypothesis of \cref{thm:lex+construction}, it may not be left-exact even though its iterations converge to a left-exact functor.
\end{rem}

For $W\to \Arr \cE$ any diagram of maps in a $n$\=/topos $\cE$ ($n\leq \infty$), let $W^\Delta\subto \Arr \cE$ be the full subcategory spanned by all the diagonals $\Delta^n w = \pbh {s^n}w$ of maps $w$ in $W$,
and let $W^{\Delta mod}=(W^\Delta)^{mod}$ the modulator envelope of $W^\Delta$ with respect to some generators (\cref{defi:base-change-envelope}).

\begin{thm}[Generation of left-exact localizations]
\label{thm:lex-loc-gen-by-map}
Let $W\to \Arr \cE$ be a diagram of maps in an $n$\=/topos $\cE$ ($n\leq \infty$).
\begin{enumerate}[label=(\alph*), leftmargin=*]
\item \label{thm:lex-loc-gen-by-map:1} The left-exact modality generated by $W$ is the factorization system generated by the modulator $W^{\Delta mod}$.
\item \label{thm:lex-loc-gen-by-map:2} The left-exact localization $P_W$ of $\cE$ generated by $W$ is the iteration of the plus-construction associated to the modulator $W^{\Delta mod}$.
\item \label{thm:lex-loc-gen-by-map:3} An object $X$ in $\cE$ is local for the left-exact localization generated by $W$ if and only if $X$ is orthogonal to all the diagonals of $W$ and all their base change.
\end{enumerate}
\end{thm}
\begin{proof}
\noindent \ref{thm:lex-loc-gen-by-map:1}
Let $(\cL_W^{\Delta mod},\cR_W^{\Delta mod})$ be the factorization system generated by $W^{\Delta mod}$.
The proof of \cref{cor:diagonal-lex-gen} shows that $W^{\Delta mod}$ is a $\Delta$\=/modulator. 
By \cref{thm:lex-diagonal}, $(\cL_W^{\Delta mod},\cR_W^{\Delta mod})$ is a lex modality.
The class $W^{\Delta mod}$ is necessarily contained in the left class $\cL$ of any lex modality containing $W$, hence $\cL_W^{\Delta mod}\subset \cL$.
This prove that $(\cL_W^{\Delta mod},\cR_W^{\Delta mod})$ is the smallest lex modality containing $\cL$.

\smallskip
\noindent \ref{thm:lex-loc-gen-by-map:2}
This is \cref{rem:+localization}.

\smallskip
\noindent \ref{thm:lex-loc-gen-by-map:3}
This is just a reformulation of $X\to 1$ being in $(W^{\Delta mod})^\bot$.
\end{proof}

\begin{rem}
In accordance with \cref{rem:mono-not-lex+}, the operator $\Plusconstruction$ of \cref{thm:lex-loc-gen-by-map}\ref{thm:lex-loc-gen-by-map:2} need not be left-exact.
\end{rem}

\subsubsection{Lex localization of truncated objects}
\label{sec:sheaf-truncated}

In \cref{prop:idempotent+} we gave a saturation condition under which the plus-construction converges in one step.
In this section, we give a weaker condition under which the plus-construction associated to a lex modulator converges in $n+2$ steps on $n$\=/truncated objects.

\begin{defi}[mono-saturation]
\label{lem:mono-saturated}
Let $\cE$ be an $n$\=/topos ($n\leq \infty$) and $C\subset \cE$ a generating category.
Let $W\subto \Arr \cE$ be a lex modulator and $(\cL,\cR)$ the corresponding factorization system.
We shall say that $W$ is {\em mono-saturated} if 
\begin{enumerate}[label=(\roman*), leftmargin=*]
\setcounter{enumi}{6}
\item \label{enum:modulator:7} any monomorphism in $\cL$ with codomain in $C$ is in $W$.
\end{enumerate}
For example, any Grothendieck topology on $\P C$ is mono-saturated.
Because an $n$\=/topos is well-powered, a lex modulator can always be completed in a mono-saturated one.
\end{defi}

The following lemma is the main fact behind \cref{prop:truncated-sheaf}.

\begin{lem}
\label{lem:truncated-sheaf}
Let $\cE$ be an $n$\=/topos ($n\leq \infty$) and $W\subto \Arr \cE$ a mono-saturated lex modulator.
Let $(\cL,\cR)$ be the factorization system generated by $W$.
Then, for any monomorphism $m$ in $\cL$, $m^+$ is an isomorphism.
More generally, if $w$ is in $\cL$ and $n$\=/truncated, $w^+$ is $(n-1)$\=/truncated.
\end{lem}
\begin{proof}
Recall that a map $f:A\to B$ in $\cC$ is a monomorphism if and only if $f \simeq f\times_{1_B} f$ in $\Arr \cC$.
Because the plus-construction preserves codomains and is left-exact, it preserves monomorphisms.
Hence, in order to prove that $m^+$ is invertible it is enough to prove that it has a section.

Let $C$ be the generating category for $\cC$ and $X$ in $C$.
Since $W$ is mono-saturated, any base change $m_\xi$ of $m$ along $\xi:X\to \t m$ is in $W$.
Because colimits are universal in $\cC$, $m$ is the colimit of the $m_\xi$.
By construction of $m^+$, we have lifts
\[
\begin{tikzcd}
\s m_\xi \ar[r] \ar[d] & \s m \ar[r]\ar[d,"m" near end] &\s m^+ \ar[d,"m^+"]\\
\t m_\xi \ar[r] \ar[rru,dashed, bend left=05] & \t m \ar[r,equal] &\t m
\end{tikzcd}
\]
These lifts are unique since $m^+$ is a mono.
Passing to the colimit, they define a lift
\[
\begin{tikzcd}
\s m \ar[r]\ar[d,"m"'] &\s m^+ \ar[d,"m^+"]\\
\t m\ar[r,equal] \ar[ru,dashed]&\t m.
\end{tikzcd}
\]
This provide a section of $m^+$.

For the second statement, we use that $f$ is $n$\=/truncated if and only if $\Delta^{n+2} f$ invertible if and only if $\Delta^{n+1} f$ is a monomorphism.
The left-exactness of $+$ and the previous computation implies that $\Delta^{n+1} (\Plusmap f)$ is invertible, hence $\Plusmap f$ is $(n-1)$\=/truncated.
\end{proof}

%\medskip
We need another lemma.
Any two maps $f:A\to B$ and $g:B\to C$ in a category $\cC$ define a commutative cube
\[
\begin{tikzcd}
A\ar[rr,equal]\ar[rd,"f"] \ar[dd,"f"'] && A \ar[rd,"f"] \ar[dd,"gf"' near start]\\
&B && B \ar[from=ll, crossing over, equal]\ar[dd, "g"]\\
B \ar[rd, equal] \ar[rr, "g" near start] && C \ar[rd, equal]\\
&B \ar[from = uu, crossing over, equal] \ar[rr, "g"] && C
\end{tikzcd}
\]    
where both top and bottom faces are cartesian and cocartesian. 
This proves the following result.

\begin{lem}
\label{lem:facto-bicartesian}
Viewed from above, the previous cube define a square in $\Arr\cC$ which is both cartesian and cocartesian.
\[
\begin{tikzcd}
f \ar[r]\ar[d] \pbmark &gf \ar[d]\\
1_B\ar[r]&g \pomark
\end{tikzcd}
\]    
\end{lem}

\begin{lem}
\label{lem:+facto}
Let $\cE$ be an $n$\=/topos ($n\leq \infty$) and $W\subto \Arr \cE$ a mono-saturated lex modulator.
Let $f:A\to B$ be a map in $\cE$ and $\rho(f)\lambda(f):A\to M\to C$ its factorization for the left-exact modality generated by $W$.
Then, we have $\rho(\Plusmap f) = \rho(f)$ and $\lambda(\Plusmap f) = \lambda(f)^+$.
\end{lem}
\begin{proof}
Using \cref{lem:facto-bicartesian} for $f=\rho(f)\lambda(f)$ and the left-exactness of the plus-construction, we get a cartesian square
\[
\begin{tikzcd}
\lambda(f)^+ \ar[r]\ar[d] \pbmark & \Plusmap f\ar[d]\\
(1_M)^+\ar[r]&\rho(f)^+.
\end{tikzcd}
\]
We have $\rho(f)^+ = \rho(f)$ since $+$ fixes the class $\cR$.
This implies that $M$ is also the middle object of the factorization of $\Plusmap f$. 
Using $(1_M)^+ = 1_M$ and \cref{lem:facto-bicartesian} for $\Plusmap f=\rho(\Plusmap f)\lambda(\Plusmap f)$, we get $\lambda(\Plusmap f)=\lambda(f)^+$.
\end{proof}

\begin{prop}[Sheafification of truncated objects]
\label{prop:truncated-sheaf}
Let $\cE$ be an $n$\=/topos ($n\leq \infty$) and $W\subto \Arr \cE$ a mono-saturated lex modulator.
Then, if $f$ is an $m$\=/truncated map ($m\leq n$), we have $f^{+m+2} = \rho(f)$.
\end{prop}
\begin{proof}
By \cref{lem:+facto}, $f^{+m+2} = \rho(f^{+m+2})\lambda(f^{+m+2}) = \rho(f)\lambda(f)^{+m+2}$.
Then, by \cref{lem:truncated-sheaf}, $\lambda(f)^{+m+2}$ is an isomorphism and $f^{+m+2} = \rho(f)$.
\end{proof}

\begin{rem}
\label{rem:truncated-sheaf}
Notice that we do not need $W\to \Arr \cC$ to be made of monomorphisms (although it needs to be mono-saturated).
\Cref{prop:truncated-sheaf} works for all accessible left-exact localizations, topological or not.
However, the fact that the modulator is lex is crucial.
The result need not be true with a $\Delta$\=/modulator for which the plus-construction is not left-exact.
\end{rem}

\begin{cor}
Let $\cE$ be an $n$\=/topos (for $n<\infty$) and $W\subset \cE^\mono$ a topology.
Then, independently of the size of the maps in $W$, 
the sheafification can be computed by applying the plus-construction $n+1$~times.
\end{cor}
\begin{proof}
Because of \cref{cor:lex-mono-gen} and \cref{rem:mono-not-lex+}, we can apply \cref{thm:lex+construction}.
The result follows from \cref{prop:truncated-sheaf} and the fact that all objects of an $n$\=/topos are $(n-1)$\=/truncated.
\end{proof}

\begin{table}[htbp]
\begin{center}	
\caption{Summary of the conditions for the plus-construction}
\label{table:cdt+}
\medskip
\renewcommand{\arraystretch}{2}
\begin{tabularx}{\textwidth}{
|>{\hsize=.7\hsize\linewidth=\hsize\centering\arraybackslash}X
|>{\hsize=.4\hsize\linewidth=\hsize\centering\arraybackslash}X
>{\hsize=1.3\hsize\linewidth=\hsize\raggedright\arraybackslash}X
|>{\hsize=1.6\hsize\linewidth=\hsize\centering\arraybackslash}X
|}
\hline
$W\to \cC\comma C$
    & \multicolumn{2}{c|}{{\em Condition}}
    & {\em Property}
\\
% \hline
\hline
Pre-modulator
    & \ref{enum:modulator:1}-\ref{enum:modulator:4} 
    & $\t:W\rightleftarrows C:id$ is a reflective~localization
    & The $+$ and $k$\=/constructions coincide (\cref{thm:+construction}).
\\
\hline
Modulator
    & \ref{enum:modulator:1}-\ref{enum:modulator:5} 
    &
    {\begin{tikzcd}[ampersand replacement=\&,sep=small,cramped]
        W\ar[rd,"\t"']\ar[r,hook] \& \cC\comma C \ar[d,"\ \t"]\\
        \& C    
      \end{tikzcd}}
      is a sub-fibration
    & The plus-construction generates modalities (\cref{thm:+modality}) 
    and is given by the usual formula in $\cC=\P C$ (\cref{lem:+=+}).
\\
\hline
$\Delta$\=/modulator
    & \ref{enum:modulator:1}-\ref{enum:modulator:5} 
    & and \ref{enum:modulator:6-} $W^\Delta\subset \cL_W$
    & The plus-construction may not be lex but converges to a lex localization (\cref{thm:lex-diagonal}).
\\
\hline
Lex-modulator
    & \ref{enum:modulator:1}-\ref{enum:modulator:6} 
    & $W\subto \cC\comma C$ is a sub-fibration and~$W(X)$ are co-filtered
    & The plus-construction is left-exact (\cref{thm:lex+construction}).
\\
\hline
Lex-modulator mono-saturated
    & \ref{enum:modulator:1}-\ref{enum:modulator:7} 
    & contains all the inverted monomorphisms with codomain in $C$
    & The plus-construction converges in finite time on truncated objects (\cref{prop:truncated-sheaf}).
\\
\hline 
\end{tabularx}
\end{center}
\end{table}

\subsubsection{Example: Grothendieck topologies}
\label[example]{ex:GTopo}

Let $C$ be a small category, a {\em Grothendieck topology} on $\P C$~\cite[6.2.2]{Lurie:HTT} can be defined as a modulator $W\subto \P C\comma C$ such that
\begin{enumerate}[label=(GT\arabic*), leftmargin=*]
\item \label{enum:Gtopology:1} (monomorphic modulator) all maps are monomorphisms,
\item \label{enum:Gtopology:2} (locality) a map $R'\to X$ is in $W(X)$ if and only if there exists some $w:R\to X$ in $W$ such that for all $f:Y\to R$, the map $R'\times_XY\to Y$ is in $W(Y)$.
\[
\begin{tikzcd}
R'\times_XY \ar[rr] \ar[d,"W \ni"'] \pbmark
    && R'\ar[d]
\\
Y\ar[r]
    & R\ar[r,"\in W"']
    & X
\end{tikzcd}
\]
\end{enumerate}

Let us call a {\em monomorphic modulator} a modulator satisfying \ref{enum:Gtopology:1}.
They correspond to generators for {\em topological modalities} in sense of~\cite[Def. 3.13]{RSS}.
We saw in \cref{cor:lex-mono-gen}, that the factorization system generated by such a monomorphic modulator is always left-exact.
But as is mentioned in \cref{rem:mono-not-lex+}, the plus-construction need not be a left-exact functor.
This is somehow the motivation for the locality axiom \ref{enum:Gtopology:2}.
It implies that the posets $W(X)$ have meets, that is, that they are co-filtered.
The modulator is then lex in the sense on \cref{defi:lex-modulator}.
Moreover, if the generator $C$ is stable under finite limits, then $W$ is in fact stable under finite limits too.
In any case, because of \ref{enum:Gtopology:2} we can apply \cref{thm:lex+construction}.
The sheafification is then given by a transfinite iteration of the plus-construction (as was proven in \cite[Prop. 6.2.2.7]{Lurie:HTT}).
And, for $n$\=/truncated objects, \cref{prop:truncated-sheaf} says that $n+2$ iterations are enough.

\subsubsection{Example: open localization}
\label[example]{ex:open-loc}

Let $\cE$ be a topos and $U\mono 1$ a subterminal object.
The left-exact localization generated by $U\mono 1$ is the open subtopos corresponding to $U$.
We fix a generating subcategory $C\subset \cE$.
Let $w$ be the map $U\mono 1$, then $w^\Delta$ is reduced to $w$ and the identity of $U$.
Then $w^{\Delta mod}$ is the full subcategory of $\cE\comma C$ spanned by the identity maps of the objects of $C$ and by the maps $A\times U\to A$.
In fact, we have $w^{\Delta mod} \simeq I\times C$, where $I=\{0\to 1\}$ is the arrow category.
Using the fact that a coend indexed by $I$ is a pushout, the plus-construction is
\begin{align*}
\Plusoriginal F 
&= \int^{v\in w^{\Delta mod}} \map{\s v} F \times \t v     \\
&= \int^{A\in C} \map{A\times U} F \times A \coprod_{\int^{A:C} \map{A} F \times A} \int^{A:C} \map{A} F \times A     \\
&= \left(\int^{A\in C} \map{A} {F^U} \times A\right) \coprod_FF\\
&= F^U.
\end{align*}
Because $U$ is subterminal, the plus-construction is idempotent and directly gives the reflector of the lex localization.

\subsubsection{Example: closed localization}
\label[example]{ex:closed-loc}

With the same notations as in \cref{ex:open-loc}, the left-exact localization generated by $0\to U$ is the open subtopos corresponding to $U$.
Let $w$ be the map $0\to U$, then $w^\Delta$ is reduced to $w$ and the identity of $0$.
Then $w^{\Delta mod}$ is the full subcategory of $\cE\comma C$ spanned by the identity maps of the objects of $C$ and the maps $0_A:0 \to A$ for those $A$ such that the canonical map $A\to 1$ factors through $U$, that is, those $A$ such that $P_{-1}(A) \subset U$.
The category $w^{\Delta mod}$ has a projection to the arrow category $I=\{0\to 1\}$ such that the fiber over 0 is $C\comma U$ and the fiber over 1 is $C$.
In fact it can be shown that $w^{\Delta mod}$ is the total category (cograph or cylinder) of the forgetful functor $C\comma U\to C$.
Hence it is fibered and op-fibered over $I$.
A careful computation with coends shows that
\begin{align*}
\Plusoriginal F 
&= \int^{v\in w^{\Delta mod}} \map{\s v} F \times \t v    \\
&= \int^{A\in C\comma U} \map{\s (0_A)} F \times \t (0_A) \coprod_{\int^{A\in C\comma U} \map{\s (1_A)} F \times \t (0_A)} \int^{A\in C} \map{\s (1_A)} F \times \t (1_A)    \\
&= U \coprod_{U\times F} F\\
&= U\join F.
\end{align*}
Because $U$ is assumed to be subterminal, the plus-construction is idempotent and directly provide the reflector of the lex localization.

The computation actually never uses the fact that $U$ is subterminal.
For a general $U$, the plus-construction is no longer idempotent, but converges in a countable number of steps to $U^{\join \infty}\join F = P_{-1}(U) \join F$.
In other terms, the left-exact localization generated by the map $0\to U$, for any object $U$, is the closed subtopos complement of the open $P_{-1}(U)$, that is, of the support of $U$.
This provides also an example of a $\Delta$\=/modulator for which the plus-construction is not left-exact (see \cref{rem:mono-not-lex+}).

\subsubsection{Example: classifier of truncated objects}
\label[example]{ex:classifier-truncated}

Let $\Fin$ be the category of finite spaces. 
Then, $\Fin\op$ is the free lex completion of the punctual category 1.
We consider the topos $\S X := \P{\Fin\op} = \fun {\Fin} \cS$.
For $I$ in $\Fin$, we shall denote by $X^I$ the functor corresponding by the Yoneda embedding $\Fin\op \to \S X$.
The object $X$ in $\S X$ is $X^1$. 
The corresponding functor is the canonical embedding $\Fin \subto \cS$.

If $\cE$ is another topos, a cocontinuous and left-exact functor $f^*:\S X\to \cE$ reduces to a lex functor $\Fin\op\to \cE$, and finally to a functor $1\to \cE$, that is, an object of $\cE$.
This object is the image of $X$ by the functor $f^*:\S X\to \cE$.
In other words, the topos $\S X$ classifies objects.
The left-exact localizations of $\S X$ classify objects having certain properties.

We shall focus on the localization classifying $n$\=/truncated objects.
Recall that an object is $n$\=/truncated if and only if the higher diagonal $\Delta^{n+2}X = X\to X^{S^{n+1}}$ is invertible.
Therefore, the topos $\S {X\truncated n}$ classifying $n$\=/truncated objects is the left-exact localization of $\S X$ generated by $w=\Delta^{n+2}X$:
\[
\S {X\truncated n} = \relDKLoc {\text{cc}}{\text{lex}} {\S X} {\Delta^{n+2}X}.
\]
We compute $w^{\Delta mod}$.
The diagonals of $w$ are the higher diagonals $\Delta^{m+1}X:X\to X^{S^m}$, for $m>n$.
A base change of $\Delta^{m+1}X$ along a map $X^I\to X^{S^m}$ is a map $X^{I/S^m}\to X^I$
associated to a pushout
\[
\begin{tikzcd}
S^m \ar[r]\ar[d]& I \ar[d]
\\
1 \ar[r]& I/S^m \pomark
\end{tikzcd}
\]
The cocontinuous and lex localization generated by $w$ is the cocontinuous localization generated by all the $X^{I/S^m}\to X^I$.
Since all these maps are in $\Fin\op \subset \P {\Fin \op}$, we have by \cref{cor:arbitrary-loc}
\begin{align*}
\relDKLoc {\text{cc}}{} {\P {\Fin \op}} {\{X^{I/S^m}\to X^I\}} 
&= \P {\DKLoc {\Fin\op} {\{X^{I/S^m}\to X^I\}}}\\
&= \P {\DKLoc {\Fin} {\{I\to I/S^m\}} \op}.
\end{align*}
Recall that a map in $\Fin$ is $P_n$\=/equivalence if it induces an equivalence between the $n$\=/th Postnikov truncation.
We leave to the reader the proof that the localization $\DKLoc {\Fin} {\{I\to I/S^m\}}$ inverts all $P_n$\=/equivalences and that $\DKLoc {\Fin} {I\to I/S^m} = P_n(\Fin)$ where $P_n$ is the full subcategory of $\cS\truncated n$ spanned by the $n$\=/truncations of $\Fin$ (see \cref{ex:SOA:Fin-truncation}).
Finally, we get that $\S {X\truncated n}$ is the presheaf topos
\[
\S {X\truncated n} = \fun {P_n(\Fin)} \cS.
\]

\subsubsection{Example: lex localizations in one step}
\label[example]{ex:Grothendieck-one-step}

We fix a topos $\cE$ with some generators $C\subset \cE$ and an accessible left-exact localization $f^*:\cE\to \cF$.
As considered in \cref{prop:lex-mod-4-loc-lex}, it is possible to provide a lex modulator $W'$ associated to $C$.
But more is true, since $W'$ is in fact $\kappa$\=/closed and stable by towers in the sense of \cref{defi:tower}.
We can apply \cref{prop:idempotent+} to deduce that the object
\[
\Plusoriginal F = \colim_{W'\commaindex F} \t w
\]
is directly the sheafification of $F$.

\subsubsection{Example: hypercomplete lex localizations}
\label[example]{ex:hypercovers}

We fix a left-exact localization of $f^*:\cE \to \cF$.
We shall say that a map $h:R\to X$ is a {\em covering} for $f^*$ if $f^*h$ is a cover in $\cF$
and that a map $h:R\to X$ is a {\em hypercovering} for $f^*$ if $f^*h$ is \oo connected in $\cF$, that is, if all the diagonal maps $\Delta^n h:=\pbh {s^{n+1}} h$ are coverings.
Recall from~\cite{Lurie:HTT} that hypercoverings generate the left-exact localization of $\cE$ which is the hypercompletion of $\cF$.
It is proven in~\cite[Prop. 6.5.2.8]{Lurie:HTT} that hypercompletion is an accessible localization. 
We can then chose generators $C\subset \cE$ and a cardinal $\kappa$ as in \cref{ex:Grothendieck-one-step}. 
We define $H(\kappa)$ to be the modulator of $\kappa$\=/small hypercoverings.
It is a lex modulator.
$H(\kappa)$ is $\kappa$\=/closed and stable by towers. 
By \cref{prop:idempotent+} the object
\[
\Plusoriginal F = \colim_{H(\kappa)\commaindex F} \t w
\]
is directely the hypercompletion of the sheafification of $F$.
Applied to $\cE=\cF$, $\Plusoriginal F$ is simply the hypercompletion of $F$.

%%%%%%%%%%%%%%%%%%%%%%%%%%%%%%%%%%%%%%%%%%%%%%%%%%%%%%%%%%%%%%%%%%%%%%%%%%%%%%%%%%%%%%%%%%
%%%%%%%%%%%%%%%%%%%%%%%%%%%%%%%%%%%%%%%%%%%%%%%%%%%%%%%%%%%%%%%%%%%%%%%%%%%%%%%%%%%%%%%%%%

\end{document}